\documentclass[a4paper,11pt,reqno]{amsart}

\usepackage{amsfonts}
\usepackage{amssymb}
\usepackage{amscd}
\usepackage{amsthm}
\usepackage{a4wide}
\usepackage{mathrsfs}
\usepackage[applemac]{inputenc}
\usepackage{amsmath}
\usepackage{amssymb}
\usepackage{amsthm}
\usepackage{textcomp}
\usepackage{graphicx}
\usepackage{enumerate}
\usepackage{mathrsfs}
\usepackage{frcursive}
\usepackage{tikz}
\usepackage[cyr]{aeguill}
\usepackage{xspace}
\usepackage{hyperref}
\usepackage{appendix}
\usepackage{esint}
\usepackage{cancel}
\usepackage{calc}

\newtheorem{defn}{Definition}[section]
\newtheorem{lemma}[defn]{Lemma}
\newtheorem{prop}[defn]{Proposition}
\newtheorem{theo}[defn]{Theorem}
\newtheorem{coro}[defn]{Corollary}

\newtheorem{claim}{Claim}
\newtheorem{rk}[defn]{Remark}

\def\Rm{\mathop{\rm Rm}\nolimits}
\def\tr{\mathop{\rm tr}\nolimits}
\def\det{\mathop{\rm det}\nolimits}

\def\vol{\mathop{\rm vol}\nolimits}
\def\eucl{\mathop{\rm eucl}\nolimits}
\def\dim{\mathop{\rm dim}\nolimits}
\def\vol{\mathop{\rm Vol}\nolimits}

\def\div{\mathop{\rm div}\nolimits}

\def\codim{\mathop{\rm codim}\nolimits}

\def\Rm{\mathop{\rm Rm}\nolimits}
\def\tr{\mathop{\rm tr}\nolimits}
\def\det{\mathop{\rm det}\nolimits}

\def\vol{\mathop{\rm vol}\nolimits}
\def\eucl{\mathop{\rm eucl}\nolimits}
\def\dim{\mathop{\rm dim}\nolimits}
\def\vol{\mathop{\rm Vol}\nolimits}

\def\div{\mathop{\rm div}\nolimits}

\def\codim{\mathop{\rm codim}\nolimits}

\def\supp{\mathop{\rm supp}\nolimits}

\def\R{\mathop{\rm \mathbb{R}}\nolimits}

\newcommand{\Sp}{\mathbb{S}}
\newcommand{\Ent}{\mathcal{E}}
\newcommand{\Ima}{\text{Im}}

\newcommand{\Ob}{\mathcal{V}}

\newsavebox\CBox
\newcommand\hcancel[2][0.5pt]{%
  \ifmmode\sbox\CBox{$#2$}\else\sbox\CBox{#2}\fi%
  \makebox[0pt][l]{\usebox\CBox}%
  \rule[0.5\ht\CBox-#1/2]{\wd\CBox}{#1}}

\begin{document}
\title{A relative entropy for expanders of the harmonic map flow   }
\date{\today}
\begin{abstract}In this paper we focus on the uniqueness question for (expanding) solutions of the Harmonic map flow coming out of smooth $0$-homogeneous maps with values into a closed Riemannian manifold. We introduce a relative entropy for two purposes. On the one hand, we prove the existence of two expanding solutions associated to any suitable solution coming out of a $0$-homogeneous map by a blow-up and a blow-down process. On the other hand, generic uniqueness of expanding solutions coming out of the same $0$-homogeneous map of $0$ relative entropy is proved.
\end{abstract}

\author{Alix Deruelle}
\address[Alix Deruelle]{Institut de Math\'ematiques de Jussieu, Paris Rive Gauche (IMJ-PRG) UPMC - Campus Jussieu, 4, place Jussieu Bo\^ite Courrier 247 - 75252 Paris Cedex 05}
\email{alix.deruelle@imj-prg.fr}

\maketitle

\section{Introduction}
Given a connected closed Riemannian manifold $(N,g)$ isometrically embedded in some Euclidean space $\R^m$, $m\geq 2$, we consider solutions of the harmonic map flow of maps $(u(t))_{t\geq 0}$ from $\R^n$, $n\geq 3$ to $(N,g)$. More precisely, we study the parabolic system
\begin{equation}
\left\{\begin{aligned}
&\partial_tu=\Delta u+A(u)(\nabla u,\nabla u),\quad\mbox{on $\mathbb{R}^n\times\mathbb{R}_+$},&\label{eq-HMP} \\
&u|_{t=0}=u_0,&
\end{aligned}
\right.
\end{equation}
for a given map $u_0:\mathbb{R}^n\rightarrow N$, where $A(u)(\cdot,\cdot):T_uN\times T_uN\rightarrow (T_uN)^{\perp}$ denotes the second fundamental form of the embedding $N \hookrightarrow \R^m$ evaluated at $u$. Observe that equation (\ref{eq-HMP}) is equivalent to $\partial_tu-\Delta u\perp T_uN$ for a family of maps $(u(t))_{t\geq0}$ which map into $N$.
In this paper, we focus on expanding solutions of the Harmonic map flow coming out of $0$-homogeneous maps, i.e. solutions that are invariant under parabolic rescalings,
\begin{eqnarray}
u_{\lambda}(x,t)&:=&u(\lambda x,\lambda^2t)=u(x,t), \quad \lambda>0,\quad \text{$(x,t)\in \R^n\times\R_+$}.\label{resc-cond}
\end{eqnarray}
 
 The condition (\ref{resc-cond}) reflects the homogeneity of the initial condition $u_0$ in a parabolic sense.
 
  In this setting, it turns out that (\ref{eq-HMP}) is equivalent to an elliptic equation. Indeed, if $u$ is an expanding solution in the sense of (\ref{resc-cond}) then the map $U(x):=u(x,1)$ for $x\in \mathbb{R}^n$, satisfies the elliptic system
\begin{equation}
\left\{\begin{aligned}
&\Delta_f U+A(U)(\nabla U,\nabla U)=0,\quad\mbox{on $\mathbb{R}^n$},&\label{eq-HMP-Stat}\\
&\lim_{|x|\rightarrow+\infty}U(x)=u_0(x/|x|),&
\end{aligned}
\right.
\end{equation}
 
where, $f$ and $\Delta_f$ are defined by
\begin{eqnarray*}
&&f(x):=\frac{|x|^2}{4}+\frac{n}{2},\quad x\in\mathbb{R}^n,\\
&&\Delta_fU:=\Delta U+\nabla f\cdot\nabla U=\Delta U+\frac{r}{2}\partial_rU.
\end{eqnarray*}
The function $f$ is called the potential function and it is defined up to an additive constant. The choice of this constant is dictated by the requirement $$\Delta_ff=f.$$ The operator $\Delta_f$ is called a weighted laplacian and it is unitarily conjugate to a harmonic oscillator $\Delta-|x|^2/16$.

Conversely, if $U$ is a solution to (\ref{eq-HMP-Stat}) then the map $u(x,t):=U(x/\sqrt{t})$, for $(x,t)\in\R^n\times\R_+$, is a solution to (\ref{eq-HMP}). Because of this equivalence, $u_0$ can be interpreted either as an initial condition or as a boundary data at infinity. 

Existence of expanding solutions has been investigated by Germain and Rupflin \cite{Ger-Rup}, Biernat and Bizon \cite{Bie-Biz} and Germain, Ghoul and Miura \cite{Ger-Gho-Miu} in an equivariant setting. The author and T. Lamm studied the existence of weak expanding solutions of the Harmonic map flow for maps with values into an arbitrary smooth closed Riemannian manifold: \cite{Der-Lam-HMF}. There, we proved the existence of such solutions in the case the initial condition is a Lipschitz map that is homotopic to a constant. Moreover, our solutions are proved to be regular outside a compact set whose diameter depends on the $L^2_{loc}$ energy of the initial condition only.\\

We decide in this paper to focus on the uniqueness question for expanding solutions coming out of $0$-homogeneous maps $u_0$ that are sufficiently smooth outside the origin. Besides, we restrict ourselves to expanding solutions that are sufficiently regular at infinity: Definition \ref{def-reg-inf} ensures that the solution reaches its initial condition in a sufficiently smooth sense.

Now, from a variational viewpoint, expanding solutions can be interpreted as critical points of the following formal entropy or weighted energy:
\begin{eqnarray}
\mathcal{E}^+(u):=\int_{\mathbb{R}^n}|\nabla u|^2e^{\frac{|x|^2}{4}}dx.\label{inf-ent-exp}
\end{eqnarray}
 The issue is that this quantity associated to an expanding solution of the Harmonic map flow is infinite unless the solution is constant. Indeed, the pointwise energy of an expanding solution $u$ decays at most quadratically, i.e. $$|\nabla u|^2(x)\sim|x|^{-2},\quad\text{as $x\rightarrow+\infty$}.$$

This is in sharp contrast with ancient solutions called shrinking solutions: these solutions live on $(-\infty,0)$ and can potentially arise from finite-time singularities of the Harmonic map flow for maps between two closed Riemannian manifolds. Shrinking solutions are true critical points of the following well-defined entropy: 
\begin{eqnarray*}
\mathcal{E}^-(u):=\int_{\mathbb{R}^n}|\nabla u|^2e^{-\frac{|x|^2}{4}}dx<+\infty.
\end{eqnarray*}

As explained to the author by T. Ilmanen, there should be a notion of relative entropy for expanding solutions that we now explain despite the issue caused by (\ref{inf-ent-exp}).

Let $u$ be an expanding solution coming out of a $0$-homogeneous map and assume there exists a background expanding solution $u_b$ coming out of the same $0$-homogeneous map. Then, the relative entropy of $u$ and $u_b$ is formally defined by:
\begin{eqnarray}
\mathcal{E}(u,u_b):=\lim_{R\rightarrow+\infty}\int_{B(0,R)}\left(|\nabla u|^2-|\nabla u_b|^2\right)e^{\frac{|x|^2}{4}}dx.\label{rel-ent-intro-def}
\end{eqnarray}
In order to prove that this quantity (\ref{rel-ent-intro-def}) is well-defined (Theorem \ref{mono-rel-ent}), we need to establish a convergence rate for the difference $u-u_b$ of the form $f^{-n/2}e^{-f}$: see Theorem \ref{prop-dec-time-diff-sol}. This rate is sharp as shown by a unique continuation result at infinity proved in Theorem \ref{uni-cont-inf}. Observe that the naive estimate of the difference $u-u_b=u-u_0+u_0-u_b$ given by comparing the solutions to their common initial condition $u_0$ only yields a quadratic decay, i.e. $u-u_0=\textit{O}(f^{-1})=u_b-u_0$. We actually prove that this entropy (\ref{rel-ent-intro-def}) can be defined more generally for suitable solutions of the Harmonic map flow coming out of a $0$-homogeneous map $u_0$: if $u(t)$ is such a suitable solution given by Definition \ref{def-reg-inf} and if $u_b(t)$ denotes the time-dependent solution associated to $u_b$ as explained at the beginning of this intoduction then the function
\begin{eqnarray}
\mathcal{E}(u(t),u_b(t)):t\in\R_+\rightarrow\lim_{R\rightarrow+\infty}\int_{B(0,R)}t\left(|\nabla u(t)|^2-|\nabla u_b(t)|^2\right)\frac{e^{\frac{|x|^2}{4t}}}{(4\pi t)^{n/2}}dx,\label{rel-ent-def-intro-gal-sol}
\end{eqnarray}
is well-defined. Moreover, it is monotone decreasing along the Harmonic map flow and is constant precisely on expanding solutions coming out of the same initial condition $u_0$.

Again, we emphasize that (\ref{rel-ent-def-intro-gal-sol}) is produced by taking differences rather than by considering a renormalization: this makes the analysis much harder since one has to match the asymptotics of such expanding solutions (or suitable solutions to the Harmonic map flow) in a much more precise way. Observe that a renormalization on an increasing sequence of exhausting balls in Euclidean space $\R^n$ would have made the first variation of $\mathcal{E}$ vanish since Theorem \ref{mono-rel-ent} shows that the weighted $L^2$ norm of the obstruction tensor $$\partial_tu+\frac{x}{2t}\cdot\nabla u,$$ for $u$ to be an expanding solution is finite in the Lebesgue sense:
\begin{eqnarray}
\frac{d}{dt}\mathcal{E}(u,u_b)(t)=-2t\int_{\mathbb{R}^n}\left|\partial_tu+\frac{x}{2t}\cdot\nabla u\right|^2\frac{e^{\frac{|x|^2}{4t}}}{(4\pi t)^{n/2}}dx>-\infty,\quad t>0.
\end{eqnarray}

Besides its monotonicity along the Harmonic map flow, the interest of this relative entropy is twofold. On one hand, we are able to prove a version of Ilmanen's conjecture originally stated for the Mean curvature flow in [Section F., \cite{Ilm-Lec-Not}]. Ilmanen's conjecture can be stated roughly as follows:
\begin{theo}[Ilmanen's conjecture: the smooth case]\label{Ilm-conj-rough}
 Blowing up and blowing down any (suitable) solution of the Harmonic map flow coming out of a $0$-homogeneous map give rise to two expanding solutions coming out of the same $0$-homogeneous map provided there already exists another expanding solution coming out of the same $0$-homogeneous map.
 \end{theo}
 We refer the reader to Theorem \ref{ilmanen-smooth-conj} for a precise statement. 
 
We notice that a similar relative entropy has been previously considered by Ilmanen, Neves and Schulze for the Network flow for regular networks: \cite{Ilm-Nev-Sch}.

 Theorem \ref{Ilm-conj-rough} reduces so to speak the uniqueness question for suitable solutions of the Harmonic map flow coming out of a given $0$-homogeneous map $u_0$ to the uniqueness question for expanding solutions of the Harmonic map flow coming out the same map $u_0$.
 This principle is used to show a uniqueness result in case the target is non-positively curved: 
 
 \begin{theo}[Non-positively curved targets]\label{theo-non-pos-cur-tar-intro}
Let $(N,g)$ be a Riemannian manifold with non-positive sectional curvature.
Let $n\geq 3$ and let $u_0:\R^n\setminus\{0\}\rightarrow (N,g)$ be a smooth $0$-homogeneous map. Then there exists a unique suitable smooth solution coming out of $u_0$: this solution must be expanding.
\end{theo}
The existence part could be adapted from \cite{Der-Lam-HMF} but we prove it here directly: we use a continuity method by connecting any $0$-homogeneous map with values into a non-positively curved target to a constant map. The existence of such a path of initial conditions is granted by Hadamard's Theorem which ensures in particular that a complete Riemannian manifold with non-positive sectional curvature is aspherical. Actually, Theorem \ref{theo-non-pos-cur-tar-intro} can be interpreted as the non-compact version of Hamilton's Theorem \cite{Ham-HMF-Bdy} where the Dirichlet boundary data is now pushed at infinity. Finally, Theorem \ref{theo-non-pos-cur-tar-intro} is the analogous statement of a theorem due to the author for expanding solutions of the Ricci flow: see Remark \ref{rk-theo-rf-vs-hmf} and the references therein.

On the other hand, we prove a generic uniqueness result for expanding solutions. Again, it can be stated roughly as follows:
\begin{theo}[Generic uniqueness: unformal statement]\label{rough-sta-gen-uni-intro}
The set of $0$-homogeneous maps that are smoothed out by more than one (suitable) expanding solution with $0$ relative entropy is of first category in the Baire sense. 
\end{theo}
Again, we refer the reader to Theorem \ref{theo-gen-uni} for a precise statement. Both the statement and the proof of Theorem \ref{rough-sta-gen-uni-intro} are motivated by the work of L. Mou and R. Hardt on harmonic maps defined on a domain of Euclidean space (with a boundary at finite distance): [Theorem $6.8$ and Corollary $6.9$, \cite{Har-Mou}]. In order to prove such a generic statement, one needs to understand the (Banach) manifold structure of the moduli space of suitable expanding solutions. It requires in particular to understand the Fredholm properties of the linearized operator, also called the Jacobi operator, of equation (\ref{eq-HMP-Stat}). This approach originates from the work of B. White on minimal surfaces \cite{White-var-met}. Notice that the paper \cite{Har-Mou} also adapts White's work in the setting of harmonic maps. More recently, J. Bernstein and L. Wang \cite{Ber-Wan-MCF} have been able to adapt White's approach to expanding solutions of the Mean curvature flow. Our arguments are very close in spirit to their work and are adapted from the author's work in collaboration with T. Lamm on expanding solutions of harmonic map flow \cite{Der-Lam-HMF}.  \\

We end this introduction by describing the structure of this paper:\\

Section \ref{section-L^2-max-ppe} defines solutions of the harmonic map flow that are regular at infinity.
The main result of Section \ref{section-L^2-max-ppe} is Lemma \ref{theo-max-ppe-l2}: it establishes exponential decay for subsolutions of  the heat equation with a potential decaying quadratically in space-time, vanishing at time $t=0$ in a suitable sense. Two consequences are derived from this. Theorem \ref{prop-dec-time-diff-sol} gives a sharp decay on the difference of two solutions of the harmonic map flow that are regular at infinity and that come out of the same $0$-homogeneous map. Next, Theorem \ref{comp-exp-sol-gal-sol} considers two solutions $(u_i)_{i=1,2}$ of the harmonic map flow coming out of the same initial $0$-homogeneous map such that $u_1$ is an expanding solution. The obstruction tensor $\partial_tu_2+\frac{x}{2t}\cdot\nabla u_2,$ for $u_2$ to be an expanding solution is shown to decay exponentially as well. Finally, Proposition \ref{autom-reg-infty} establishes estimates on the derivatives of the rescaled difference $U$ of two solutions $(u_i)_{i=1,2}$ of the Harmonic map flow coming out of the same $0$-homogeneous initial condition:
$f^{\frac{n}{2}}e^f(u_2-u_1)=:U.$ To do so, an intermediate step consists in showing that this rescaled difference satisfies a backward heat equation-like: this explains at least heuristically why such estimates hold. Notice that the derivatives of the un-rescaled difference $u_2-u_1$ are decaying exponentially but the degree of the polynomial in front of the exponential term is increasing: Proposition \ref{autom-reg-infty} is crucial to prove the relative entropy mentioned above is well-defined.\\

Section \ref{section-rel-ent-Ilm-conj} proves the relative entropy (\ref{rel-ent-intro-def})  is well-defined for an arbitrary suitable solution of the Harmonic map flow and an expanding solution coming out of the same initial $0$-homogeneous map. Moreover, we show this relative entropy (\ref{rel-ent-intro-def})  is monotone along the Harmonic map flow and is constant precisely on expanding solutions: this is the content of Theorem \ref{mono-rel-ent}. This section ends with the proof of Ilmanen's conjecture for solutions of the Harmonic map flow: Theorem \ref{ilmanen-smooth-conj}.\\

Then Section \ref{section-pohozaev-identity} is a short intermission establishing a rigidity statement about solutions coming out of constant maps seen as $0$-homogeneous maps: Corollary \ref{coro-unique-exp-pt} is based on a crucial (static) Pohozaev identity proved in Corollary \ref{Poho-exp}.\\

The moduli space of smooth expanding solutions is investigated in Section \ref{section-mod-space-smooth-exp}. We follow closely the arguments due to Hardt and Mou \cite{Har-Mou} on harmonic maps defined on a domain of Euclidean space which in turn are based on works of B. White on minimal surfaces, as said before: \cite{Whi-para-ell-fct} and \cite{White-var-met}. However, the analysis is substantially different since the domain is non-compact and the Jacobi operator is a drift Laplacian. More precisely, Section \ref{fct-spa-first-sec-var-rel-ent} introduces the relevant weighted function spaces and establishes the first and second variation of this new relative entropy given by (\ref{rel-ent-intro-def}). Finally, Theorem \ref{Analysis-Jacobi-field} analyses the decay of deformations lying in the kernel of the Jacobi operator that keep the boundary value at infinity fixed. When appropriately rescaled, such deformations have a well-defined trace at infinity: this statement has been proved by the author \cite{Uni-Con-Egs-Der} in the case of expanding solutions coming out of Ricci flat cones and by J. Bernstein and L. Wang for expanding solutions of the Mean curvature flow \cite{Ber-Wan-MCF}. Moreover, a unique continuation at infinity holds for such deformations in the sense that this trace at infinity determines the deformation globally: this is the last part of Theorem \ref{Analysis-Jacobi-field}.

The necessary Fredholm properties of the Jacobi operator are the content of Theorem \ref{theo-fred-prop-jac-op} that follows the approach due to the author and T. Lamm \cite{Der-Lam-HMF}: the analysis is very similar to \cite{Ber-Wan-MCF} for the Mean curvature flow. The local structure of the moduli space of smooth expanders is proved in Theorem \ref{loc-str} where Theorem \ref{Analysis-Jacobi-field} plays a crucial role. Theorem \ref{theo-glo-str} ends the analysis of the moduli space of smooth expanding solutions of the Harmonic map flow: the kernel of the Jacobi operator is integrable and the regular values of the projection map $\Pi$ associating to an expanding solution its boundary value at infinity is a residual set. \\

As its title suggests, Section \ref{uni-cont-gen-uni} starts with a non linear unique continuation at infinity for expanding solutions of the Harmonic map flow (Theorem \ref{uni-cont-inf}) whose analysis is essentially contained in \cite{Uni-Con-Egs-Der} for the Ricci flow. This leads us to prove a generic uniqueness result for expanding solutions with the same entropy as explained in Theorem \ref{rough-sta-gen-uni-intro}: see Theorem \ref{theo-gen-uni} for a rigorous statement.\\

The last section, Section \ref{sec-com-asy-est-neg-cur}, studies the boundary of the moduli space of expanding solutions of the Harmonic map flow: Theorem \ref{theo-high-der-a-priori}. In case the target is non-positively curved, the compactness of such moduli space is proved as expected: Corollary \ref{coro-non-neg-tar-comp}. This compactness is in turn crucial to prove an existence and uniqueness result for expanding solutions coming out of smooth $0$-homogeneous maps with values into a non-positively curved target: Theorem \ref{exi-uni-non-neg-cur}. 

\textbf{Acknowledgements.}
The author wishes to thank Tom Ilmanen for many useful discussions and for sharing his ideas while in residence at the Mathematical Sciences Research Institute (MSRI) in Berkeley, California, during the Spring 2016 semester. The author also thanks Felix Schulze for so many fruitful conversations. The author was supported by the grant ANR-17-CE40-0034 of the French National Research Agency ANR (project CCEM).

\section{An $L^2$ maximum principle and its consequences}\label{section-L^2-max-ppe}

Let $(N,g)$ be a closed smooth Riemannian manifold isometrically embedded in some Euclidean space $\R^m$. Let $u_1$ be an expanding solution coming out of a $0$-homogeneous map $u_0:\R^n\rightarrow N$  and let $u_2$ be any solution to the Harmonic map flow coming out of the same initial condition $u_0$. Then the relative entropy should be defined as follows:
\begin{eqnarray}\label{def-rel-entropy}
\mathcal{E}(u_2,u_1)(t)&:=&\lim_{R\rightarrow+\infty}\int_{B(0,R)}t(|\nabla u_2|^2-|\nabla u_1|^2)(x,t)\frac{e^{\frac{|x|^2}{4t}}}{(4\pi t)^{n/2}}dx,
\end{eqnarray}
for $t>0$,
at least formally speaking.

A first step to prove this limit is meaningful consists in deriving a sharp bound on the decay of the difference $u_2-u_1$ in the following space-time region $$\Omega_{\lambda}:=\left\{(x,t)\in \mathbb{R}^n\times (0,T]\quad|\quad |x|^2/{4t}>\lambda\right\},$$ for some positive $\lambda$ and some positive time $T$. For that purpose, we define the following barrier function that will be extremely useful in the sequel:
\begin{eqnarray*}
f(x,t)&:=&\frac{|x|^2}{4t}+\frac{n}{2},\quad (x,t) \in\R^n\times\mathbb{R}_+.
\end{eqnarray*}
Note that:
\begin{eqnarray*}
(\partial_t-\Delta)f=-\frac{f}{t}.
\end{eqnarray*}
We will denote by the same symbol $f$ the function $f(\cdot,1)$ when there is no ambiguity.

The solutions of the harmonic map flow we consider here have regularity at infinity: 
\begin{defn}\label{def-reg-inf}
Let $u:\R^n\times\R_+\rightarrow N$ be a smooth solution of the harmonic map flow coming out of the $0$-homogeneous map $u_0:\R^n\setminus\{0\}\rightarrow N$. The solution $u$ is said to be \textbf{regular at infinity} if there exists $\lambda\in\R_+$ such that for each positive integer $k$, there exists a positive constant $C_k$ such that:
\begin{eqnarray*}
|\nabla^k u|(x,t)\leq\frac{C_k}{(|x|^2+t)^{\frac{k}{2}}},\quad\forall (x,t)\in\Omega_{\lambda}.
\end{eqnarray*}

\end{defn}

\begin{rk}
From Definition \ref{def-reg-inf}, if a regular solution of the harmonic map flow coming out of the $0$-homogeneous map $u_0:\R^n\setminus\{0\}\rightarrow N$ is regular at infinity then $u_0$ is necessarily smooth.
\end{rk}

The main theorem of this section is the following maximum principle at infinity:
 \begin{lemma}\label{theo-max-ppe-l2}
Let $w:\R^n\times(0,T)\rightarrow[0,+\infty)$ be a smooth bounded subsolution on $\R^n\rightarrow[0,+\infty)$ of the following differential inequality:
\begin{eqnarray}
\partial_tw\leq\Delta w+\frac{c_0}{|x|^2+t}w,\label{inequ-w}
\end{eqnarray}
 for some positive constant $c_0$.
 
 Assume that $\lim_{t\rightarrow 0}w(x,t)=0$ for every $x\in\R^n\setminus\{0\}$. Then there exists a positive constant $\lambda(n,c_0)>0$ such that, $$w\leq Cf^{-\frac{n}{2}}e^{-f},\quad\text{on $\Omega_{\lambda},\quad\lambda\geq\lambda(n,c_0),$}$$
 for some positive constant $C=C\left(n,c_0,\|w\|_{L^{\infty}(\R^n\times(0,T))}
\right).$ 
\end{lemma}

\begin{proof}
The idea is to adapt the $L^2$-version of the maximum principle due to Karp and Li [Theorem $7.39$, \cite{Cho-Boo}] as follows.\\

First of all, we need to construct a suitable barrier function: let us compute the evolution equation satisfied by $f^{-n/2}e^{-f}$ on $\Omega_{\lambda}$:

\begin{eqnarray*}
(\partial_t-\Delta)(e^{-f})&=&\left(-\partial_tf+\Delta f-|\nabla f|^2\right)e^{-f}=\frac{n}{2t}e^{-f},\\
(\partial_t-\Delta)(f^{-n/2})&=&-\frac{n f^{-n/2}}{2t}\left\{-1+\left(n+1\right)f^{-1}-\frac{n}{2}\left(\frac{n}{2}+1\right)f^{-2}\right\},\\
(\partial_t-\Delta)\left(f^{-n/2}e^{-f}\right)&=&f^{-n/2}(\partial_t-\Delta)(e^{-f})-n\frac{|\nabla f|^2}{f}\cdot f^{-n/2}e^{-f}+(\partial_t-\Delta)(f^{-n/2})e^{-f}\\
&=&\left\{\frac{n}{2t}-\frac{n}{2t}\left(-1+\left(n+1\right)f^{-1}-\frac{n}{2}\left(\frac{n}{2}+1\right)f^{-2}\right)-n\frac{|\nabla f|^2}{f}\right\}f^{-n/2}e^{-f}\\
&=&-\frac{n}{2t}\left\{f^{-1}-\frac{n}{2}\left(\frac{n}{2}+1\right)f^{-2}\right\}f^{-n/2}e^{-f}\\
&\geq&-\frac{C(n)}{tf}\left\{1+C(n)\lambda^{-1}\right\}f^{-n/2}e^{-f},\\
&\geq&-\frac{C(n)}{tf}f^{-n/2}e^{-f},
\end{eqnarray*}
where $C(n)$ is a positive constant that may vary from line to line and where $\lambda\geq1.$
On the other hand,
\begin{eqnarray*}
(\partial_t-\Delta)\left(e^{-Bf^{-1}}\right)&=&-\frac{Be^{Bf^{-1}}}{tf}\left(1+\textit{O}(\lambda^{-1})+B\cdot\textit{O}(\lambda^{-3})\right)\\
&=&-\frac{Be^{Bf^{-1}}}{tf}\left(1+\textit{O}(\lambda^{-1})\right),
\end{eqnarray*}
if $\lambda^2\geq B\geq 1$.

Consequently,
\begin{eqnarray*}
\left(\partial_t-\Delta\right)\left(f^{-n/2}e^{-f-Bf^{-1}}\right)&=&\left(\partial_t-\Delta\right)\left(f^{-n/2}e^{-f}\right)e^{-Bf^{-1}}-2\nabla\left(f^{-n/2}e^{-f}\right)\cdot\nabla\left(e^{-Bf^{-1}}\right)\\
&&+\left(f^{-n/2}e^{-f}\right)\left(\partial_t-\Delta\right)\left(e^{-Bf^{-1}}\right)\\
&\geq&\left(-\frac{C(n)}{tf}-\frac{B}{tf}\left(1+\textit{O}(\lambda^{-1})\right)+\frac{2B}{tf}\left(1+\textit{O}(\lambda^{-1})\right)\right)f^{-n/2}e^{-f-Bf^{-1}}\\
&\geq&\left(-\frac{C(n)}{tf}+\frac{B}{tf}\left(1+\textit{O}(\lambda^{-1})\right)\right)f^{-n/2}e^{-f-Bf^{-1}}\\
&\geq&\frac{B}{2tf}\left(1+\textit{O}(\lambda^{-1})\right)f^{-n/2}e^{-f-Bf^{-1}},
\end{eqnarray*}
if $\lambda^2\geq B>>1.$

In particular, this shows that the function $v_{A,B}:=w-Af^{-n/2}e^{-f-Bf^{-1}}$ is a subsolution of the same differential inequality satisfied by $w$ for every positive constant $A$ if $\lambda^2\geq B>>c_0.$
For such parameters $\lambda$ and $B$, fix $A$ such that 

\begin{eqnarray}
v_{A,B}\leq \|w\|_{L^{\infty}(\R^n\times(0,T))}-A\lambda^{-n/2}e^{-\lambda-B\lambda^{-1}}\leq 0,\quad\text{on $\left\{(x,t)\in \mathbb{R}^n\times (0,T]\quad|\quad\frac{|x|^2}{4t}=\lambda\right\}$,}\label{cond-bdy}
\end{eqnarray}
where $T$ is a fixed positive time. 

Consider the following time-dependent weight function $h:\mathbb{R}^n\times[0,T)\rightarrow \mathbb{R} $ defined by:
\begin{eqnarray*}
&&h_{\beta,\gamma}(x,t):=\frac{\beta|x|^2}{{4(t-T)}}-\gamma f^{-1},\quad x\in \mathbb{R}^n,\quad t\in(0,T),\quad \beta\in (0,1),\quad \gamma>0.\\
\end{eqnarray*}
We compute:
\begin{eqnarray}\label{ham-jac-equ}
\partial_th_{\beta,\gamma}&=&-\frac{\beta|x|^2}{{4(t-T)^2}}+\gamma f^{-2}\partial_tf\\
&=&-\frac{\beta|x|^2}{{4(t-T)^2}}-\gamma\frac{1-(n/2)f^{-1}}{tf},\\
|\nabla h_{\beta,\gamma}|^2&=&\left|\frac{\beta x}{2(t-T)}+\gamma\frac{\nabla f}{f^2}\right|^2\\
&=&\frac{\beta^2|x|^2}{4(t-T)^2}+\frac{\gamma\beta}{(t-T)}\frac{x\cdot \nabla f}{f^2}+\gamma^2\frac{|\nabla f|^2}{f^4},\\
\partial_th_{\beta,\gamma}+|\nabla h_{\beta,\gamma}|^2&=&\beta(\beta-1)\frac{|x|^2}{4(t-T)^2}\\
&&-\frac{\gamma}{tf}\left(1-\frac{n}{2}f^{-1}-\gamma\frac{(f-(n/2))}{f^3}-\frac{\beta|x|^2}{2(t-T)f}\right).
\end{eqnarray}
Now, observe that, if $(x,t)\in\Omega_{\lambda}$,
\begin{eqnarray*}
\beta(\beta-1)\frac{|x|^2}{4(t-T)^2}&\leq& 0,\\
1-\frac{n}{2}f^{-1}-\gamma\frac{(f-(n/2))}{f^3}-\frac{\beta|x|^2}{2(t-T)f}&\geq&1-\frac{n}{2}\lambda^{-1}-\gamma\lambda^{-2}\\
&\geq& \frac{1}{4},
\end{eqnarray*}
if we choose $\lambda$, $\beta$ and $\gamma$ as follows:
\begin{eqnarray*}
\lambda>>\lambda(n,c_0)>0,\quad \gamma<< \lambda^2.
\end{eqnarray*}
 Therefore, with such choices of parameters, one has:
\begin{eqnarray*}
\partial_th_{\beta,\gamma}+|\nabla h_{\beta,\gamma}|^2&\leq&-\frac{\gamma}{|x|^2+t},\quad \forall \gamma>0.
\end{eqnarray*}

From now on, we write $h$ for $h_{\beta,\gamma}$ and $v$ for $v_{A,B}$.

Then, define $v_+:=\max\{v,0\}$ which is Lipschitz and observe that for each positive time $t\in(0,T)$, $v_+(t)\in L^2(e^{h(t)}dx)$ by the uniform boundedness of $w$. Therefore, by the co-area formula together with (\ref{cond-bdy}) and by integration by parts:

\begin{eqnarray*}
\frac{1}{2}\partial_t\int_{|x|^2>4\lambda t}v_+^2e^{h}dx&\leq&\int_{|x|^2>4\lambda t}v_+\left(\Delta v+\frac{c_0}{|x|^2+t} v+\frac{1}{2}\partial_thv\right)e^{h}dx\\
&&-2\lambda\int_{|x|^2=4\lambda t}v_+^2e^{h}dx\\
&\leq& \int_{|x|^2>4\lambda t}\left(-|\nabla v_+|^2+\frac{1}{2}\partial_thv_+^2+v_+|\nabla v_+||\nabla h|+\frac{c_0}{|x|^2+t} v_+^2\right)e^{h}dx\\
&\leq& \int_{|x|^2>4\lambda t}\left(-\frac{1}{2}|\nabla v_+|^2+\frac{1}{2}(\partial_th+|\nabla h|^2)v_+^2+\frac{c_0}{|x|^2+t} v_+^2\right)e^{h}dx\\
&\leq&\int_{|x|^2>4\lambda t}\left(-\frac{1}{2}|\nabla v_+|^2-\frac{\gamma}{|x|^2+t}v_+^2+\frac{c_0}{|x|^2+t} v_+^2\right)e^{h}dx\\
&\leq&0,
\end{eqnarray*}
if $\gamma$ is chosen such that $\gamma\geq c_0$.

By assumption on $w$ together with the Dominated Convergence Theorem, one concludes that:

\begin{eqnarray*}
\int_{|x|^2>4\lambda t}v_+^2e^{h}dx&\leq&\lim_{s\rightarrow 0}\int_{|x|^2>4\lambda s}v_+^2e^{h}dx=0,
\end{eqnarray*}
which implies the expected result.

\end{proof}
As a first consequence of Theorem \ref{theo-max-ppe-l2}, one gets:

\begin{theo}\label{prop-dec-time-diff-sol}
Let $u_i:\mathbb{R}^n\rightarrow N\subset \R^m$, $i=1,2$ be two solutions to the harmonic map flow coming out of the $0$-homogeneous map. Assume $u_1$ and $u_2$ are regular at infinity. Let $c_0$ be a positive constant such that: 
\begin{eqnarray*}
\sup_{\R^n\times\R_+}(|x|^2+t)|\nabla u_i|^2\leq c_0,\quad i=1,2.
\end{eqnarray*}
Then,
\begin{eqnarray}
|u_1-u_2|(x,t)\leq C_{\lambda}f^{-\frac{n}{2}}(x,t)e^{-f(x,t)},\quad (x,t)\in \Omega_{\lambda},\label{fast-decay-diff-sol}
\end{eqnarray}
 for some positive $\lambda\geq \lambda(n,c_0)>0$ and some positive constant $C=C\left(n,c_0\right)$.
\end{theo}
\begin{rk}
Theorem \ref{prop-dec-time-diff-sol} is stated for solutions that are regular at infinity but its proof only requires the first derivatives to decay as in Definition \ref{def-reg-inf}.
\end{rk}
\begin{proof}
We first do it in the case $N$ is a Euclidean sphere $\mathbb{S}^{m-1}\subset\R^m$. By a straightforward computation, one gets:
\begin{eqnarray*}
\partial_t|u_1-u_2|^2&=&\Delta|u_1-u_2|^2-2|\nabla (u_1-u_2)|^2\\
&&+2|\nabla u_1|^2<u_1,u_1-u_2>-2|\nabla u_2|^2<u_2,u_1-u_2>\\
&=&\Delta|u_1-u_2|^2-2|\nabla (u_1-u_2)|^2+(|\nabla u_1|^2+|\nabla u_2|^2)|u_1-u_2|^2.\\
\end{eqnarray*}
In particular, if one denotes the norm of the difference $w:=|u_1-u_2|$, then $w$ satisfies the following differential inequality in a weak sense:
\begin{eqnarray*}
\partial_tw\leq \Delta w+\frac{C}{|x|^2+t}w,
\end{eqnarray*}
where we used the fact that $\nabla u_i$, $i=1,2$ decay like $\left(t+|x|^2\right)^{-1/2}$. The result follows by applying Theorem \ref{theo-max-ppe-l2} to $w$. One could argue that Theorem \ref{theo-max-ppe-l2} is only stated for smooth subsolutions but its proof can be adapted to a regularization of $w$ of the form $w^{\varepsilon}:=(w^2+\varepsilon^2)^{1/2}$ for $\varepsilon$ positive.

For a general target $N$ isometrically embedded in $\R^m$, we use the following remark due to [p.$ 458$, \cite{Har-Mou}]: denote the vector field $u_2-u_1$ by $\xi$ and let $\xi^{\top}$ and $\xi^{\perp}$ be the orthogonal projections onto $T_{u_1}N$ and $(T_{u_1}N)^{\perp}$. Then, since $N$ is locally the graph of the second fundamental form $A$ over $T_{u_1}N$,
\begin{eqnarray*}
u_1+\xi=u_1+\xi^{\top}-\frac{1}{2}A(u_1)(\xi^{\top},\xi^{\top})+\textit{O}(|\xi|^3).
\end{eqnarray*}
In particular, this implies:
\begin{eqnarray*}
<u_1-u_2,A(u_1)(\nabla u_1,\nabla u_1)>&=&-<\xi^{\perp},A(u_1)(\nabla u_1,\nabla u_1)>\\
&=&\left<\frac{1}{2}A(u_1)(\xi^{\top},\xi^{\top})+\textit{O}(|\xi|^3),A(u_1)(\nabla u_1,\nabla u_1)\right>.
\end{eqnarray*}
By assumption, the vector field $\xi$ is bounded which allows us to estimate the previous term as follows:
\begin{eqnarray*}
|<u_1-u_2,A(u_1)(\nabla u_1,\nabla u_1)>|&\leq&C(N)|\nabla u_1|^2|\xi|^2\\
&\leq&\frac{C(N,u_1)}{|x|^2+t}|\xi|^2\\
&=&\frac{C(N,u_1)}{|x|^2+t}|u_1-u_2|^2.
\end{eqnarray*}
Consequently, by applying the previous reasoning to $u_2$, one gets:
\begin{eqnarray*}
\partial_t|u_1-u_2|^2\leq \Delta |u_1-u_2|^2-2|\nabla(u_1-u_2)|^2+\frac{C(N,u_1,u_2)}{|x|^2+t}|u_1-u_2|^2,
\end{eqnarray*}
i.e. $w:=|u_1-u_2|$ is a weak subsolution of the form

\begin{eqnarray*}
\partial_tw\leq \Delta w+\frac{C}{|x|^2+t}w,
\end{eqnarray*}
where $C$ is a positive constant depending on the geometry of $N$ and the first derivatives of the maps $u_1$ and $u_2$.
The result follows by applying Theorem \ref{theo-max-ppe-l2} to $w$.
\end{proof}

Similarly, Theorem \ref{theo-max-ppe-l2} can be applied to the vector field $\partial_tu+\frac{x}{2t}\cdot\nabla u$ if $u$ is a solution of the harmonic map flow coming out of the same $0$-homogeneous map: this vector field measures the deviation of $u$ from being an expanding solution.

\begin{theo}\label{comp-exp-sol-gal-sol}
Let $u:\R^n\times(0,T)\rightarrow N\subset\R^{m}$ be a smooth solution of the harmonic map flow. Then the function $\left|t\partial_tu+\frac{x}{2}\cdot\nabla u\right|$ is a weak subsolution of the heat equation with potential $|\nabla u|^2$:
$$(\partial_t-\Delta) \left|t\partial_tu+\frac{x}{2}\cdot\nabla u\right|\leq C(N)|\nabla u|^2\left|t\partial_tu+\frac{x}{2}\cdot\nabla u\right|.$$

In particular, let $u_i:\mathbb{R}^n\times(0,T)\rightarrow N\subset \R^m$, $i=1,2$ be two smooth solutions to the harmonic map flow coming out of the $0$-homogeneous map that are regular at infinity and such that $u_1$ is an expanding solution. Let $c_0$ be a positive constant such that:
\begin{eqnarray*}
\left|t\partial_t(u_2-u_1)+\frac{x}{2}\cdot\nabla (u_2-u_1)\right|+(|x|^2+t)|\nabla u_i|^2\leq c_0,\quad \text{on $\Omega_{\lambda}$}.
\end{eqnarray*}
Then,
\begin{eqnarray}
\left|t\partial_t(u_2-u_1)+\frac{x}{2}\cdot\nabla (u_2-u_1)\right|\leq Cf^{-\frac{n}{2}}e^{-f},\quad \text{on $ \Omega_{\lambda}$},\label{fast-decay-diff-sol-ene}
\end{eqnarray}
 for some positive $\lambda\geq \lambda(n,c_0)>0$ and some positive constant $C=C(n,c_0)$. 

\end{theo}

\begin{proof}
Let $u$ be as above and let us compute the evolution equation of $\partial_tu+\frac{x}{2t}\cdot\nabla u$ as follows:
\begin{eqnarray*}
(\partial_t-\Delta)(\partial_tu)&=&2A(u)(\nabla u,\nabla (\partial_tu))+D_uA(\partial_tu)(\nabla u,\nabla u),\\
(\partial_t-\Delta)\left(\frac{x}{2t}\cdot\nabla u\right)&=&-t^{-1}\left(\frac{x}{2t}\cdot\nabla u+\partial_tu\right)+\frac{A(u)(\nabla u,\nabla u)}{t}+\frac{x}{2t}\cdot\nabla (A(u)(\nabla u,\nabla u)),\\
(\partial_t-\Delta)\left(\partial_tu+\frac{x}{2t}\cdot\nabla u\right)&=&-t^{-1}\left(\partial_tu+\frac{x}{2t}\cdot\nabla u\right)+2A(u)(\nabla u,\nabla (\partial_tu))+D_uA(\partial_tu)(\nabla u,\nabla u)\\
&&+\frac{A(u)(\nabla u,\nabla u)}{t}+\frac{x}{2t}\cdot\nabla (A(u)(\nabla u,\nabla u)).
\end{eqnarray*}

Now, note that:
\begin{eqnarray*}
\frac{x}{2t}\cdot\nabla (A(u)(\nabla u,\nabla u))&=&D_uA\left(\nabla_{\frac{x}{2t}}u\right)(\nabla u,\nabla u)+2A(u)\left(\nabla_{\frac{x}{2t}}\nabla u,\nabla u\right)\\
&=&D_uA\left(\nabla_{\frac{x}{2t}}u\right)(\nabla u,\nabla u)+A(u)\left(\nabla\left(\nabla_{\frac{x}{t}}u\right),\nabla u\right)-\frac{1}{t}\cdot A(u)(\nabla u,\nabla u).
\end{eqnarray*}

Therefore,
\begin{eqnarray*}
(\partial_t-\Delta)\left(\partial_tu+\frac{x}{2t}\cdot\nabla u\right)&=&-t^{-1}\left(\partial_tu+\frac{x}{2t}\cdot\nabla u\right)+D_uA\left(\partial_tu+\nabla_{\frac{x}{2t}}u\right)(\nabla u,\nabla u)\\
&&+2A(u)\left(\nabla\left(\partial_tu+\nabla_{\frac{x}{2t}}u\right),\nabla u\right).
\end{eqnarray*}
Since,
\begin{eqnarray*}
A(u)\left(\nabla\left(\partial_tu+\nabla_{\frac{x}{2t}}u\right),\nabla u\right)\perp \partial_tu+\nabla_{\frac{x}{2t}}u,
\end{eqnarray*}
 
one gets:
\begin{eqnarray*}
(\partial_t-\Delta)\left|t\partial_tu+\frac{x}{2}\cdot\nabla u\right|^2\leq-2\left|\nabla\left(t\partial_tu+\frac{x}{2}\cdot\nabla u\right)\right|^2+C(N)|\nabla u|^2\left|t\partial_tu+\frac{x}{2}\cdot\nabla u\right|^2.
\end{eqnarray*}

 Therefore, by the Kato inequality, one gets, in the weak sense:
\begin{eqnarray*}
(\partial_t-\Delta)\left|t\partial_tu+\frac{x}{2}\cdot\nabla u\right|\leq C(N) |\nabla u|^2\left|t\partial_tu+\frac{x}{2}\cdot\nabla u\right|
.
\end{eqnarray*}
If $u_i$, $i=1,2$ are two solutions as in the statement of Proposition \ref{comp-exp-sol-gal-sol} then $w:=t\partial_t(u_1-u_2)+\frac{x}{2}\cdot\nabla (u_1-u_2)$ is equal to $t\partial_tu_1+\frac{x}{2}\cdot\nabla u_1$ since $u_2$ is an expander. Therefore, the result follows by applying Theorem \ref{theo-max-ppe-l2}.

\end{proof}
As a corollary, gradient estimates for the vector field $\partial_tu+\frac{x}{2t}\cdot\nabla u$ follows:

\begin{coro}[Shi's estimates]\label{coro-shi-est-grad-obst-tensor}
If $u_i:\mathbb{R}^n\times(0,T)\rightarrow N\subset \R^m$, $i=1,2$ are two solutions to the harmonic map flow smoothly coming out of the $0$-homogeneous map such that $u_1$ is an expanding solution, then
\begin{eqnarray*}
\sup_{\Omega_{2\lambda}}\sqrt{t}\left|\nabla\left(\partial_t(u_2-u_1)+\frac{x}{2t}\cdot\nabla (u_2-u_1)\right)\right|\leq C \sup_{\Omega_{\lambda}}\left|\partial_t(u_2-u_1)+\frac{x}{2t}\cdot\nabla (u_2-u_1)\right|,
\end{eqnarray*}
where $C$ is a uniform positive constant and where $\lambda$ is large enough. In particular, the gradient of the vector field $\partial_t(u_2-u_1)+\frac{x}{2t}\cdot\nabla (u_2-u_1)$ decays exponentially in space.

\end{coro}
\begin{rk}
Corollary \ref{coro-shi-est-grad-obst-tensor} does not give a sharp decay in space for the gradient of the vector field $\partial_t(u_2-u_1)+\frac{x}{2t}\cdot\nabla (u_2-u_1)$. Nonetheless, this will be sufficient for the sequel.
\end{rk}

\begin{proof}
The proof is very similar to the proof of Proposition \ref{autom-reg-infty}. Nonetheless, as we need Corollary \ref{coro-shi-est-grad-obst-tensor} for the proof of Proposition \ref{autom-reg-infty}, we give the main steps. Recall from the proof of Proposition \ref{comp-exp-sol-gal-sol} that the obstruction vector field $\Ob:= t\partial_t(u_2-u_1)+\frac{x}{2}\cdot\nabla (u_2-u_1)$ satisfies the following equation:
\begin{eqnarray*}
(\partial_t-\Delta)\Ob&=&2A(u_2)(\nabla u_2,\nabla \Ob)+D_{u_2}A(\Ob)(\nabla u_2,\nabla u_2),
\end{eqnarray*}
which implies by using Young's inequality together with the fact that $u_2$ is regular at infinity:
\begin{eqnarray*}
(\partial_t-\Delta)|\Ob|^2&\leq&-|\nabla\Ob|^2+\textit{O}\left((tf)^{-1}\right)|\Ob|^2.
\end{eqnarray*}
Therefore, the gradient $\nabla \Ob$ satisfies schematically:
\begin{eqnarray*}
(\partial_t-\Delta)\nabla\Ob&=&\Ob\ast\nabla u_2^{*3}+\nabla\Ob\ast\nabla u_2^{*2}+\Ob\ast\nabla^2u_2\ast\nabla u_2+\nabla^2\Ob\ast\nabla u_2+\nabla\Ob\ast\nabla^2u_2,
\end{eqnarray*}
where $A\ast B$ denotes any linear combination of contractions of two tensors $A$ and $B$. In particular, by using Young's inequality together with the fact that $u_2$ is regular at infinity:
\begin{eqnarray*}
(\partial_t-\Delta)|\nabla\Ob|^2&\leq&-|\nabla^2\Ob|^2+\textit{O}\left((tf)^{-1}\right)|\nabla\Ob|^2+\textit{O}\left((tf)^{-2}\right)|\Ob|^2.
\end{eqnarray*}
Now, consider the function $|\nabla\Ob|^2(a+|\Ob|^2)$ where $a$ is a non-negative constant, universally proportional to $\sup_{\Omega_{\lambda}}|\Ob|$ where $\lambda$ is such that Proposition \ref{comp-exp-sol-gal-sol} is applicable. Then, one computes:
\begin{eqnarray*}
(\partial_t-\Delta)\left(|\nabla\Ob|^2(a+|\Ob|^2)\right)\leq -\frac{|\nabla\Ob|^4}{2}+C\frac{a^2}{(tf)^2},
\end{eqnarray*}
for some positive constant $C$ not depending on $\Ob$. By using a suitable cut-off function $\psi$ in space-time coordinates, one gets the result by applying the maximum principle to $\psi|\nabla\Ob|^2(a+|\Ob|^2)$ as in the proof of Proposition \ref{autom-reg-infty}.
\end{proof}

Now, if $u_1$ and $u_2$ are two solutions of the harmonic map flow coming out of the same $0$-homogeneous initial condition, we want to show that not only the difference $u_1-u_2$ is decaying really fast but also that $f^{n/2}e^f(u_2-u_1)$ converges radially and smoothly to a map defined on $\mathbb{S}^{n-1}\subset \mathbb{R}^n$, as $t$ goes to zero.

A first step towards this assertion is to prove that the vector field $U:=f^{n/2}e^f(u_1-u_2)$ is regular at infinity in the following sense:

\begin{prop}\label{autom-reg-infty}
\begin{enumerate}
\item If $u_1$ and $u_2$ are expanding solutions that are regular at infinity, then for all $k\geq 1$, there exists a positive constant $C_k$ such that:
\begin{eqnarray*}
\sup_{\Omega_{2\lambda}}(tf)^{\frac{k}{2}}|\nabla^kU|(x,t)\leq C_k\sup_{\Omega_{\lambda}}|U|.
\end{eqnarray*}
\item Let $u_2$ be any solution of the harmonic map flow coming out of the $0$-homogeneous map $u_0$ and let $u_1$ be an expanding solution coming out of the same map $u_0$. Assume that $u_1$ and $u_2$ are regular at infinity then for all $k\geq 1$, there exists a positive constant $C_k$ such that:
\begin{eqnarray*}
\sup_{\Omega_{2\lambda}}t^{\frac{k}{2}}|\nabla^kU|(x,t)\leq C_k\sup_{\Omega_{\lambda}}|U|.
\end{eqnarray*}
\end{enumerate}
\end{prop}
\begin{rk}\label{rk-autom-reg}
Proposition \ref{autom-reg-infty} is stated in terms of solutions that are regular at infinity. By inspecting the proof given below, it turns out that if $u_2$ and $u_1$ are at least $C^{k+2}$, $k\geq 1$, regular at infinity in the sense that their derivatives, up to order $k+2$ decay as in Definition \ref{def-reg-inf} then the conclusion of Proposition \ref{autom-reg-infty} holds for the derivatives up to order $k$ of the rescaled difference $U=f^{n/2}e^f(u_1-u_2)$.  
\end{rk}
In order to prove Proposition \ref{autom-reg-infty}, we need the following technical lemma that establishes the evolution equations satisfied by the vector field $U$ and its derivative $\nabla U$:

\begin{lemma}\label{evo-equ-resc-vec-fiel}
One has the following evolution equations:
\begin{eqnarray}\label{evo-equ-resc-vec-fiel-1}
&&\left(\partial_t+X\cdot\nabla \right)|U|^2=\Delta |U|^2-2|\nabla U|^2 + \textit{O}\left((tf)^{-1}\right)|U|^2,\\
&&X:=\left(1+\frac{n}{2}f^{-1}\right)\frac{x}{t},\\
&&\left(\partial_t+X\cdot\nabla-\Delta \right)((tf)|\nabla U|^2)\leq-(tf)|\nabla^2U|^2+\textit{O}\left((tf)^{-1}\right)\left[|U|^2+(tf)|\nabla U|^2\right],\label{evo-equ-resc-vec-fiel-2}
\end{eqnarray}
on $\Omega_{\lambda}$.

\end{lemma}
\begin{proof}
Define $u:=u_1-u_2$.
We recall from the proofs of Theorem \ref{theo-max-ppe-l2} and Proposition \ref{prop-dec-time-diff-sol} the following computations that hold on $\Omega_{\lambda}$ :
\begin{eqnarray*}
\partial_tu&=&\Delta u+A(u_1)(\nabla u_1,\nabla u_1)-A(u_2)(\nabla u_2,\nabla u_2),\\
\partial_tF&=&\Delta F+\textit{O}\left((tf)^{-1}\right)F,\quad F:=f^{-n/2}e^{-f}.
\end{eqnarray*}
Therefore, on $\Omega_{\lambda}$:
\begin{eqnarray}
(\partial_t-\Delta)U&=&(\partial_t-\Delta)(F^{-1}u)\label{evo-equ-U-0}\\
&=&(\partial_t-\Delta)(F^{-1})u-2\nabla F^{-1}\cdot\nabla u+F^{-1}(\partial_t-\Delta)(u)\label{evo-equ-U-1}\\
&=&\textit{O}\left((tf)^{-1}\right)U-2\nabla \ln F^{-1}\cdot\nabla U\label{evo-equ-U-2}\\
&&+F^{-1}\left(A(u_1)(\nabla u_1,\nabla u_1)-A(u_2)(\nabla u_2,\nabla u_2)\right).\label{evo-equ-U-3}
\end{eqnarray}
Now,
\begin{eqnarray*}
\nabla \ln F=-\nabla f-\frac{n}{2}\nabla\ln f=-\frac{x}{2t}\left(1+\frac{n}{2}f^{-1}\right)=-\frac{X}{2}.
\end{eqnarray*}
Again, by using the proof of Proposition \ref{prop-dec-time-diff-sol}, this ends the proof of the evolution equation satisfied by $|U|^2$ on $\Omega_{\lambda}$.

Let us investigate the evolution equation satisfied by $(tf)|\nabla U|^2$. By derivating (\ref{evo-equ-U-0}), (\ref{evo-equ-U-2}) and (\ref{evo-equ-U-3}):
\begin{eqnarray}
\left(\partial_t+X\cdot\nabla \right)\nabla U&=&\Delta\nabla U-\nabla X\cdot\nabla U+\textit{O}\left((tf)^{-1}\right)\nabla U+\textit{O}\left(\nabla (tf)^{-1}\right)U\label{evo-equ-nabla-U-1}\\
&&+\nabla\left(F^{-1}\left\{A(u_1)(\nabla u_1,\nabla u_1)-A(u_2)(\nabla u_2,\nabla u_2)\right\}\right).\label{evo-equ-nabla-U-2}
\end{eqnarray}
Now,
\begin{eqnarray}
\nabla X\cdot \nabla U&=&\frac{1}{t}\nabla U+\textit{O}\left((tf)^{-1}\right)\nabla U\label{est-vec-field-X-1}\\
\nabla (tf)^{-1}&=&\textit{O}\left( (tf)^{-3/2}\right).\label{est-vec-field-X-2}
\end{eqnarray}
Then, we analyse the difference between the two second fundamental forms more precisely:
\begin{eqnarray*}
F^{-1}\left(A(u_1)(\nabla u_1,\nabla u_1)-A(u_2)(\nabla u_2,\nabla u_2)\right)&=&F^{-1}(\nabla u_1^{\ast2}+\nabla u_2^{\ast2})\ast u\\
&&+F^{-1}\nabla u\ast (\nabla u_1+\nabla u_2)\ast (u_1+u_2)\\
&=&\textit{O}((tf)^{-1})U +\textit{O}((tf)^{-1/2})\nabla U\\
&&+U\ast \nabla_{\nabla \ln F}(u_1+u_2)\ast (u_1+u_2).
\end{eqnarray*}

Therefore, by derivating the previous expression together with Young's inequality:
\begin{eqnarray*}
&&\left|\left<\nabla\left[F^{-1}\left(A(u_1)(\nabla u_1,\nabla u_1)-A(u_2)(\nabla u_2,\nabla u_2)\right)\right],\nabla U\right>\right|\leq\\
&&\textit{O}\left((tf)^{-1}+|\nabla_{\nabla \ln F}(u_1+u_2)|\right)|\nabla U|^2+\textit{O}\left((tf)^{-2}\right)|U|^2+\frac{1}{2}|\nabla^2U|^2\\
&&+\textit{O}\left(\left|\nabla(\nabla_{\nabla \ln F}(u_1+u_2))\right|+\left|\nabla_{\nabla \ln F}(u_1+u_2)\right||\nabla(u_1+u_2)|\right)|\nabla U||U|.\\
\end{eqnarray*}
The lemma is proved provided we show that the radial derivatives decay appropriately:
\begin{claim}\label{rad-der-ext-claim}
\begin{eqnarray*}
\nabla_{\nabla \ln F}(u_i)&=&\textit{O}((tf)^{-1}),\quad i=1,2,\\
\nabla(\nabla_{\nabla \ln F}(u_i))&=&\textit{O}((tf)^{-3/2}),\quad i=1,2.
\end{eqnarray*}

\end{claim}

\begin{proof}[Proof of Claim \ref{rad-der-ext-claim}]
As $u_1$ is an expanding solution: 
\begin{eqnarray*}
\nabla_{\frac{x}{2}}u_1=-t\partial_tu_1=-t(\Delta u_1+A(u_1)(\nabla u_1,\nabla u_1))=\textit{O}(f^{-1}),
\end{eqnarray*}
together with its derivatives since $u_1$ is assumed to be regular at infinity.

Now, by Proposition \ref{comp-exp-sol-gal-sol},
\begin{eqnarray*}
\nabla_{\frac{x}{2}}u_2&=&-t\partial_tu_2+\textit{O}\left(f^{-\frac{n}{2}}e^{-f}\right)\\
&=&-t\Delta u_2-tA(u_2)(\nabla u_2,\nabla u_2)+\textit{O}\left(f^{-\frac{n}{2}}e^{-f}\right)\\
&=&\textit{O}(f^{-1}).
\end{eqnarray*}
This implies the first gradient estimate for $u_1$.
By Corollary \ref{coro-shi-est-grad-obst-tensor},  the gradient of the vector field $\partial_t(u_2-u_1)+\frac{x}{2t}\cdot\nabla (u_2-u_1)$ decays faster than any polynomial. As the solution $u_2$ is assumed to be regular at infinity,
\begin{eqnarray*}
\nabla\left(\nabla_{\frac{x}{2}}u_2\right)&=&-t\nabla\left(\partial_tu_2\right)+\textit{O}\left(t^{-1/2}f^{-3/2}\right)\\
&=&-t\nabla\left(\Delta u_2+A(u_2)(\nabla u_2,\nabla u_2)\right)+\textit{O}\left(t^{-1/2}f^{-3/2}\right)\\
&=&\textit{O}(t^{-1/2}f^{-3/2}).
\end{eqnarray*}
This ends the proof of claim \ref{rad-der-ext-claim}.
\end{proof}

Finally, using (\ref{evo-equ-nabla-U-1}), (\ref{evo-equ-nabla-U-2}) together with the estimates (\ref{est-vec-field-X-1}) and (\ref{est-vec-field-X-2}) and Claim \ref{rad-der-ext-claim}, we can sum up this discussion as follows:
\begin{eqnarray*}
\left(\partial_t+X\cdot\nabla \right)|\nabla U|^2&\leq&\Delta|\nabla U|^2-\frac{3}{2}|\nabla^2U|^2-\frac{2}{t}|\nabla U|^2\\&&+\textit{O}\left((tf)^{-1}\right)|\nabla U|^2+\textit{O}\left((tf)^{-2}\right)|U|^2.
\end{eqnarray*}

Finally, recall that the space-time function $tf(x,t)=|x|^2/4+\frac{n}{2}t$ satisfies:
\begin{eqnarray*}
\left(\partial_t+X\cdot\nabla-\Delta \right)(tf)=\frac{2}{t}(tf)+\textit{O}(1),
\end{eqnarray*}
so that, by using Young's inequality regarding the norm of the second derivatives of $U$:
\begin{eqnarray*}
\left(\partial_t+X\cdot\nabla-\Delta \right)((tf)|\nabla U|^2)\leq-(tf)|\nabla^2U|^2+\textit{O}\left((tf)^{-1}\right)(tf)|\nabla U|^2+\textit{O}\left((tf)^{-1}\right)|U|^2.
\end{eqnarray*}

This finishes the proof of Lemma \ref{evo-equ-resc-vec-fiel}.
\end{proof}

\begin{proof}[Proof of Proposition \ref{autom-reg-infty}]
In order to prove the case $k=1$ (which is the most important one for the sequel), one proceeds analogously to a method initiated by Shi \cite{Shi-Def} which consists in considering the space-time function $(|U|^2+a)(tf)|\nabla U|^2$ where $a$ is a positive constant to be determined later. Schematically, we define $U_1:=|U|^2$ and $U_2:=(tf)|\nabla U|^2$ and we use (\ref{evo-equ-resc-vec-fiel-1}) and (\ref{evo-equ-resc-vec-fiel-2}) to get:
\begin{eqnarray*}
\left(\partial_t+X\cdot\nabla-\Delta \right)[(U_1+a)U_2]&\leq& -2(tf)^{-1}U_2^2+\textit{O}((tf)^{-1})U_1U_2-(tf)(U_1+a)|\nabla^2U|^2\\
&&+\textit{O}((tf)^{-1})(U_1+U_2)(U_1+a)-2\nabla U_1\cdot\nabla U_2.
\end{eqnarray*}
Now, by Young's inequality,
\begin{eqnarray*}
2|\nabla U_1\cdot\nabla U_2|&\leq& 8(tf)|U||\nabla U|^2|\nabla^2U|+2|\nabla (tf)||U||\nabla U|^3\\
&\leq&(tf)^{-1}U_2^2+\textit{O}(tf)U_1|\nabla^2U|^2+\textit{O}((tf)^{-1})U_1U_2.
\end{eqnarray*}
Therefore, by choosing $a$ proportional to $\sup_{\Omega_{\lambda}}|U|^2$ carefully,
\begin{eqnarray*}
\left(\partial_t+X\cdot\nabla-\Delta \right)[(U_1+a)U_2]&\leq&-(tf)^{-1}U_2^2+\textit{O}((tf)^{-1})(a^2+U_2(U_1+a))\\
&\leq&-(tf)^{-1}\left[\frac{(U_2(U_1+a))^2}{a^2}-a^2\right].
\end{eqnarray*}
By mimicking the case where the solutions $u_i$, $i=1,2$ are expanders, we consider the following rescaling $\bar{U}_t(x):=U(\sqrt{t}x,t)$ which implies:
\begin{eqnarray*}
\left(t\partial_t+\overline{X}\cdot\nabla-\Delta \right)[(\overline{U_1}+a)\overline{U_2}]&\leq&-C\overline{f}^{-1}\left[\frac{(\overline{U_2}(\overline{U_1}+a))^2}{a^2}-a^2\right],
\end{eqnarray*}
where $C$ is a positive constant and where the quantities $U_1$ and $U_2$ denoted with a bar have been composed with $(\sqrt{t}x,t)$. and where,
\begin{eqnarray*}
\bar{f}_t(x):=\frac{|x|^2}{4}+\frac{n}{2},\quad \overline{X}_t(x)=(1+\textit{O}(\bar{f}^{-1}))\frac{x}{2}.
\end{eqnarray*}
Now, choose a radial cut-off function $\phi_R$ defined on $M$ such that:
\begin{eqnarray*}
&& \supp(\phi_R)\subset \{4\lambda\leq|x|^2\leq R^2\},\quad\phi_R\equiv 1\quad\mbox{in $\{8\lambda\leq|x|^2\leq (R/2)^2\}$},\\
&& \phi_R\equiv 0\quad\mbox{in $\{|x|^2\leq 4\lambda\}\cup\{|x|^2\geq R^2\}$},\\
 &&-\frac{c}{R}\leq \partial_r\phi_R\leq 0,\quad\frac{(\partial_r\phi_R)^2}{\phi_R}\leq \frac{c}{R^2},\quad \nabla^2\phi_R(\partial_r,\partial_r)\geq -\frac{c}{R^2},\quad \mbox{on $\{(R/2)^2\leq|x|^2\leq R^2\}$}.
 \end{eqnarray*}
 In particular:
 \begin{eqnarray*}
&&\left(t\partial_t+\overline{X}\cdot\nabla-\Delta \right)\phi_R+\frac{|\nabla\phi_R|^2}{\phi_R}\leq \frac{C}{\bar{f}},
\end{eqnarray*}
for some positive universal constant $C$. Note the crucial sign here in front of the vector field $\overline{X}$. Therefore, one gets:
\begin{eqnarray*}
\phi_R\left(t\partial_t+\overline{X}\cdot\nabla-\Delta \right)((\overline{U_1}+a)\overline{U_2}\phi_R)&\leq& -C\overline{f}^{-1}\left[\frac{(\overline{U_2}(\overline{U_1}+a)\phi_R)^2}{a^2}-a^2\right]\\
&&-2\phi_R\nabla(\overline{U_2}(\overline{U_1}+a))\cdot\nabla\phi_R.
\end{eqnarray*}
If $(\overline{U_2}(\overline{U_1}+a)\phi_R$ attains its (positive) maximum on $M\times(0,T]$ at a point $(x_0,t_0)$ then the following relations hold at this point:

\begin{eqnarray*}
&&\overline{U_2}(\overline{U_1}+a)\nabla(\phi_R)+\phi_R\nabla(\overline{U_2}(\overline{U_1}+a))=0,\\
0&\leq& -C\left[\frac{(\overline{U_2}(\overline{U_1}+a)\phi_R)^2}{a^2}-a^2\right]-2\phi_R\nabla(\overline{U_2}(\overline{U_1}+a))\cdot\nabla\phi_R\bar{f}\\
&\leq& -C\left[\frac{(\overline{U_2}(\overline{U_1}+a)\phi_R)^2}{a^2}-a^2\right]+2\bar{f}\frac{|\nabla\phi_R|^2}{\phi_R}(\overline{U_2}(\overline{U_1}+a)\phi_R)\\
&\leq&-C\left[\frac{(\overline{U_2}(\overline{U_1}+a)\phi_R)^2}{a^2}-a^2\right]+C(\overline{U_2}(\overline{U_1}+a)\phi_R),
\end{eqnarray*}
which implies the expected result, i.e. 
$$\sup_{M\times(0,T]}\overline{U_2}(\overline{U_1}+a)\phi_R\leq a^2,$$ which is equivalent to: $$\sup_{\Omega_{2\lambda}}(t+|x|^2)|\nabla U|^2\leq C\sup_{\Omega_{\lambda}}|U|^2.$$

If $u_1$ and $u_2$ are expanding solutions, this proves Proposition \ref{autom-reg-infty}, since the function $\overline{U_2}(\overline{U_1}+a)$ is constant in time. 

If $u_1$ or $u_2$ is an arbitrary solution then one multiplies $\overline{U_2}(\overline{U_1}+a)$ by a cut-off function in time which leads to the expected estimate: in this case, the derivative of $U$ does not decay slower than $U$ but is not expected to decay faster as in the case of expanders.

\end{proof}

\section{Definition of a relative entropy}\label{section-rel-ent-Ilm-conj}
We are now in a position to state and prove the main theorem of this section:
\begin{theo}\label{mono-rel-ent}
Let $n\geq 3$. Let $u:\mathbb{R}^n\rightarrow N\subset\R^m$ be a smooth solution to the harmonic map flow coming out of a smooth $0$-homogeneous map $u_0$. Let $u_b$ be a background smooth expander smoothly coming out of $u_0$. Then the entropy $\mathcal{E}(u,u_b)$ relative to $u_b$ introduced in (\ref{def-rel-entropy}) is well-defined and is non-increasing. More precisely:
\begin{eqnarray}
\frac{d}{dt}\mathcal{E}(u,u_b)(t)=-2t\int_{\mathbb{R}^n}\left|\partial_tu+\frac{x}{2t}\cdot\nabla u\right|^2\frac{e^{\frac{|x|^2}{4t}}}{(4\pi t)^{n/2}}dx,\quad t>0.\label{evo-rel-entropy}
\end{eqnarray}
Moreover, the relative entropy $\mathcal{E}(u,u_b)(t)$ is bounded from above and from below for all time:
\begin{eqnarray*}
-\infty<\lim_{t\rightarrow+\infty}\mathcal{E}(u,u_b)(t)\leq \lim_{t\rightarrow 0}\mathcal{E}(u,u_b)(t)<+\infty.
\end{eqnarray*}

\end{theo}
\begin{proof}
We start by proving that the right-hand side of (\ref{def-rel-entropy}) is well-defined. By rescaling arguments, it suffices to prove it at time $t=1$.
To do so, define the backward heat kernel evaluated at time $t=1$ by:
$$G(x):=\frac{e^{\frac{|x|^2}{4}}}{(4\pi )^{n/2}},$$
where $x\in\mathbb{R}^n$, and
observe that for every positive radius $R$:
\begin{eqnarray*}
\int_{B(0,R)}\left(|\nabla u|^2-|\nabla u_b|^2 \right)Gdx&=&\int_{B(0,R)}\left<u-u_b,-\left(\Delta+\frac{x}{2}\cdot\nabla\right)(u+u_b)\right>Gdx\\
&&+\int_{S(0,R)}\left<u-u_b,\nabla_{\mathbf{n}}(u+u_b)\right>Gd\sigma_{S(0,R)}\\
&=&\int_{B(0,R)}\left<u-u_b,A(u)(\nabla u,\nabla u)+A(u_b)(\nabla u_b,\nabla u_b)\right>Gdx\\
&&-\int_{B(0,R)}\left<u-u_b,\partial_tu+\frac{x}{2}\cdot\nabla u\right>Gdx\\
&&+\int_{S(0,R)}\left<u-u_b,\nabla_{\mathbf{n}}(u+u_b)\right>Gd\sigma_{S(0,R)}.
\end{eqnarray*}
Now, by Propositions \ref{prop-dec-time-diff-sol} and \ref{comp-exp-sol-gal-sol} together with the fact that $u$ and $u_b$ are regular at infinity,  
\begin{eqnarray*}
\left<u-u_b,A(u)(\nabla u,\nabla u)+A(u_b)(\nabla u_b,\nabla u_b)\right>G&=&\textit{O}(f^{-n/2-1}),\\
\left<u-u_b,\partial_tu+\frac{x}{2}\cdot\nabla u\right>G&=&\textit{O}(f^{-n}e^{-f}),\\
<u-u_b,\nabla_{\mathbf{n}}(u+u_b)>G&=&\textit{O}(f^{-n/2-1/2}).
\end{eqnarray*}
(Here, we do not use the sharp decay of the radial derivatives of $u$ and $u_b$.)

Therefore, it implies that the limit $$\lim_{R\rightarrow+\infty}\int_{B(0,R)}\left(|\nabla u|^2-|\nabla u_b|^2 \right)Gdx,$$ is well-defined.

Now, we prove the monotonicity result.
Firstly, note that the right-hand side of $(\ref{evo-rel-entropy})$ is well-defined by Proposition \ref{comp-exp-sol-gal-sol}.

In order to prove (\ref{evo-rel-entropy}), we proceed as in \cite{Str-Har-Map} by introducing the rescaled map $u_t(\cdot):=u(\sqrt{t}\cdot,t)$ if $u$ is a solution to the harmonic map flow. 
By the co-area formula, 
\begin{eqnarray*}
\frac{d}{dt}\bigg\rvert_{t=1}\int_{B(0,R)}t|\nabla u|^2(x,t)\frac{e^{\frac{|x|^2}{4t}}}{(4\pi t)^{n/2}}dx&=&\frac{d}{dt}\bigg\rvert_{t=1}\int_{B(0,R/\sqrt{t})}|\nabla u_t|^2(x)\frac{e^{\frac{|x|^2}{4}}}{(4\pi)^{n/2}}dx\\
&=&2\int_{B(0,R)}\left<\nabla u,\nabla\left(\partial_tu+\frac{x}{2}\cdot\nabla u\right)\right>Gdx\\
&&-\frac{R}{2}\int_{S(0,R)}|\nabla u|^2Gdx.
\end{eqnarray*}
Now, let us handle the first term of the right-hand side of the previous equality. By integrating by parts once:
\begin{eqnarray*}
\int_{B(0,R)}\left<\nabla u,\nabla\left(\partial_tu+\frac{x}{2}\cdot\nabla u\right)\right>Gdx&=&-\int_{B(0,R)}\left<\Delta u+\frac{x}{2}\cdot\nabla u,\partial_tu+\frac{x}{2}\cdot\nabla u\right>Gdx\\
&&+\int_{S(0,R)}\left<\nabla_{\partial_r}u,\partial_tu+\frac{x}{2}\cdot\nabla u\right>Gd\sigma_{S(0,R)}.
\end{eqnarray*}
Now, as $u$ is a solution to the harmonic map flow, $\partial_tu-\Delta u\perp \partial_t u+\frac{x}{2}\cdot\nabla u$, so:
\begin{eqnarray*}
\int_{B(0,R)}\left<\nabla u,\nabla\left(\partial_tu+\frac{x}{2}\cdot\nabla u\right)\right>Gdx&=&-\int_{B(0,R)}\left|\partial_tu+\frac{x}{2}\cdot\nabla u\right|^2Gdx\\
&&+\int_{S(0,R)}\left<\nabla_{\partial_r}u,\partial_tu+\frac{x}{2}\cdot\nabla u\right>Gd\sigma_{S(0,R)}\\
&=&-\int_{B(0,R)}\left|\partial_tu+\frac{x}{2}\cdot\nabla u\right|^2Gdx+\textit{O}\left(R^{-2}\right),
\end{eqnarray*}
where we used Proposition \ref{comp-exp-sol-gal-sol} to estimate the boundary term and the fact that $\nabla u=\textit{O}(f^{-1/2})$. To sum it up, we get:
\begin{eqnarray*}
\frac{d}{dt}\bigg\rvert_{t=1}\int_{B(0,R)}t|\nabla u|^2(x,t)\frac{e^{\frac{|x|^2}{4t}}}{(4\pi t)^{n/2}}dx&=&
-2\int_{B(0,R)}\left|\partial_tu+\frac{x}{2}\cdot\nabla u\right|^2Gdx\\
&&-\frac{R}{2}\int_{S(0,R)}|\nabla u|^2Gd\sigma_{S(0,R)}+\textit{O}\left(R^{-2}\right).
\end{eqnarray*}
This implies when one applies the previous calculation to the background expander $u_b$:
\begin{eqnarray*}
\frac{d}{dt}\bigg\rvert_{t=1}\int_{B(0,R)}t(|\nabla u|^2-|\nabla u_b|^2)(x,t)\frac{e^{\frac{|x|^2}{4t}}}{(4\pi t)^{n/2}}dx&=&
-2\int_{B(0,R)}\left|\partial_tu+\frac{x}{2}\cdot\nabla u\right|^2Gdx\\
&&-\frac{R}{2}\int_{S(0,R)}(|\nabla u|^2-|\nabla u_b|^2)Gd\sigma_{S(0,R)}\\
&&+\textit{O}\left(R^{-2}\right).
\end{eqnarray*}
What remains to be done is the analysis of the boundary integral of the previous estimate: observe that

\begin{eqnarray*}
\int_{S(0,R)}(|\nabla u|^2-|\nabla u_b|^2)Gd\sigma_{S(0,R)}&=&\int_{S(0,R)}<\nabla (u-u_b),\nabla (u+u_b)>Gd\sigma_{S(0,R)}\\
&=&\int_{S(0,R)}\left<\nabla ((u-u_b)G),\nabla (u+u_b)\right>d\sigma_{S(0,R)}\\
&&-\int_{S(0,R)}\left<u-u_b,\nabla_{\frac{x}{2}} (u+u_b)\right>Gd\sigma_{S(0,R)}\\
&=&\int_{S(0,R)}\left<\nabla ((u-u_b)G),\nabla (u+u_b)\right>d\sigma_{S(0,R)}\\
&&+\int_{S(0,R)}\left<\textit{O}(f^{-n/2}),\nabla_{\frac{x}{2}} (u+u_b)\right>d\sigma_{S(0,R)}.
\end{eqnarray*}
Now, since $\nabla_{\frac{x}{2}} (u+u_b)=\textit{O}(f^{-1})$, 
\begin{eqnarray*}
\int_{S(0,R)}\left<\textit{O}(f^{-n/2}),\nabla_{\frac{x}{2}} (u+u_b)\right>d\sigma_{S(0,R)}=\textit{O}(R^{-3}),
\end{eqnarray*}

and we are left with the understanding of the term $\nabla ((u-u_b)G)$. Thanks to Proposition \ref{autom-reg-infty} and with the same notations, 
\begin{eqnarray*}
\nabla ((u-u_b)G)=\nabla (f^{-n/2}U)=f^{-n/2}\nabla U+\textit{O}(f^{-(n+1)/2})=\textit{O}(f^{-n/2}).
\end{eqnarray*}
Consequently,
\begin{eqnarray*}
\int_{S(0,R)}\left<\nabla ((u-u_b)G),\nabla (u+u_b)\right>d\sigma_{S(0,R)}=\textit{O}(R^{-2}).
\end{eqnarray*}
Note the crucial decay given by Proposition \ref{autom-reg-infty}, a rough decay of the form $G\nabla (u-u_b)=\textit{O}(f^{(1-n)/2})$ would not have been sufficient to conclude the monotonicity formula.

Finally, 
\begin{eqnarray*}
\frac{d}{dt}\bigg\rvert_{t=1}\int_{B(0,R)}t(|\nabla u|^2-|\nabla u_b|^2)(x,t)\frac{e^{\frac{|x|^2}{4t}}}{(4\pi t)^{n/2}}dx&=&
-2\int_{B(0,R)}\left|\partial_tu+\frac{x}{2}\cdot\nabla u\right|^2Gdx+\textit{O}\left(R^{-1}\right).
\end{eqnarray*}
By integrating first in time and by reasoning as we did previously to ensure the finiteness of the relative entropy, one obtains by letting $R$ go to $+\infty$ the desired result.\\

We still need to prove that the limits of $\mathcal{E}(u,u_b)(t)$ as $t$ goes to $0$ or to $+\infty$ are finite. To do so, we cut the integral into two parts as follows:
\begin{eqnarray*}
\mathcal{E}(u,u_b)(t)&=&\int_{B(0,\sqrt{t})}t(|\nabla u|^2-|\nabla u_b|^2)G_tdx+\int_{|x|>\sqrt{t}}t(|\nabla u|^2-|\nabla u_b|^2)G_tdx,
\end{eqnarray*}
where $$G_t(x):=\frac{e^{\frac{|x|^2}{4t}}}{(4\pi t)^{\frac{n}{2}}}.$$

Now, on $B(0,\sqrt{t})$, one has:
\begin{eqnarray*}
|\nabla u|^2+|\nabla u_b|^2\leq \frac{C}{|x|^2+t},\quad G_t(x)\leq \frac{e^{\frac{1}{4}}}{(4\pi t)^{\frac{n}{2}}}.
\end{eqnarray*}
This implies:
\begin{eqnarray*}
\left|\int_{B(0,\sqrt{t})}t(|\nabla u|^2-|\nabla u_b|^2)G_tdx\right|\leq\frac{C}{t^{n/2}}\vol B(0,\sqrt{t})\leq C,
\end{eqnarray*}
where $C$ is a positive constant independent of time that may vary from line to line. 

On the region $\R^n\setminus B(0,\sqrt{t})$, by integrating by parts,
\begin{eqnarray*}
\int_{|x|>\sqrt{t}}t(|\nabla u|^2-|\nabla u_b|^2)G_tdx&=&\int_{|x|>\sqrt{t}}t\left<u-u_b,A(u)(\nabla u,\nabla u)+A(u_b)(\nabla u_b,\nabla u_b)\right>G_tdx\\
&&-\int_{|x|>\sqrt{t}}t\left<u-u_b,\partial_tu+\frac{x}{2t}\cdot\nabla u\right>G_tdx\\
&&-\int_{S(0,\sqrt{t})}t<u-u_b,\nabla_{\mathbf{n}}(u+u_b)>G_td\sigma_{S(0,\sqrt{t})}\\
&=&I(t)+II(t)+III(t).
\end{eqnarray*}
We only prove the boundedness in time of $I$, the terms $II$ and $III$ can be handled similarly.

By Proposition \ref{prop-dec-time-diff-sol}, 
\begin{eqnarray*}
|I(t)|&\leq& Ct\int_{|x|>\sqrt{t}}(tf)^{-n/2}\cdot (tf)^{-1}dx\\
&\leq& C\int_{|y|>1}\frac{dy}{(|y|^2+1)^{\frac{n}{2}+1}}.
\end{eqnarray*}

\end{proof}

Recall that if $u:\R^n\times\R_+\rightarrow N$ is a smooth solution of the harmonic map flow coming out of a smooth $0$-homogeneous map $u_0:\R^n\setminus\{0\}\rightarrow N$ then the family of rescaled maps $(u_{\lambda})_{\lambda>0}$ is defined by (\ref{resc-cond}).
\begin{theo}(Ilmanen's conjecture on expanders: the smooth case)\label{ilmanen-smooth-conj}

Let $n\geq 3$. 
Let $u:\R^n\times\R_+\rightarrow N$ be a smooth solution of the harmonic map flow coming out of a smooth $0$-homogeneous map $u_0:\R^n\setminus\{0\}\rightarrow N$. Assume $u$ is regular at infinity. Then the family of rescaled maps $(u_{\lambda})_{\lambda>0}$ is a compact set in the smooth topology of solutions of the harmonic map flow coming out of the same initial condition $u_0$. In particular, as $ \lambda$ goes to $0$ or $+\infty$, there is a subsequence $(u_{\lambda_i})_i$ converging to a smooth expanding solution coming out of $u_0$ that is regular at infinity. 
\end{theo}
\begin{proof}
By the definition of the rescaled maps $u_{\lambda}$ together with the fact that $u$ is regular at infinity, the following inequality holds for every integer $k\geq 0$ and $l\geq0$:
\begin{eqnarray*}
|\nabla^k\partial_t^l( u_{\lambda})|(x,t)\leq\frac{C_{k,l}}{(|x|^2+t)^{\frac{k}{2}+l}},\quad\forall (x,t)\in\R^n\times\R_+,
\end{eqnarray*}
for some uniform positive constant $C_{k,l}$.
By Arzela-Ascoli's Theorem, the family of maps $(u_{\lambda})_{\lambda>0}$ is compact in the smooth topology. Now, observe that:
\begin{eqnarray*}
\mathcal{E}(u_{\lambda},(u_b)_{\lambda})(t)=\mathcal{E}(u_{\lambda},u_b)(t)=\mathcal{E}(u,u_b)(\lambda^2t), \quad t>0,\quad \lambda >0.
\end{eqnarray*}

Therefore, by Theorem \ref{mono-rel-ent}, as $\lambda$ goes to $+\infty$ (or $0$),  $\mathcal{E}(u_{\lambda},(u_b)_{\lambda})(t)$ converges to a finite value independent of time. Finally, $(u_{\lambda})_{\lambda>0}$ subconverges to a smooth solution of the harmonic map flow $u_{\infty}$ that is regular at infinity.

 We claim that $u_{\infty}$ is necessarily an expanding solution that comes out of $u_0$. 

Indeed, if $s<t$, and for any $\lambda >0$:
\begin{eqnarray*}
\mathcal{E}(u_{\lambda},u_b)(t)-\mathcal{E}(u_{\lambda},u_b)(s)&=&-2\int_{[s,t]\times\R^n}\tau\left|\partial_{\tau}u_{\lambda}+\frac{x}{2\tau}\cdot\nabla u_{\lambda}\right|^2G_{\tau}dxd\tau.
\end{eqnarray*}
By the previous observations, we are done if one can invert the limits in the previous integral by using Lebesgue's Theorem. Since 
\begin{eqnarray*}
\left|t\partial_t(u_{\lambda}-u_b)+\frac{x}{2}\cdot\nabla (u_{\lambda}-u_b)\right|+(|x|^2+t)|\nabla u_{\lambda}|^2\leq c_0,\quad \text{on $\R^n\times\R_+$},
\end{eqnarray*}
 for a positive constant $c_0$ independent of $\lambda$, one can apply Proposition \ref{comp-exp-sol-gal-sol} to get:
\begin{eqnarray*}
\left|t \partial_{t}u_{\lambda}+\frac{x}{2}\cdot\nabla u_{\lambda}\right|\leq Cf^{-\frac{n}{2}}e^{-f},\quad \text{on $\{x\in\R^n||x|^2>4\mu t\}$},
\end{eqnarray*}
for some large $\mu\geq \mu(n,c_0)>0$ and for some positive constant $C$ independent of $\lambda$.

Finally, since the maps $u_{\lambda}$, $\lambda>0$, are uniformly regular at infinity, the harmonic map flow equation shows that there exists a uniform positive constant $C$ such that: $|u_{\lambda}(x,t)-u_0(x/|x|)|\leq C/(|x|^2+t)$ holds for $(x,t)\in\R^n\times\R_+$ and $\lambda>0$. This shows immediately that $u_{\infty}$ is coming out of $u_0$ when one lets $\lambda$ go to $+\infty$ (or $0$). This finishes the proof of Theorem \ref{ilmanen-smooth-conj}.
\end{proof}

\section{Pohozaev identity for expanders of the harmonic map flow}\label{section-pohozaev-identity}
In this section, we derive Pohozaev identities for solutions of the harmonic map flow. The main result when applied to expanding solutions is a static monotonicity formula in the spirit of the well-known monotonicity formula established by Struwe for solutions of the Harmonic map flow \cite{Str-Har-Map}.

Define formally the pointwise energy of a map $u:\R^n\times(0,T)\rightarrow N\subset\R^m$ by $$e(u):=\frac{|\nabla u|^2}{2},\text{ on $\R^n\times (0,T)$}.$$
The following proposition is a straightforward adaptation of [Proposition $3.16$, \cite{Der-Lam-HMF}]:
\begin{prop}[Pohozaev formula]\label{Poho-for-gal-sol}
Let $u:\R^n\times(0,T)\rightarrow N\subset\R^m$ be a (smooth) solution to the harmonic map flow. Then, for any $C^1$ vector field $\zeta:\R^n\times(0,T)\rightarrow \R^m$ compactly supported in space,
\begin{eqnarray*}
<\partial_tu,\nabla_{\zeta}u>_{L^2(\R^n\times[t_1,t_2])}&=&<e(u),\div\zeta>_{L^2(\R^n\times[t_1,t_2])}\\
&&-\frac{1}{2}<\mathcal{L}_{\zeta}(\eucl),\nabla u\otimes\nabla u>_{L^2(\R^n\times[t_1,t_2])},\\
\frac{1}{2}<\mathcal{L}_{\zeta}(\eucl),\nabla u\otimes\nabla u>&:=&\nabla_i\zeta_j\nabla_iu_k\nabla_ju_k,
\end{eqnarray*}
where $\mathcal{L}_{\zeta}(\eucl)$ denotes the Lie derivative of the Euclidean metric along the vector field $\zeta$.

And, for any $C^1$ function $\theta:\R^n\times(0,T)\rightarrow \R$ compactly supported in space, and $0<t_1<t_2<T$,

\begin{eqnarray*}
&&\int_{L^2(\R^n\times[t_1,t_2])}|\partial_tu|^2\theta dxdt+\left[\int_{\R^n}e(u)\theta dx\right]_{t_1}^{t_2}=\int_{\R^n\times[t_1,t_2]}e(u)\partial_t\theta-<\nabla_{\nabla \theta}u,\partial_tu> dxdt.
\end{eqnarray*}

\end{prop}
By adapting the proof of Proposition \ref{Poho-for-gal-sol}, one gets the following static Pohozaev identity for expanding solutions:
\begin{coro}[Pohozaev identity for expanders]\label{Poho-exp}
Let $u:\R^n\rightarrow N\subset\R^m$ be a smooth expanding solution to the harmonic map flow. Then, for any radius $R>0$,
\begin{eqnarray*}
R^{-1}\int_{S(0,R)}|\nabla_{r\partial r}u|^2d\sigma_f+\int_{B(0,R)}\left(\frac{r^2}{2}+n-2\right)|\nabla u|^2d\mu_f=R^{-1}\int_{S(0,R)}|\nabla^{sph}u|_{sph}^2d\sigma_f,
\end{eqnarray*}
where $f(x):=|x|^2/4$, $d\mu_f:=e^fdx$ and where $d\sigma_f:=e^fd\sigma_{\partial B(0,R)}$ denotes the induced measure on $\partial B(0,R)$. 
In other words:
\begin{eqnarray*}
\partial_R\left(R^{2-n}\int_{B(0,R)}|\nabla u|^2d\mu_f\right)=R^{1-n}\int_{B(0,R)}\frac{r^2}{2}|\nabla u|^2d\mu_f+2R^{-n}\int_{S(0,R)}|\nabla_{r\partial_r}u|^2d\sigma_f.
\end{eqnarray*}
In particular, the frequency function $r\in\R_+\rightarrow r^2\fint_{B(0,r)}|\nabla u|^2d\mu_f$ is increasing unless $u$ is a constant map. 
\end{coro}
\begin{proof}
The proof is along the same lines as the one of Proposition \ref{Poho-for-gal-sol}. Indeed, multiply the static equation satisfied by $u$ by $\nabla_{r\partial_r/2}=:\nabla_{\nabla f}u$ and integrate over $B(0,R)$ with respect to the measure $e^fdx$ as follows:
\begin{eqnarray*}
\int_{B(0,R)}<\Delta_fu,\nabla_{\nabla f}u>d\mu_f&=&\int_{B(0,R)}<-A(u)(\nabla u,\nabla u),\nabla_{\nabla f}u>d\mu_f\\
&=&0
\end{eqnarray*}
since $A(u)(\nabla u,\nabla u)\perp T_uN$ and $\nabla_{\nabla f}u\in T_uN$. By integrating by parts, one gets:
\begin{eqnarray*}
\int_{B(0,R)}<\Delta_fu,\nabla_{\nabla f}u>d\mu_f&=&-\int_{B(0,R)}<\nabla u,\nabla (\nabla_{\nabla f}u)>d\mu_f\\
&&+\int_{S(0,R)}<\nabla_{\partial_r}u,\nabla_{\nabla f}u>d\sigma_f\\
&=&-\int_{B(0,R)}\nabla^2f(\nabla u,\nabla u)+\nabla_{\nabla f}\frac{|\nabla u|^2}{2}d\mu_f\\
&&+\frac{R}{2}\int_{S(0,R)}|\nabla_{\partial_r}u|^2d\sigma_f\\
&=&-\int_{B(0,R)}\frac{|\nabla u|^2}{2}d\mu_f+\int_{B(0,R)}\Delta_ff\frac{|\nabla u|^2}{2}d\mu_f\\
&&+\frac{R}{2}\int_{S(0,R)}|\nabla_{\partial_r}u|^2d\sigma_f-\frac{R}{4}\int_{S(0,R)}|\nabla u|^2d\sigma_f,\\
\end{eqnarray*}
The result follows by noting that $\Delta_ff=n/2+r^2/4$ and $\vol_{n-1}S(0,R)=n\cdot R^{-1}\vol B(0,R).$

\end{proof}

As a first corollary, we get the following rigidity statement about constant maps, interpreted here as expanders of the harmonic map flow:
\begin{coro}\label{coro-unique-exp-pt}
Let $u:\R^n\rightarrow N$ be a smooth expanding solution coming out of the constant map $P_0\in N$ which is regular at infinity. Then $u\equiv P_0$. 
\end{coro}

\begin{proof}
By Corollary \ref{Poho-exp}, 
\begin{eqnarray*}
\int_{B(0,R)}\frac{1}{2}|\nabla^{sph} u|_{sph}^2d\mu_f=\int_{B(0,R)}\frac{r^2}{2}|\nabla u|^2d\mu_f\leq R^{-1}\int_{\partial B(0,R)}|\nabla^{sph}u|_{sph}^2d\sigma_f,
\end{eqnarray*}
for every positive radii $R$. Define $$y(R):=\int_{B(0,R)}|\nabla^{sph} u|_{sph}^2d\mu_f,\quad R>0.$$ Then, the previous inequality can be interpreted as follows:
\begin{eqnarray*}
y'(R)\geq \frac{R}{2}y(R),\quad R>0.
\end{eqnarray*}

By Gronwall's inequality, $e^{-R_2^2/4}y(R_2)\geq e^{-R_1^2/4} y(R_1)$ for any $R_2>R_1>0$. If one can prove that $\lim_{R\rightarrow+\infty} e^{-R^2/4}y(R)=0$ then the previous inequality shows that $\nabla^{sph}u=0$ and by applying Pohozaev identity again, one gets that $u\equiv P_0$. 

Therefore, it suffices to prove that $\lim_{R\rightarrow+\infty} e^{-R^2/4}y(R)=0$. Now, observe that:
\begin{eqnarray*}
y(R)&=&\int_{B(0,R)}r^2|\nabla u|^2d\mu_f\leq R^2\int_{B(0,R)}\left(|\nabla u|^2-|\nabla P_0|^2\right)d\mu_f,
\end{eqnarray*}
for every positive radius $R>0$. By Theorem \ref{mono-rel-ent}, $\lim_{R\rightarrow+\infty}\int_{B(0,R)}\left(|\nabla u|^2-|\nabla P_0|^2\right)d\mu_f$ is finite so: $y(R)=\textit{O}(R^2)$. This leads to the result.

\end{proof}

\section{Moduli space of smooth expanding solutions of the harmonic map flow}\label{section-mod-space-smooth-exp}

\subsection{First and second variation of the relative entropy}\label{fct-spa-first-sec-var-rel-ent}
\subsubsection{Function spaces}
Let $n\geq 3$ and let $(N,g)$ be a closed Riemannian manifold isometrically embedded in $\R^m$. We define below the main function spaces we will use:
\begin{enumerate}
\item (Energy space) Define the spaces $H^1_f(\R^n,\R^m)$ and $H^1_f(\R^n,N)$ to be the following sets:
\begin{eqnarray*}
H^1_f(\R^n,\R^m)&:=&\left\{k\in H^1_{loc}(\R^n,\R^m)\quad|\quad k\in L^2(e^fdx),\quad \nabla k \in L^2(e^fdx)\right\},\\
H^1_f(\R^n,N)&:=&\left\{k\in H^1_{f}(\R^n,\R^m)\quad|\quad k(x)\in N,\quad \text{ for a.e. $x\in\R^n$}\right\},
\end{eqnarray*}
endowed with the same following norm: $$\|k\|_{H^1_f}:=\|k\|_{L^2(e^fdx)}+\|\nabla k\|_{L^2(e^fdx)}.$$
The space $\left(H^1_f(\R^n,\R^m), \|\cdot\|_{H^1_f}\right))$ is a Hilbert space.
\item (Schauder spaces)
For a nonnegative integer $k\in\mathbb{N}$ and some real number $\alpha\in[0,1)$, we define the following spaces:
\begin{eqnarray*}
C^{k,\alpha}(\R^n,\R^m)&:=&\{u\in C^{k,\alpha}_{loc}(\R^n,\R^m)\quad|\quad f^{i/2}\nabla^iu\in C^{0,\alpha}(\R^n,\R^m),\quad i=0,..., k\},\\
C^{k,\alpha}(\R^n,N)&:=&\{u\in C^{k,\alpha}(\R^n,\R^m)\quad|\quad \text{$u(x)\in N$ for $x\in \R^n$}\},
\end{eqnarray*}
endowed with the norm $\|u\|_{C^{k,\alpha}}:=\sum_{i=0}^k\|f^{i/2}\nabla^iu\|_{C^{0,\alpha}}.$\\

These spaces reflect the conical structure of the Euclidean space apart form the $k+\alpha$ semi-norm. They will be useful to study the Fredholm properties of the Jacobi operator.

We introduce the corresponding H\" older spaces that are conical up to the $\alpha$ semi-norm. We first define the following $\alpha$ semi-norm for a tensor map:
\begin{eqnarray*}
[u]_{k+\alpha,\R^n}:=\sup_{x\in\R^n}\sup_{y,z\in B(x,\min\{1,|x|/2\})}\min\{f(y),f(z)\}^{\frac{k+\alpha}{2}}\frac{|\nabla^ku(y)-\nabla^ku(z)|}{|y-z|^{\alpha}}.
\end{eqnarray*}
Of course, $[u]_{k+\alpha,\R^n}$ controls the usual $\alpha$ semi-norm on $\R^n$ but the converse is not true.
\begin{eqnarray*}
C^{k,\alpha}_{con}(\R^n,\R^m)&:=&\{u\in C^{k}(\R^n,\R^m)\quad|\quad [u]_{k+\alpha,\R^n}<+\infty\},\\
C^{k,\alpha}_{con}(\R^n,N)&:=&\{u\in C_{con}^{k,\alpha}(\R^n,\R^m)\quad|\quad \text{$u(x)\in N$ for $x\in \R^n$}\},
\end{eqnarray*}
endowed with the norm $\|u\|_{C^{k,\alpha}_{con}}:=\|u\|_{C^k}+[u]_{k+\alpha,\R^n}.$\\

 Actually, we will focus on the subspace of maps in $C_{con}^{k,\alpha}(\R^n,N)$ that admits a radial limit at infinity with the same regularity, i.e. such that if $u\in C_{con}^{k,\alpha}(\R^n,N)$, 
\begin{eqnarray*}
 u_0(\omega):=\lim_{r\rightarrow+\infty} u(r,\omega), \quad\omega\in\mathbb{S}^{n-1},
 \end{eqnarray*}
exists in the $C^{k,\alpha'}(\mathbb{S}^{n-1},N)$ topology, for some (hence for all) $\alpha'\in(0,\alpha)$ which implies that $u_0$ is in $C^{k,\alpha}(\mathbb{S}^{n-1},N)$. This space will be denoted by $C_{con}^{k,\alpha}(\overline{\R^n},N)$.\\

We also define sets of vector fields along a given map $u:\R^n\rightarrow N$:
\begin{eqnarray*}
C^{k,\alpha}(\R^n,T_uN)&:=&\{\kappa\in C^{k,\alpha}(\R^n,\R^m)\quad|\quad \kappa\in T_uN\}.\\
\end{eqnarray*}

We finally define the spaces of boundary data:
\begin{eqnarray*}
C^{k,\alpha}(\Sp^{n-1},N)&:=&\{\psi\in C^{k,\alpha}(\Sp^{n-1},\R^m)\quad|\quad \text{$\psi(x)\in N$ for $x\in \R^n$}\},
\end{eqnarray*}
and if $\psi\in C^{k,\alpha}(\Sp^{n-1},N)$, then we define the corresponding spaces of vector fields along $\psi$:
\begin{eqnarray*}
C^{k,\alpha}(\Sp^{n-1},T_{\psi}N)&:=&\{\kappa\in C^{k,\alpha}(\Sp^{n-1},\R^m)\quad|\quad \text{$\kappa\in T_{\psi}N$} \}.
\end{eqnarray*}

\end{enumerate}

Note that the spaces $C_{con}^{k,\alpha}(\R^n,N)$ (respectively $C^{k,\alpha}(\Sp^{n-1},N)$) are Banach manifolds modeled on their tangent spaces which are $C_{con}^{k,\alpha}(\R^n,T_uN)$ at a point $u\in C_{con}^{k,\alpha}(\R^n,N)$ (respectively $C^{k,\alpha}(\Sp^{n-1},T_{\psi}N)$ at a point $\psi\in C^{k,\alpha}(\Sp^{n-1},N)$.)

We define now the following weighted Schauder spaces:
\begin{eqnarray*}
C_f^{k,\alpha}(\R^n,N)&:=&f^{-1}\cdot C_{con}^{k,\alpha}(\R^n,N),\\
C_f^{k,\alpha}(\R^n,T_uN)&:=&f^{-1}\cdot C_{con}^{k,\alpha}(\R^n,T_uN),\quad u\in C_{con}^{k,\alpha}(\R^n,N),\\
C_{Exp}^{k,\alpha}(\R^n,T_uN)&:=&e^{-f}f^{-\frac{n}{2}}\cdot C^{k,\alpha}(\R^n,T_uN),\quad u\in C_{con}^{k,\alpha}(\R^n,N).
\end{eqnarray*}
\begin{rk}
The reason we introduce these spaces with these special weights comes from the analysis of the kernel of the Jacobi operator together with the convergence rate of an expanding solution to its asymptotic boundary data. 
\end{rk}

\subsubsection{First variational formula of entropy}
Denote by $\pi:T_{\delta}(N)\rightarrow N$ the projection on $N$ where $T_{\delta}(N)$ is a tubular neighborhood with $\delta$ sufficiently small so that $\pi$ is smooth. For a map $u:\R^n\rightarrow N\subset\R^m$, define the projection map $P_u:\R^m\rightarrow T_uN$.

Suppose $u\in C_{con}^{k,\alpha}(\R^n,N)$ with $k\geq 2$ and let $\kappa\in C_0^{\infty}(\R^n,T_uN)$. Define the following curve: $$u_t:=\pi(u+t\kappa):t\in (-\varepsilon,\varepsilon)\rightarrow C_{con}^{k,\alpha}(\R^n,N),$$ where $\varepsilon$ is a positive number sufficiently small compared to $\delta$ and $\|\kappa\|_{C^0}$. 

This curve of maps satisfies: $\frac{d}{dt}|_{t=0}u_t=\kappa\in T_uN$.
Define the relative entropy as follows:
\begin{eqnarray*}
\Ent(u_t,u):=\lim_{R\rightarrow+\infty}\int_{B(0,R)}(|\nabla u_t|^2-|\nabla u|^2)d\mu_f.
\end{eqnarray*}
One easily sees that $\Ent(u_t,u)$ is well-defined since $\kappa$ is compactly supported. We do not specify at the moment the optimal space of deformations $\kappa$ for which $\Ent(u_t,u)$ is well-defined. In case this relative entropy is well-defined, the integral $$\int_{\R^n}(|\nabla u_t|^2-|\nabla u|^2)d\mu_f,$$ is understood in the sense of improper integrals.

Then the first variation formula of the relative entropy is:
\begin{eqnarray*}
\frac{1}{2}\left.\frac{d}{dt}\right|_{t=0}\Ent(u_t,u)&=&\frac{1}{2}\left.\frac{d}{dt}\right|_{t=0}\int_{\R^n}\left(|\nabla u_t|^2-|\nabla u|^2\right)d\mu_f\\
&=&\int_{\R^n}<-E(u),\kappa>d\mu_f,
\end{eqnarray*}
where,
\begin{eqnarray*}
E(u):=P_u(\Delta_fu)=\Delta_fu+A(u)(\nabla u,\nabla u),
\end{eqnarray*}
where $A(u):T_uN\times T_uN\rightarrow (T_uN)^{\perp}$ denotes the second fundamental form of $N$ evaluated at $u$.
\begin{rk}
From an intrinsic viewpoint, $E(u)=\tr(\nabla du)+d_{\nabla f}u=0,$
where $\nabla$ denotes the connection on $T^*\R^n\otimes u^*TN$ induced from  $T^*\R^n$ and $u^*TN$. \end{rk}

\subsubsection{Second variational formula of entropy}

Assume $u$ is a smooth expanding solution of the Harmonic map flow $C_{con}^{k,\alpha}(\overline{\R^n},N)$, $k\geq 2$, $\alpha\in(0,1)$ and let $\kappa$ and $\eta$ be two deformations along $u$ in $C^{\infty}_0(\R^n,T_uN)$  and consider the following two-parameter variation:
\begin{eqnarray*}
u_{s,t}:=\pi(u+s\kappa+t\eta),\quad s,t\in(-\varepsilon,\varepsilon).
\end{eqnarray*}
Then the second variation of the relative entropy is equal to:
\begin{eqnarray*}
\frac{1}{2}\left.\frac{\partial^2}{\partial s\partial t}\right|_{s,t=0}\Ent(u_{s,t},u)=-\int_{\R^n}<L_u(\kappa),\eta>d\mu_f,
\end{eqnarray*}
where $L_u:C^{\infty}_0(\R^n,T_uN)\rightarrow C^{\infty}_0(\R^n,T_uN)$ is the Jacobi operator with respect to $u$ defined by:
\begin{eqnarray*}
L_u(\kappa):=P_u(D_uE(\kappa))=\Delta_f\kappa+D_u(A)(\kappa)(\nabla u,\nabla u)+2A(u)(\nabla u,\nabla \kappa).
\end{eqnarray*}
Since $A(u)(\nabla u,\nabla \kappa)\in (T_uN)^{\perp}$, one gets:
\begin{eqnarray*}
\frac{1}{2}\left.\frac{\partial^2}{\partial s\partial t}\right|_{s,t=0}\Ent(u_{s,t},u)&=&-\int_{\R^n}<\Delta_f\kappa+D_u(A)(\kappa)(\nabla u,\nabla u),\eta>d\mu_f,
\end{eqnarray*}
which shows that $L_u$ is a symmetric operator. One can also get a more intrinsic formula for the second variation:
\begin{eqnarray*}
\frac{1}{2}\left.\frac{\partial^2}{\partial s\partial t}\right|_{s,t=0}\Ent(u_{s,t},u)&=&-\int_{\R^n}<\Delta_f\kappa+\tr\left<\Rm(h)(du,\kappa)\eta,du\right>d\mu_f,\quad \kappa,\eta\in C^{\infty}_0(\R^n,T_uN),
\end{eqnarray*}
where $(N,h)$ is isometrically embedded in $\R^m$. The curvature term is defined as follows:
\begin{eqnarray*}
\tr\left<\Rm(h)(du,\kappa)\eta,du\right>:=\tr\left((v,w)\rightarrow\left<\Rm(h)(\nabla_vu,\kappa)\eta,\nabla_wu\right>_{u^*h}\right).
\end{eqnarray*}

In order to understand its spectral properties, we introduce the following space:
\begin{eqnarray*}
D(L_u):=\{\kappa\in H^1_f(\R^n,T_uN)\quad|\quad L_u(\kappa)\in L^2_f(\R^n,T_uN)\}.
\end{eqnarray*}

Since $u\in C_{con}^{k,\alpha}(\overline{\R^n},N)$, one can show that
\begin{eqnarray*}
D(L_u)=\{\kappa\in H^1_f(\R^n,T_uN)\quad|\quad \Delta_f\kappa\in L^2_f(\R^n,\R^m)\}.
\end{eqnarray*}

The spectral properties of the Jacobi operator $L_u$ are summarized below:

\begin{prop}\label{emp-spec-egs}
Let $u \in C_{con}^{k,\alpha}(\overline{\R^n},N)$ be an expanding solution of the harmonic map flow. 
Then the operator $L_u|_{C_0^{\infty}(\R^n,T_uN)}$ admits a unique self-adjoint extension to $D(L_u)$ whose domain is contained in $H^{1}_f(\R^n,T_uN)$. Moreover, the essential spectrum of $L_u$ is empty:
$\sigma_{ess}(L_u)=\emptyset$.
\end{prop}

\begin{proof}
Let $\kappa\in C_0^{\infty}(\R^n,T_uN)$. Then, 
\begin{eqnarray*}
e^{f/2}L_u(e^{-f/2}\kappa)&=&e^{f/2}\left((\Delta_f(e^{-f/2}\kappa)+2A(u)(\nabla u,\nabla(e^{-f/2}\kappa))+D_uA(e^{-f/2}\kappa)(\nabla u,\nabla u)\right)\\
&=&\Delta\kappa+e^{f/2}\Delta_f(e^{-f/2})\kappa+2e^{f/2}A(u)(\nabla u,\nabla(e^{-f/2}\kappa))+D_uA(\kappa)(\nabla u,\nabla u)\\
&=&P_u(\Delta\kappa)-\frac{1}{2}\left(\Delta f+\frac{\arrowvert\nabla f\arrowvert^2}{2}\right)\kappa-A(u)(\nabla_{\nabla f} u,\kappa)\\
&=:&P_u(\Delta\kappa)-V\kappa,
\end{eqnarray*}
where $$V\kappa:=\frac{1}{2}\left(\Delta f+\frac{\arrowvert\nabla f\arrowvert^2}{2}\right)\kappa+A(u)(\nabla_{\nabla f} u,\kappa).$$
Observe that $A(u)(\nabla_{\nabla f} u,\kappa)\perp T_uN$.

Now, define 
\begin{eqnarray*}
&&\bar{H}_{f}^1(\R^n,T_uN):=\{\kappa\in H^1(\R^n,T_uN),\quad f^{1/2}\cdot \kappa\in L^2(\R^n,T_uN)\},
\end{eqnarray*}
endowed with the norm $$\|\kappa\|_{\bar{H}_{f}^1}:=\|\kappa\|_{H^1}+\|f^{1/2}\cdot\kappa\|_{L^2}.$$
\begin{eqnarray*}
\bar{H}_f^2(\R^n,T_uN):=\{\kappa\in H^2(\R^n,T_uN),\quad
f\cdot\kappa\in L^2(\R^n,T_uN) \}
\end{eqnarray*}
endowed with the norm $$\|\kappa\|_{\bar{H}_f^2}:=\|\kappa\|_{H^2}+\|f\cdot\kappa\|_{L^2}.$$
Invoking [Chap.$8$, Sec. $2$, \cite{Tay-Boo-II}], the operator $-P_u(\Delta)+V\cdot|_{C_0^{\infty}(\R^n,T_uN)}$ admits a unique self-adjoint extension to $\bar{H}_f^2(\R^n,T_uN)$ whose domain is contained in $\bar{H}_f^1(\R^n,T_uN)$. Moreover, since $V$ is proper in the sense of quadratic forms, by a straightforward adaptation of Proposition $2.8$ of [Chap.$8$, Sec. $2$, \cite{Tay-Boo-II}], the operator $-P_u(\Delta)+V$ has compact resolvent. \\

Indeed, it suffices to prove that the injection $\bar{H}_f^1(\R^n,T_uN)\hookrightarrow L^2(\R^n,T_uN)$ is compact. If $(\kappa_i)_i$ is a bounded sequence in $\bar{H}_f^1(\R^n,T_uN)$, then by Rellich's theorem and a diagonal argument, there exists a subsequence still denoted by $(\kappa_i)_i$ that converges in $L^2(B(0,l),T_uN)$-norm where $(B(0,l))_l$ is an exhaustion of $\R^n$. Now, let $\eta$ be a positive number and fix an index $l$ sufficiently large such that, for any $i,i'\geq i(l,\eta)$,
\begin{eqnarray*}
&&\|\kappa_i-\kappa_{i'}\|_{L^2(B(0,l),T_uN)}\leq \eta,\quad\inf_{\R^n\setminus B(0,l)}\arrowvert\nabla f\arrowvert\geq 1/\eta,
\end{eqnarray*}
since $|\nabla f|$ is proper.
\begin{eqnarray*}
\|\kappa_i-\kappa_{i'}\|_{L^2(\R^n,T_uN)}&\leq& \|\kappa_i-\kappa_{i'}\|_{L^2(B(0,l),T_uN)}\\
&&+\frac{1}{\inf_{\R^n\setminus B(0,l)}\arrowvert\nabla f\arrowvert}\cdot\left(\int_{\R^n\setminus B(0,l)}\arrowvert\nabla f\arrowvert^2\arrowvert \kappa_i-\kappa_{i'}\arrowvert^2dx\right)^{1/2}\\
&\leq&\eta+\eta\cdot(\|\kappa_i\|_{\bar{H}_f^1(\R^n,T_uN)}+\|\kappa_{i'}\|_{\bar{H}_f^1(\R^n,T_uN)})\\
&\leq& C\eta,
\end{eqnarray*}
where $C$ does not depend on $l$ since $(\kappa_i)_i$ is bounded in $\bar{H}_f^1(\R^n,T_uN)$.

Finally, as $L_u$ is unitarily conjugate to $P_u(\Delta)-V$, it follows that the essential spectrum of $L_u$ is empty.

\end{proof}

As an illustration, we consider the case where $N=\Sp^{m-1}$ is a Euclidean sphere of radius $1$. Then, if $u\in C_{con}^{k,\alpha}(\overline{\R^n},N)$ is a smooth expanding solution to the Harmonic map flow, the Jacobi operator with respect to $u$ is:
\begin{eqnarray*}
L_u(\kappa)=\Delta_f\kappa+2<\nabla u,\nabla \kappa>u+|\nabla u|^2\kappa,\quad \kappa\in C_{con}^{k,\alpha}(\R^n,T_uN).
\end{eqnarray*}

A vector field $\kappa\in H^1_{loc}(\R^n,T_uN)$ that satisfies $L_u\kappa=0$ is called a \textbf{Jacobi field} along $u$. Denote this space by $\ker L_u$. Since $u$ is regular, elliptic regularity shows that $\kappa$ is smooth on $\R^n$. It turns out that there are roughly two kinds of behavior at infinity for such Jacobi fields: bounded ones and those who converge to $0$ at infinity.

 Now, define the subspace of Jacobi fields along $u$ that vanishes at infinity:
\begin{eqnarray*}
\ker_0L_u:=\ker L_u\cap \left\{\kappa\in C^{k,\alpha}(\R^n,T_uN)\quad|\quad \lim_{+\infty}\kappa=0\right\}.
\end{eqnarray*}

The purpose of the next theorem is to analyse the convergence rate of such Jacobi fields.

\begin{theo}\label{Analysis-Jacobi-field}
Let $\kappa$ be a Jacobi field along an expanding map $u\in C_{con}^{k,\alpha}(\overline{\R^n},N)$ that vanishes at infinity, i.e. $\kappa\in \ker_0L_u$. Then, 
$\kappa \in \cap_{k\geq 0}C^{k,\alpha}(\R^n,T_uN)$ and the following radial limit
\begin{eqnarray*}
\kappa_{\infty}(\omega):=\lim_{r\rightarrow+\infty}\left(e^f f^{\frac{n}{2}}\kappa\right)(r,\omega),\quad\omega\in \Sp^{n-1},
\end{eqnarray*}
exists and defines a vector field $\kappa_{\infty}\in C^{k,\alpha}(\Sp^{n-1},T_uN)$ along $u$. In particular, 
\begin{eqnarray}
\ker_0L_u\subset\ker L_u\cap L^2_f(\R^n,T_uN),\label{rk-incl}
\end{eqnarray}
 and has finite dimension.

Finally, one has the following unique continuation result at infinity: if $\kappa_{\infty}\equiv 0$ then $\kappa\equiv0$. 
\end{theo}

\begin{rk}\label{rk-bdy-map-jac-field}
Theorem \ref{Analysis-Jacobi-field} shows in particular that the map 
\begin{eqnarray}
\kappa\in \ker_0 L_u\rightarrow \{\kappa_{\infty}\in C^{k,\alpha}(\Sp^{n-1},T_uN)\}=:B_u^{k,\alpha},
\end{eqnarray}
where $B_u^{k,\alpha}$ stands for boundary maps at infinity,
 is an isomorphism of finite dimensional vector spaces. 

One can actually show that the inclusion (\ref{rk-incl}) is an equality: $\ker_0L_u=\ker L_u\cap L^2_f(\R^n,T_uN).$
\end{rk}

\begin{proof}
Let $\kappa\in\ker_0L_u$, that is:
\begin{eqnarray*}
L_u(\kappa)&=&\Delta_f\kappa+D_uA(\kappa)(\nabla u,\nabla u)+2A(u)(\nabla u,\nabla \kappa)\\
&=&0.
\end{eqnarray*}
Since $A(u)(\nabla u,\nabla \kappa)\in (T_uN)^{\perp}$, $|\kappa|^2$ satisfies the following differential inequality:
\begin{eqnarray*}
\Delta_f|\kappa|^2\geq2|\nabla\kappa|^2-C(N)|\nabla u|^2|\kappa|^2.
\end{eqnarray*}
In particular, a regularization of $|\kappa|$ of the form $\kappa_{\varepsilon}:=\sqrt{|\kappa|^2+\varepsilon^2}$ with $\varepsilon\in(0,1)$ satisfies:
\begin{eqnarray*}
\Delta_f\kappa_{\varepsilon}\geq-\frac{C(N)}{f}\kappa_{\varepsilon}.
\end{eqnarray*}
According to [Section $2.3$, \cite{Der-Asy-Com-Egs}], there exist positive constants $C$ and $A_0$ such that $e^{-Cf^{-1}}\kappa_{\varepsilon}-Af^{-n/2}e^{-f}$ satisfies the maximum principle for any $A\geq A_0$  outside a sufficiently large ball $B(0,R_0)$ independent of $\varepsilon$:
\begin{eqnarray*}
\max_{B(0,R)\setminus B(0,R_0)}\left(e^{-Cf^{-1}}\kappa_{\varepsilon}-Af^{-n/2}e^{-f}\right)=\max_{\partial B(0,R)\cup \partial B(0,R_0)}\left(e^{-Cf^{-1}}\kappa_{\varepsilon}-Af^{-n/2}e^{-f}\right).
\end{eqnarray*}
By letting $\varepsilon$ go to $0$ and by using the fact that $\lim_{+\infty}\kappa=0$, one gets:
\begin{eqnarray*}
\max_{\R^n\setminus B(0,R_0)}\left(e^{-Cf^{-1}}|\kappa|-Af^{-n/2}e^{-f}\right)=\max_{\partial B(0,R_0)}\left(e^{-Cf^{-1}}|\kappa|-Af^{-n/2}e^{-f}\right)\leq 0,
\end{eqnarray*}
if $A$ is chosen large enough compared to $\max_{\partial B(0,R_0)}|\kappa|$. This implies already a very fast decay for $\kappa$. By applying Schauder estimates to $\kappa$, its derivatives decay exponentially at infinity so in particular, they decay faster than any polynomial: this ensures that $\kappa \in \cap_{k\geq 0}C^{k,\alpha}(\R^n,N)$.

Now, we need to show that the rescaled vector field $e^ff^{n/2}\kappa$ along $u$ converges to a vector field $\kappa_{\infty}$ defined on the sphere (at infinity) $\Sp^{n-1}$ in the $C^{k,\alpha}(\Sp^{n-1},T_uN)$ topology. For this purpose, we first compute the evolution equation satisfied by the vector field $e^ff^{n/2}\kappa=:\kappa_f$ along $u$:
\begin{eqnarray*}
\Delta_{-f}\kappa_f&=&V_1(u)\ast\kappa_f+V_2(u)\ast\nabla \kappa_f,\\
V_1(u)&=&\nabla u\ast\nabla u+\textit{O}(f^{-1})\in C^{k-1,\alpha}_{con}(\R^n,\R^m),\\
V_2(u)\ast\nabla \kappa_f&=&\textit{O}(f^{-1/2})\ast\nabla \kappa_f+e^ff^{n/2}A(u)(\nabla u,\nabla \kappa)\\
&=&\textit{O}(f^{-1/2})\ast\nabla \kappa_f-(1+\textit{O}(f^{-1}))A(u)(\nabla_{\nabla f} u,\kappa_f)\\
&=&\textit{O}(f^{-1/2})\ast\nabla \kappa_f+\textit{O}(f^{-1})\ast\kappa_f,
\end{eqnarray*}
where we used the fact that $u$ is an expanding solution in the last line.

Notice that the terms $\textit{O}(f^{-1})$ (resp. $\textit{O}(f^{-1/2})$) are in $C_{f}^{k-1,\alpha}(\R^n,(\R^m)^*\otimes\R^m)$, (resp. $C_{f^{1/2}}^{k-1,\alpha}(\R^n,(\nabla \R^m)^*\otimes\R^m)$). 



To sum it up, we have:
\begin{eqnarray}
&&\Delta_{-f}\kappa_f=V_1(u)\ast\kappa_f+V_2(u)\ast\nabla \kappa_f,\label{back-heat-eq-resc-jac-fiel}\\
&&V_1(u)\in C_{f}^{k-2,\alpha}(\R^n,(\R^m)^*\otimes\R^m),\quad V_2(u)\in C_{f^{1/2}}^{k-1,\alpha}(\R^n,(\nabla \R^m)^*\otimes\R^m)).\label{pot-bac-heat-eq}
\end{eqnarray}
The evolution equation (\ref{back-heat-eq-resc-jac-fiel}) can be reinterpreted as a linear backward heat equation with data given by the right-hand side of (\ref{back-heat-eq-resc-jac-fiel}) with some amount of regularity at infinity described by (\ref{pot-bac-heat-eq}). Therefore, standard parabolic Schauder estimates in their local version in the case of an ancient solution to the heat equation give:
 \begin{eqnarray}
\|\kappa_f\|_{C^{k,\alpha}_{con}(\R^n,\R^m)}\leq C(k,\alpha,n,m)\|\kappa_f\|_{C^0(\R^n,T_uN)}.\label{sch-est-bac-heat-eq}
\end{eqnarray}
Finally, to prove that $\kappa_f\in C^{k,\alpha}_{con}(\overline{\R^n},T_uN)$, it suffices to prove that $\kappa_f$ has a radial limit in the $C^{k,\alpha'}(\Sp^{n-1},N)$ topology, $\alpha'\in(0,\alpha)$, that lies in $C^{k,\alpha}(\Sp^{n-1},N).$ Since we assume $k\geq 2$, the estimates (\ref{sch-est-bac-heat-eq}) together with (\ref{back-heat-eq-resc-jac-fiel}) show that the radial derivative decay much faster than expected: $\partial_r\kappa_f=\textit{O}(r^{-3})$ which implies that $\kappa_f(r,\cdot)$ converges radially in the $C^0$ topology to a vector field $\kappa_{\infty}\in C^0(\Sp^{n-1},T_{u_0N})$. Now, recall that the inclusion $C^{k,\alpha}(\Sp^{n-1},T_uN)\hookrightarrow C^{k,\alpha'}(\Sp^{n-1},T_uN)$ is compact for all $\alpha'\in(0,\alpha)$. This implies that $\kappa_f$ subconverges radially to a vector field $\kappa'_{\infty}\in C^{k,\alpha}(\Sp^{n-1},T_uN)$ in the $C^{k,\alpha'}(\Sp^{n-1},T_uN)$ topology. By uniqueness of the limit, $\kappa_{\infty}=\kappa_{\infty}'\in C^{k,\alpha}(\Sp^{n-1},T_uN)$ and the convergence holds in the expected topology.

We postpone the proof of the unique continuation property at infinity for such Jacobi fields to the proof of Theorem \ref{uni-cont-inf}.

\end{proof}

\subsection{Fredholm properties of the Jacobi operator}\label{Fred-prop-Jac-op-subsec}
Let $u\in C_{con}^{k_0,\alpha_0}(\overline{\R^n},N)$ be an expanding map for some integer $k_0\geq 2$ and some real number $\alpha_0\in(0,1)$.

In this section, we establish the required properties on the Jacobi operator associated to $u$ between Schauder spaces. Because the natural deformation space consisting of quadratically decaying vector fields along $u$ are not contained in any energy space endowed with the weighted measure $e^fdx$, we cannot rely on Proposition \ref{emp-spec-egs}. For this purpose, we need to introduce the following function spaces:
\begin{eqnarray*}
\mathcal{D}_f^{k+2,\alpha}(\R^n,\R^m)&:=&\{\kappa\in C^{k+2,\alpha}_{loc}(\R^n,\R^m)\quad|\quad\kappa\in C_f^{k,\alpha}(\R^n,\R^m)\quad|\quad \Delta_f \kappa\in C_f^{k,\alpha}(\R^n,\R^m)\},\\
\mathcal{D}_f^{k+2,\alpha}(\R^n,T_uN)&:=&\{\kappa\in \mathcal{D}_f^{k+2,\alpha}(\R^n,\R^m)\quad|\quad \text{$\kappa\in T_uN$}\},\\
\mathcal{C}_f^{k+2,\alpha}(\R^n,T_uN)&:=&\big\{\kappa\in C^{k+2,\alpha}_{loc}(\R^n,T_uN)\cap C^{k,\alpha}_f(\R^n,T_uN),\\
&& f^{k/2}\nabla^k(f\cdot\kappa)\in C^{2}(\R^n,\R^m),\quad \left[\nabla^2\left(f^{k/2}\nabla^k(f\cdot\kappa)\right)\right]_{\alpha,\R^n}<+\infty\big\}.
\end{eqnarray*}
The space $\mathcal{D}_f^{k+2,\alpha}(\R^n,T_uN)$ is a Banach space but it is not straightforward: this is a by-product of the proof of the next theorem that establishes the Fredholm properties of the Jacobi operator.
\begin{theo}\label{theo-fred-prop-jac-op}
Let $k\geq 2$ and $\alpha\in(0,1)$. Let $u\in C_{con}^{k,\alpha}(\overline{\R^n},N)$ be an expanding solution of the Harmonic map flow. Then the projection of the weighted laplacian, 
\begin{align*}
\tilde{L}_u:\mathcal{D}_f^{k+2,\alpha}(\R^n,T_uN) &\rightarrow C_f^{k,\alpha}(\R^n,T_uN)\\
\kappa&\mapsto \Delta_f\kappa+2A(u)(\nabla u,\nabla\kappa)+(D_uA(\kappa)(\nabla u,\nabla u))^{\perp}
\end{align*}
 is an isomorphism of Banach spaces. As a consequence, the Jacobi operator associated to $u$, $$L_u: \mathcal{D}_f^{k+2,\alpha}(\R^n,T_uN)\rightarrow C_f^{k,\alpha}(\R^n,T_uN),$$ is a Fredholm operator of index $0$. Moreover, the space $\mathcal{C}^{k+2,\alpha}_f(\R^n,T_uN)$ continuously injects in $\mathcal{D}_f^{k+2,\alpha}(\R^n,T_uN)$.
\end{theo}

\begin{proof}
Once the statement about $\tilde{L}_u$ is established, the assertion about the Jacobi operator is proved as follows: observe that $L_u=\tilde{L}_u+\left(D_uA(\cdot)(\nabla u,\nabla u)\right)^{\top}$.
And the operator $$\kappa\in\mathcal{D}_f^{k+2,\alpha}(\R^n,T_uN)\rightarrow \left(D_uA(\kappa)(\nabla u,\nabla u)\right)^{\top}\in C_f^{k,\alpha}(\R^n,T_uN),$$ is a well-defined continuous map since $u\in C_{con}^{k,\alpha}(\overline{\R^n},N)$ with $k\geq 2$. Moreover, it can be shown that it is a compact operator since 
\begin{eqnarray*}
|D_uA(\kappa)(\nabla u,\nabla u)|\leq C(N)|\nabla u|^2|\kappa|\leq \frac{C(N)}{|x|^2}|\kappa|,\quad |x|\geq 1.
\end{eqnarray*}
Therefore, by Fredholm theory, $L_u$ is a Fredholm operator of degree equal to the degree of $\tilde{L}_u$, i.e. $0$.

We first prove the surjectivity of $\tilde{L}_u$. Let $Q\in C_f^{k,\alpha}(\R^n,T_uN)$ for some $k\in \mathbb{N}$ and $\alpha\in(0,1)$. Since $T_uN$ is a bundle over $\R^n$, standard elliptic theory ensures the existence of a solution $\kappa_R\in C^{2,\alpha}(B(0,R),T_uN)$ for each positive radius $R$ to the following Dirichlet problem:
\begin{eqnarray*}
&&\tilde{L}_u\kappa_R=Q,\quad \text{on $B(0,R)$},\\
&&\kappa|_{\partial B(0,R)}=0.
\end{eqnarray*}
By the very definition of $\tilde{L}_u$, 
\begin{eqnarray*}
<\tilde{L}_u(\kappa_R),\kappa_R>&=&<\Delta_f\kappa_R,\kappa_R>\\
&=&\frac{1}{2}\left(\Delta_f|\kappa_R|^2-2|\nabla \kappa_R|^2\right).
\end{eqnarray*}
The issue here to let $R$ go to $+\infty$ is that classical elliptic Schauder estimates might depend on the radii $R$. Our goal is to prove some a priori estimates independent of $R$. By the previous observation, one gets:
\begin{eqnarray*}
\Delta_f|\kappa_R|_{\varepsilon}\geq -|Q|,\quad \text{on $B(0,R)$},
\end{eqnarray*}
where $|\kappa_R|_{\varepsilon}:=\sqrt{|\kappa_R|^2+\varepsilon^2}$ for $\varepsilon\in(0,1)$. Now, we observe that $f$ is a nice barrier function since:
\begin{eqnarray*}
\Delta_ff^{-1}&=&-\frac{\Delta_ff}{f^2}+2\frac{|\nabla f|^2}{f^3}\leq-f^{-1}\left(1-2\frac{|\nabla f|^2}{f^2}\right),
\end{eqnarray*}
and one can check that $a:=\inf_{\R^n}\left(1-2\frac{|\nabla f|^2}{f^2}\right)>0.$ In particular, if $C$ is a positive constant, this implies that:
\begin{eqnarray*}
\Delta_f\left(|\kappa_R|_{\varepsilon}-Cf^{-1}\right)\geq aCf^{-1} -|Q|\geq 0,
\end{eqnarray*}
if $C= a^{-1}\|Q\|_{C^{0,\alpha}_f}.$  The maximum principle then shows that 
\begin{eqnarray*}
\sup_{B(0,R)}\left(|\kappa_R|_{\varepsilon}-Cf^{-1}\right)\leq \sup_{\partial B(0,R)}\left(|\kappa_R|_{\varepsilon}-Cf^{-1}\right)\leq 0.
\end{eqnarray*}
Therefore, since $C$ can be chosen independently of $\varepsilon\in(0,1)$, we can let $\varepsilon$ go to $0$ and we finally get the first a priori $C^0$ weighted estimate:
\begin{eqnarray*}
\|f\kappa_R\|_{C^0(B(0,R))}\leq a^{-1}\|Q\|_{C^{0,\alpha}_f},\quad R>0.
\end{eqnarray*}
By standard elliptic Schauder estimates together with Arzela-Ascoli's theorem, there is a subsequence still denoted by $(\kappa_R)_R$ that converges to $\kappa\in C^{2,\alpha}_{loc}(\R^n,T_uN)$ in the $C^2_{loc}(\R^n,T_uN)$ topology. This vector field $\kappa$ satisfies:
\begin{eqnarray}
&&\tilde{L}_u(\kappa)=Q,\quad \text{on $\R^n$},\quad \|f\kappa\|_{C^0}\leq a^{-1}\|Q\|_{C^{0}_f}.\label{a-priori-est-sol}
\end{eqnarray}
We now re-interprete the elliptic equation $\tilde{L}_u(\kappa)=Q$ as a parabolic one by noticing that if we define $\bar{\kappa}(x,t):=\kappa(x/\sqrt{t})$ for $(x,t)\in \R^n\times \R_+$, 
\begin{eqnarray*}
&&(\partial_t-\Delta )\bar{\kappa}=-2A(\bar{u})(\nabla \bar{u},\nabla\bar{\kappa})-\left(D_{\bar{u}}A(\bar{\kappa})(\nabla \bar{u},\nabla \bar{u})\right)^{\perp}-t^{-1}\bar{Q},\\
&&\bar{u}(x,t):=u(x/\sqrt{t}).
\end{eqnarray*}
Therefore, standard local parabolic Schauder estimates imply that:
\begin{eqnarray}
\sup_{x\in\R^n}f(x)\|\kappa\|_{C^{2,\alpha}(B(x,1))}\leq C\left( \|f\kappa\|_{C^0}+\|Q\|_{C^{0,\alpha}_f}\right)\leq C(1+a^{-1}) \|Q\|_{C^{0,\alpha}_f},\label{final-est}
\end{eqnarray}
for some positive constant $C$ depending on $\alpha$ and where we used (\ref{a-priori-est-sol}) in the last inequality. Notice that the semi-norm $[f\nabla^2\kappa]_{\alpha,\R^n}$ has not yet been estimated.

The last estimates that remain to be shown concern the weighted conical H\" older semi-norms $[f\kappa]_{\alpha,\R^n}$ and $[f\Delta_f\kappa]_{\alpha,\R^n}$. The bound on $[f\Delta_f\kappa]_{\alpha,\R^n}$ follows from the one on $[f\kappa]_{\alpha,\R^n}$. The bound on $[f\nabla^2\kappa]_{\alpha,\R^n}$ follows from the bounds on $[f\kappa]_{\alpha,\R^n}$ and $[f\Delta_f\kappa]_{\alpha,\R^n}$ by classical parabolic Schauder estimates.

  To prove such a bound, it is sufficient by interpolation theory to prove that if $Q\in C^1_{con}(\R^n,T_uN)$ then $\kappa$ satisfies:
\begin{eqnarray*}
\|f^{1/2}\nabla \kappa\|_{C^0}\leq C\|Q\|_{C^1_{con}},
\end{eqnarray*}
for some uniform positive constant $C$. Such an estimate can be derived in the same way we proceeded for the $C^0$ bound (\ref{a-priori-est-sol}) by deriving the equation $\tilde{L}_u(\kappa)=Q$.

Therefore both the surjectivity and the injectivity of $\tilde{L}_u$ have been established. It also shows that $\mathcal{D}_f^{2,\alpha}(\R^n,T_uN)$ inherits a Banach structure and as such, $\tilde{L}_u$ becomes an isomorphism of Banach spaces for $k=0$. The cases $k\geq 1$ follow analogously. 
They are proved along the same lines with the help of the maximum principle.
\end{proof}
\subsection{Manifold structure}
From now on, we consider the moduli space of \textit{smooth} expanding solutions of the harmonic map flow, that is the set of weak solutions $u\in H^1_{loc}(\R^n,N)$ to $H_f(u)=0$ that are smooth and have some regularity at infinity in the sense that $u\in C_{con}^{k,\alpha}(\overline{\R^n},N)$ for some $k\geq 2$, and some $\alpha\in(0,1)$. We denote such a set by:
\begin{eqnarray*}
\Ent_{xp}^{k,\alpha}(N):=\{u\in  C_{con}^{k,\alpha}(\overline{\R^n},N)\quad|\quad \Delta_fu+A(u)(\nabla u,\nabla u)=0\}. 
\end{eqnarray*}
Notice that the definition of $\Ent_{xp}^{k,\alpha}(N)$ asks for too much regularity: indeed, by the proof of Theorem \ref{Analysis-Jacobi-field}, if an expanding map $u\in  C_{con}^{k,\alpha}(\R^n,N)$ for $k\geq 2$ then it admits a radial limit $u_0\in C^{k,\alpha}(\Sp^{n-1},N)$ at infinity.


In this section, we prove that the moduli space of harmonic maps $\mathcal{E}_{xp}^{k,\alpha}(N)$ with some regularity at infinity is a Banach manifold locally modeled on $C^{k,\alpha}(\mathbb{S}^{n-1},N)\times \ker_0L_u$ if $u\in \mathcal{E}_{xp}^{k,\alpha}(N)$. We follow the presentation of \cite{Har-Mou} in the case of harmonic maps on a domain of $\R^n$ very closely. This approach is in turn due to \cite{White-var-met} in the case of minimal surfaces.

The main tools are a delicate integration by parts together with the unique continuation result from Theorem \ref{Analysis-Jacobi-field}. We start by analyzing  the local structure of $\Ent_{xp}^{k,\alpha}(N)$. Before going further, we need to introduce a bit of notation. If $u\in\Ent_{xp}^{k,\alpha}(N)$, $k\geq 2$ and $\alpha\in(0,1)$, we denote by $K_u$ a complementary space of $\ker_0L_u$ in $\mathcal{D}_f^{k,\alpha}(\R^n,T_uN)$: $K_u$ exists since $\ker_0L_u$ is finite dimensional. Similarly, let $I_u$ be a complementary space of $\Ima L_u$ in $C_f^{k-2,\alpha}(\R^n,T_uN)$. Finally, we define a first approximation map $T_u$ as follows: the map $$\Phi_u:h\in T_uN\rightarrow \pi_N(u+h)\in N,$$ is a local diffeomorphism around $0\in T_uN$. Now, consider the map
$$\begin{array}{ccccc}
T_u & : & B(u_0,\varepsilon_0)\subset C^{k,\alpha}(\mathbb{S}^{n-1},N) & \mapsto & C_{con}^{k,\alpha}(\overline{\mathbb{R}^n},T_uN) \\
 & & v_0 & \mapsto & \Phi^{-1}_{u}(\pi_N(\eta(v_0-u_0)+u)), \\
\end{array}$$
where $\eta :\R^n\rightarrow [0,1]$ is a smooth function such that $\eta\equiv 0$ on $B(0,1)\subset \R^n$ and $\eta\equiv 1$ outside $B(0,2)$ for some sufficiently small $\varepsilon_0>0$. The map $T_u$ is well-defined and smooth. 


\begin{theo}[Local structure]\label{loc-str}
Let $k\geq 2$ and $\alpha\in(0,1)$. Let $u\in\Ent_{xp}^{k,\alpha}(N)$. Then there exist a neighborhood $(u_0,0)\in U\subset C^{k,\alpha}(\mathbb{S}^{n-1},N)\times \ker_0L_u$ together with smooth maps 
\begin{eqnarray*}
&&\iota_u:U\rightarrow C_{con}^{k,\alpha}(\overline{\R^n},N),\\
&&s_u:U\rightarrow I_u\subset C_f^{k-2,\alpha}(\R^n,T_uN),
\end{eqnarray*}
 satisfying:
 \begin{enumerate}
 \item \label{item-str-inf-iota}$\iota_u(u_0,0)=u$ and $(\iota_u(v_0,\kappa))_0=v_0$, for every $(v_0,\kappa)\in U$.\\
 \item $\iota_u(v_0,\kappa)$ is in $\Ent_{xp}^{k,\alpha}(N)$ if and only if $s_u(v_0,\kappa)=0$.\\
 \item \label{item-inj} $D\iota_u(u_0,0):C^{k,\alpha}(\mathbb{S}^{n-1},T_{u_0}N)\times\ker_0L_u\rightarrow C^{k,\alpha}_{con}(\overline{\mathbb{R}^n},T_uN)$ is injective. \\ 
 \item \label{item-surj}
 $Ds_u(u_0,0):C^{k,\alpha}(\mathbb{S}^{n-1},T_{u_0}N)\times\ker_0L_u\rightarrow I_u$ is surjective.\\
 \item \label{item-unique-exp} For every sufficiently small neighborhood $B(0,\varepsilon)$ of $0$ in $\ker_0L_u$, there is a neighborhood $W$ of $u\in C^{k,\alpha}_{con}(\overline{\mathbb{R}^n},N)$ such that any expanding map $v$ in $W$ equals $\iota_u(v_0,\kappa)$ for some $\kappa$ in $B(0,\varepsilon)$.\\
 
 \item \label{item-chart}The triplet $(\iota_u,U\cap s_u^{-1}(0),W\cap \Ent_{xp}^{k,\alpha}(N))$ is a chart for $\Ent_{xp}^{k,\alpha}(N)$, i.e.\\
 $\iota_u:U\cap s_u^{-1}(0)\rightarrow W\cap \Ent_{xp}^{k,\alpha}(N)$ is a diffeomorphism and $\Pi\circ \iota_u(v_0,\kappa)=v_0$ for $(v_0,\kappa)\in U\cap s_u^{-1}(0)$. Moreover, $U\cap s_u^{-1}(0)$ has codimension the nullity of the Jacobi operator, $\dim \ker_0L_u$ and its tangent space at $(u_0,0)$ is:
 \begin{eqnarray}
 T_{(u_0,0)}\left(U\cap s_u^{-1}(0)\right)=\ker D_1s_u(u_0,0)\oplus\ker_0L_u.\label{split-tan-space}
 \end{eqnarray}
 \\
 
 \item \label{item-second-count}Finally, one can choose $U$ (respectively $W$) to be open for the $C^{k,\alpha'}(\mathbb{S}^{n-1},N)$  topology (respectively for the $C_{con}^{k,\alpha'}(\overline{\mathbb{R}^{n}},N)$  topology) if $\alpha'\in(0,\alpha)$.

 \end{enumerate}

\end{theo}

\begin{proof}

Define a (non-linear) map $N:C^{k,\alpha}(\mathbb{S}^{n-1},N)\times\ker_0L_u\times K_u\rightarrow \Ima L_u\subset C_f^{k-2,\alpha}(\R^n,T_uN)$ on a neighborhood of $(u_0,0,0)\in C^{k,\alpha}(\mathbb{S}^{n-1},N)\times\ker_0L_u\times K_u$ by

\begin{eqnarray*}
N(v_0,\kappa,\eta):=P_{\Ima L_u}\circ P_u\circ H_f(\pi_N(u+T_u(v_0)+\kappa+\eta)).
\end{eqnarray*}
$N$ is well-defined and is a smooth map.

Then one computes $D_3N(u_0,0,0)(\eta)=L_u(\eta)$ for $\eta\in K_u$. By the definition of $K_u$, $D_3N(u_0,0,0):K_u\rightarrow \Ima L_u$ is an isomorphism of Banach spaces. The implicit function theorem ensures the existence of a neighborhood $U$ of $(u_0,0)\in C^{k,\alpha}(\mathbb{S}^{n-1},N)\times\ker_0L_u$ and a neighborhood $V$ of $0\in K_u$ together with a smooth map $j: U\subset C^{k,\alpha}(\mathbb{S}^{n-1},N)\times\ker_0L_u\rightarrow V\subset K_u$ such that:
\begin{eqnarray}
\text{$N(v_0,\kappa,\eta)=0$, for $(v_0,\kappa,\eta)\in U\times V$ if and only if $\eta=j(v_0,\kappa)$.}\label{exp-criterion}
\end{eqnarray}
We define two maps $\iota_u:U\rightarrow C^{k,\alpha}_{con}(\overline{\R^n},N)$ and $s_u: U\rightarrow I_u$ by
\begin{eqnarray*}
\iota_u(v_0,\kappa)&:=&\pi_N(u+T_u(v_0)+\kappa+j(v_0,\kappa)),\\
s_u(v_0,\kappa)&:=&P_{I_u}\circ P_u\circ H_f(\iota_u(v_0,\kappa)).
\end{eqnarray*}
The maps $\iota_u$ and $s_u$ are well-defined and smooth. We are now in a position to prove the assertions of Theorem \ref{loc-str}.

\begin{enumerate}
\item By its very definition, $\iota_u(u_0,0)=u$ and since $\kappa$ and $j(v_0,\kappa)$ are going to $0$ at infinity, $(\iota_u(v_0,\kappa))_0=v_0$, for $(v_0,\kappa)\in U$.\\

\item According to (\ref{exp-criterion}), if $(v_0,\kappa)\in U$ is such that $\iota_u(v_0,\kappa)$ is an expanding solution then $H_f(\iota_u(v_0,\kappa))=0$ and $s_u(v_0,\kappa)=0$ follows by definition of the map $s_u$. \\

Conversely, assume $(v_0,\kappa)\in U\cap s_u^{-1}(0)$. Then since $\Ima L_u$ and $I_u$ are complementary, one gets that $P_u\circ H_f(\iota_u(v_0,\kappa))=0$. Now, pointwis, 
\begin{eqnarray*}
|H_f(\iota_u(v_0,\kappa))|^2&=&\left<H_f(\iota_u(v_0,\kappa)),H_f(\iota_u(v_0,\kappa))-P_{u}\circ(H_f(\iota_u(v_0,\kappa))\right>\\
&=&\left<H_f(\iota_u(v_0,\kappa)),P_{\iota_u(v_0,\kappa)}\circ H_f(\iota_u(v_0,\kappa))-P_{u}\circ(H_f(\iota_u(v_0,\kappa))\right>\\
&\leq&\|P_{\iota_u(v_0,\kappa)}-P_u\| |H_f(\iota_u(v_0,\kappa))|^2,
\end{eqnarray*}
where $\|P_{\iota_u(v_0,\kappa)}-P_u\|$ denotes the operator norm of $P_{\iota_u(v_0,\kappa)}-P_u$.

Therefore, if $(v_0,\kappa)\in s_u^{-1}(0)$ is in a sufficiently small neighborhood of $(u_0,0)$ then the previous inequalities show that $H_f(\iota_u(v_0,\kappa))=0$, i.e. that $\iota_u(v_0,\kappa)$ defines a smooth expanding map coming out of $v_0$ in $C_{con}^{k,\alpha}(\overline{\R^n},N)$.\\

\item To show that $D\iota_u(u_0,0):C^{k,\alpha}(\mathbb{S}^{n-1},T_{u_0}N)\rightarrow C_{con}^{k,\alpha}(\overline{\R^n},T_uN)$ is injective, note that:
\begin{eqnarray*}
&&D\iota_u(u_0,0)(\xi,\kappa)=D_1\iota_u(u_0,0)(\xi)+D_2\iota_u(u_0,0)(\kappa),\\
&&D_2\iota_u(u_0,0)(\xi,\kappa)=\kappa+D_2j(u_0,0)(\kappa),\\
&&(\xi,\kappa)\in C^{k,\alpha}(\mathbb{S}^{n-1},T_uN)\times C_{con}^{k,\alpha}(\overline{\R^n},T_uN).
\end{eqnarray*}
Now, since $N(v_0,\kappa,j(v_0,\kappa))=0$ for every $(v_0,\kappa)\in U$, one gets by differentiating 
\begin{eqnarray*}
L_u(\kappa+D_2j(u_0,0)(\kappa))=0.
\end{eqnarray*}
This implies that $\kappa+D_2j(u_0,0)(\kappa)\in \ker_0L_u$ which in turn gives $D_2j(u_0,0)(\kappa)\in \ker_0L_u$. By definition, $D_2j(u_0,0)(\kappa)\in K_u$, therefore $D_2j(u_0,0)(\kappa)=0$. In particular, we get:
\begin{eqnarray}
D_2\iota_u(u_0,0)(\kappa)=\kappa,\quad \forall \kappa\in \ker_0L_u.\label{remark-diff-map-F}
\end{eqnarray}

We are in a position to prove that $D\iota_u(u_0,0)$ is injective: if $(\xi_0,\kappa)\in \ker D\iota_u(u_0,0)$, then by (\ref{remark-diff-map-F}), 
\begin{eqnarray*}
0=D\iota_u(u_0,0)(\xi_0,\kappa)=D_1\iota_u(u_0,0)(\xi_0)+\kappa.
\end{eqnarray*}
At infinity, this shows that $\xi_0=(D_1\iota_u(u_0,0)(\xi_0))_0=-\kappa_0=0.$ Consequently, $\xi_0=0$ and $\kappa=0$. \\

\item
To show that $Ds_u(u_0,0)$ is surjective, it is sufficient to prove that $D_1s_u(u_0,0):B_u^{k,\alpha}\subset C^{k,\alpha}(\mathbb{S}^{n-1},T_{u_0},N) \rightarrow I_u$ is an isomorphism. Now, by Remark \ref{rk-bdy-map-jac-field}, $$\dim B_u^{k,\alpha}=\dim \ker_0L_u=\codim \Ima L_u=\dim I_u,$$ which implies that it is sufficient to prove that $D_1s_u(u_0,0)$ is injective.

Let $\kappa_{\infty}\in B_u^{k,\alpha}$ such that $D_1s_u(u_0,0)(\kappa_{\infty})=0$ where $\kappa_{\infty}$ is defined by Theorem \ref{Analysis-Jacobi-field} for $\kappa\in \ker_0L_u$. Now, $D_1\iota_u(u_0,0)(\kappa_{\infty})=:\xi\in\ker L_u$. By integrating by parts:
 \begin{eqnarray*}
0&=&\int_{B(0,R)}\left<L_u\kappa,\xi\right>-\left<\kappa,L_u\xi\right>d\mu_f\\
&=&\int_{\partial B(0,R)}\left<\nabla_{\partial_r}\kappa,\xi\right>-\left<\kappa,\nabla_{\partial_r}\xi\right>d\sigma_f.
\end{eqnarray*} 
Since $\kappa=\textit{O}\left(f^{-n/2}e^{-f}\right)$ and $\nabla_{\partial_r}\xi=\textit{O}(f^{-1/2})$ by Theorem \ref{Analysis-Jacobi-field}, one has:$$\lim_{R\rightarrow +\infty} \int_{\partial B(0,R)}\left<\kappa,\nabla_{\partial_r}\xi\right>d\sigma_f=0.$$
Finally:
\begin{eqnarray*}
\left<\nabla_{\partial_r}\kappa,\xi\right>e^f&=&\left<\nabla_{\partial_r}\left(f^{-n/2}e^{-f}\left(f^{n/2}e^f\kappa\right)\right),\xi\right>e^f\\
&=&-f^{(n-1)/2}<\kappa_{\infty},\xi>+\textit{O}(f^{-(n+1)/2}),
\end{eqnarray*}
which shows:
\begin{eqnarray*}
\int_{\mathbb{S}^{n-1}}|\kappa_{\infty}|^2d\sigma&=&\int_{\mathbb{S}^{n-1}}<\kappa_{\infty},\xi_0>d\sigma\\
&=&-c_n\lim_{R\rightarrow+\infty}\int_{\partial B(0,R)}\left<\nabla_{\partial_r}\kappa,\xi\right>d\sigma_f\\
&=&0,
\end{eqnarray*}
i.e. $\kappa_{\infty}=0$, which means that $D_1s_u(u_0,0)$ restricted to $B_u^{k,\alpha}$ is injective.\\
\item 
Let $v$ be a smooth expanding map in $\Ent_{xp}^{k,\alpha}(N)$ close to $u$. Recall that by the definition of the map $\Phi_u$:  $$v=\pi_N(u+\Phi_u^{-1}(v)).$$ 
Note that $\Phi_u^{-1}(v)$ takes its values in $T_uN$ and by construction, $-T_u(v_0)+\Phi_u^{-1}(v)=:h\in \mathcal{D}_f^{k,\alpha}(\R^n,T_uN).$ Define $\kappa:=P_{\ker_0L_u}h$ and $\eta:= P_{K_u}h$ so that $h=\kappa+\eta$. Since $v$ is an expanding map close to $u$, $v=\pi_N(u+T_u(v_0)+\kappa+\eta)=\iota_u(v_0,\kappa)$.\\
\item According to (\ref{item-surj}) and the implicit function theorem for Banach manifolds, $s_u^{-1}(0)$ is a manifold around $(u_0,0)$ of codimension $\dim \ker_0L_u$. The splitting of the tangent space of $s_u^{-1}(0)$ at $(u_0,0)$ comes from the following observation:
\begin{eqnarray}
&&Ds_u(u_0,0)(\xi,\kappa)=D_1s_u(u_0,0)(\xi)=L_u(D_1\iota_u(u_0,0)(\xi)),\label{split-diff-sub}\\
&&(\xi,\kappa)\in C^{k,\alpha}(\mathbb{S}^{n-1},T_uN)\times\ker_0L_u.
\end{eqnarray}\\

\item According to the beginning of this proof, if $u\in\Ent_{xp}^{k,\alpha}(N)\subset C_{con}^{k,\alpha'}(\overline{\R^n},N)$, $\alpha'\in(0,\alpha)$, then there exist a neighborhood $(u_0,0)\in U'\subset C^{k,\alpha'}(\mathbb{S}^{n-1},N)\times \ker_0L_u$ and a smooth map with respect to $C^{k,\alpha'}$ norms 
\begin{eqnarray*}
&&\iota_u':U'\rightarrow C_{con}^{k,\alpha'}(\overline{\R^n},N),
\end{eqnarray*}
satisfying (\ref{item-unique-exp}). We claim that $\iota_u'|_{U'\cap C^{k,\alpha}(\mathbb{S}^{n-1},N)\times \ker_0L_u}$ is smooth and $\iota_u'|_{U'\cap U}=\iota_u|_{U'\cap U}$.

Indeed, if $(v_0,\kappa)\in C^{k,\alpha}(\mathbb{S}^{n-1},N)\times \ker_0L_u$ is sufficiently close to $(u_0,0)$ then $\iota_u'(v_0,\kappa)$ is an expanding map. In particular, $$A( \iota_u'(v_0,\kappa))(\nabla \iota_u'(v_0,\kappa),\nabla \iota_u'(v_0,\kappa))\in C^{k-1,\alpha'}_f(\R^n,\R^m)\subset C_f^{k-2,\alpha}(\R^n,\R^m).$$ 
Now, let $\eta :\R^n\rightarrow [0,1]$ be a smooth function such that $\eta\equiv 0$ on $B(0,1)\subset \R^n$ and $\eta\equiv 1$ outside $B(0,2)$. Then $\Delta_f(\iota_u'(v_0,\kappa)-\eta v_0)$ is in $C_f^{k-2,\alpha}(\R^n,\R^m)$. Since $\Delta_f:\mathcal{D}_f^{k,\alpha}(\R^n,\R^m)\rightarrow C_f^{k-2,\alpha}(\R^n,\R^m)$ is an isomorphism, there exists a vector field $X$ in $\mathcal{D}_f^{k,\alpha}(\R^n,\R^m)$ such that $\Delta_fX=\Delta_f(\iota_u'(v_0,\kappa)-\eta v_0).$ Since $\iota_u'(v_0,\kappa)-\eta v_0$ converges to $0$ at infinity, the maximum principle tells us that $\iota_u'(v_0,\kappa)-\eta v_0=X$ is in $\mathcal{D}_f^{k,\alpha}(\R^n,\R^m)$. Therefore, as $f^{(k-2)/2}\nabla^k (f\cdot X)\in C_{con}^{0,\alpha}(\R^n,\R^m)$, $\iota_u'(v_0,\kappa)$ is in $C_{con}^{k,\alpha}(\overline{\R^n},N)$. Using [(\ref{item-surj}), Theorem \ref{loc-str}], this ends the claim, i.e. $\iota_u'(v_0,\kappa)=\iota_u(v_0,\kappa)$ for $(v_0,\kappa)\in U'\cap U$.
 The new corresponding neighborhood is defined by $U'\cap C^{k,\alpha}(\mathbb{S}^{n-1},N)$ which is open in $C^{k,\alpha}(\mathbb{S}^{n-1},N)$ for the $C^{k,\alpha'}(\mathbb{S}^{n-1},N)$ topology. Similarly for $W$ defined by [(\ref{item-surj}), Theorem \ref{loc-str}], one can take $W'\cap C_{con}^{k,\alpha}(\overline{\R^n},N)$ to be an open neighborhood of $u$ in $C_{con}^{k,\alpha}(\overline{\R^n},N)$ in the $C_{con}^{k,\alpha'}(\overline{\R^n},N)$ topology.

\end{enumerate}
\end{proof}

We next show that the moduli space of expanders $\Ent_{xp}^{k,\alpha}(N)$ is globally a smooth Banach manifold, more precisely:
\begin{theo}\label{theo-glo-str}
\begin{enumerate}
\item (Integrability)\label{item-integ}
The tangent space to $\Ent_{xp}^{k,\alpha}(N)$ at a point $u$ is $\ker L_u\cap C_{con}^{k,\alpha}(\overline{\R^n},T_uN)$: any bounded element in the kernel of the Jacobi operator is the initial velocity vector field of a one-parameter family of expanding maps.\\

\item (Global structure)
\begin{enumerate}
\item \label{item-Mod-spa} The moduli space of expanders $\Ent_{xp}^{k,\alpha}(N)$ is a second countable Banach manifold modeled on the boundary maps $C^{k,\alpha}(\mathbb{S}^{n-1},N)$.\\

\item \label{item-Sard-Smale}The projection map at infinity $\Pi:\Ent_{xp}^{k,\alpha}(N)\rightarrow C^{k,\alpha}(\mathbb{S}^{n-1},N)$ defined by $\Pi(u):=u_0$ is a smooth Fredholm map of degree $0$ with $\ker D\Pi(u)=\ker_0 L_u$ and the set 
\begin{eqnarray}\label{set-inf-def}
B_u^{k,\alpha}=\left\{\kappa_{\infty}\in C^{k,\alpha}(\mathbb{S}^{n-1},T_uN)\quad|\quad\kappa\in\ker_0L_u\right\},
\end{eqnarray}
is perpendicular to $D\Pi(T_u\Ent_{xp}^{k,\alpha}(N))$.\\

In particular, the set of regular values of $\Pi$ is a residual set.

\end{enumerate}

\end{enumerate}

\end{theo}

\begin{proof}
(Proof of (\ref{item-integ})) Let $u\in \Ent_{xp}^{k,\alpha}(N)$. We show that $\xi\in\ker L_u\cap C_{con}^{k,\alpha}(\overline{\R^n},T_uN)$ if and only if there exists a one-parameter family of expanding maps $(u(t))_{t\in(-\varepsilon,\varepsilon})\in \Ent_{xp}^{k,\alpha}(N)$ such that 
\begin{eqnarray*}
u(0)=u,\quad \frac{d}{dt}\bigg\rvert_{t=0}u(t)=\xi.
\end{eqnarray*}
The \guillemotleft if\guillemotright -part is left to the reader. Let $\xi\in\ker L_u\cap C^{k,\alpha}(\overline{\R^n},T_uN)$. Then consider the map $D_1\iota_u(u_0,0)(\xi_0)$ in $C^{k,\alpha}(\overline{\R^n},T_uN)$ as in the proof of [(\ref{item-inj}), Theorem \ref{loc-str}]. Then by the definition of the map $\iota_u$, the difference $\xi-D_1\iota_u(u_0,0)(\xi_0)$ is in $\mathcal{D}_f^{k,\alpha}(\R^n,T_uN)$. Now, observe that $P_{\Ima L_u}\circ L_u(D_1\iota_u(u_0,0)(\xi_0))=0$. Therefore, $L_u(\xi-D_1\iota_u(u_0,0)(\xi_0))=0$ and $\xi-D_1\iota_u(u_0,0)(\xi_0)$ is in $\mathcal{D}_f^{k,\alpha}(\R^n,T_uN)$, i.e. 

\begin{eqnarray*}
\kappa:=\xi-D_1\iota_u(u_0,0)(\xi_0)\in\ker_0L_u,\quad D_1\iota_u(u_0,0)(\xi_0)\in \ker L_u\cap C_{con}^{k,\alpha}(\overline{\R^n},T_uN).
\end{eqnarray*}

According to (\ref{split-diff-sub}) and (\ref{split-tan-space}), one deduces that $(\xi_0,\kappa)$ is in the tangent space of $s_u^{-1}(0)$ at $(u_0,0)$. By [(\ref{item-chart}), Theorem \ref{loc-str}], there exists a one-parameter family of maps $(v_0(t),\kappa(t))_{t\in(-\varepsilon,\varepsilon)}$ in $s_u^{-1}(0)$ such that
\begin{eqnarray*}
(v_0(t),\kappa(t))|_{t=0}=(u_0,0),\quad \frac{d}{dt}\bigg\rvert_{t=0}(v_0(t),\kappa(t))=(\xi_0,\kappa).
\end{eqnarray*}
Consequently, the curve $t\in(-\varepsilon,\varepsilon)\rightarrow\iota_u(v_0(t),\kappa(t))\in \Ent_{xp}^{k,\alpha}(N)$ is well-defined and smooth and its initial velocity vector field is exactly $D_1\iota(u_0,0)(\xi_0)+\kappa=\xi$. This ends the proof of (\ref{item-integ}).

(Proof of (\ref{item-Mod-spa})) Let $u_1$ and $u_2$ be two expanding maps in $\Ent_{xp}^{k,\alpha}(N)$. By using the notations and the results from [(\ref{item-chart}),Theorem \ref{loc-str}], let $(\iota_{u_i},U_i\cap s_{u_i}^{-1}(0),W_i\cap \Ent_{xp}^{k,\alpha}(N)=:W'_i)_{i=1,2}$ be the corresponding charts. Then the composition 
$$\iota_{u_1}^{-1}\circ \iota_{u_2}:\iota_{u_2}^{-1}(W'_1\cap W_2')\rightarrow \iota_{u_1}^{-1}(W'_1\cap W_2'),$$ is a smooth map.

According to [(\ref{item-second-count}), Theorem \ref{loc-str}], one can choose the neighborhhods $(U_i)_{i=1,2}$ and $(W_i)_{i=1,2}$ to be open with respect to the $C^{k,\alpha'}$ topology for $\alpha'\in(0,\alpha)$ which implies the second countability property of $\Ent_{xp}^{k,\alpha}(N)$.\\

(Proof of (\ref{item-Sard-Smale}))
Let $u$ be an expanding map in $\Ent_{xp}^{k,\alpha}(N)$ and let $(\iota_{u},U\cap s_{u}^{-1}(0),W\cap \Ent_{xp}^{k,\alpha}(N))$ be a corresponding chart. Then, $\Pi\circ \iota_u(v_0,\kappa)=v_0$ for $(v_0,\kappa)\in U\cap s_{u}^{-1}(0)$ by [(\ref{item-str-inf-iota}), Theorem \ref{loc-str}] which shows that $\Pi$ is smooth. Now, $D\Pi(u)(\xi)=\xi_0$ for $\xi\in T_u\Ent_{xp}^{k,\alpha}(N)=\ker L_u\cap C_{con}^{k,\alpha}(\overline{\R^n},T_uN)$. Therefore, $\xi\in \ker D\Pi(u)$ if and only if $\xi\in \ker L_u\cap C_{con}^{k,\alpha}(\overline{\R^n},T_uN)$ and $\kappa$ goes to $0$ at infinity, i.e. if and only if $\kappa\in \ker_0L_u$. In particular, this and the results of Theorem \ref{loc-str} show that $\Ima D\Pi(u)$ is of finite codimension and is isomorphic to $\ker L_u\cap C_{con}^{k,\alpha}(\overline{\R^n},T_uN)/\ker_0L_u$, i.e. $\Pi$ is Fredholm of degree $0$.

Let us show that the set defined by (\ref{set-inf-def}) is perpendicular to $\Ima D\Pi(u)$. Similarly to the proof of [(\ref{item-surj}), Theorem \ref{loc-str}], if $\xi\in T_u\Ent_{xp}^{k,\alpha}(N)$ and $\kappa\in\ker_0L_u$, one has by integrating by parts:
\begin{eqnarray*}
0=\int_{B(0,R)}\left<L_u\kappa,\xi\right>-\left<\kappa,L_u\xi\right>d\mu_f=\int_{\partial B(0,R)}\left<\nabla_{\partial_r}\kappa,\xi\right>-\left<\kappa,\nabla_{\partial_r}\xi\right>d\sigma_f,\quad R>0.
\end{eqnarray*}
By letting $R$ go to $+\infty$ together with Theorem \ref{Analysis-Jacobi-field}:
\begin{eqnarray*}
\int_{\mathbb{S}^{n-1}}\left<\kappa_{\infty},\xi_0\right>d\sigma=0,
\end{eqnarray*}
which is exactly the desired result. 
Finally, we are in a good position to use Sard-Smale's Theorem \cite{Sard-Smale} to ensure that the set of critical values of $\Pi$ is of first category in $C^{k,\alpha}(\mathbb{S}^{n-1},N)$. A proof using only Sard's Theorem can be given in the spirit of \cite{White-var-met} or [Theorem $6.6$, \cite{Har-Mou}].
\end{proof}
 
\section{Unique continuation and generic uniqueness}\label{uni-cont-gen-uni}

In this section, we prove a unique continuation at infinity for expanding maps coming out of the same initial condition. This result will be used in a crucial way to prove a generic property about the uniqueness of such expanders with vanishing relative entropy.

\begin{theo}(Unique continuation at infinity)\label{uni-cont-inf}
Let $u_1$ and $u_2$ be two expanding solutions in $\Ent_{xp}^{k,\alpha}(N)$ with $k\geq 4$, $\alpha\in(0,1)$ coming out of the same map $u_0\in C^{k,\alpha}(\mathbb{S}^{n-1},N)$. Then there exists a map $\kappa_{12}\in C^{k-4}(\mathbb{S}^{n-1},T_{u_0}N)$ such that the limit,
\begin{eqnarray}
\lim_{r\rightarrow+\infty}f^{\frac{n}{2}}e^f(u_2-u_1)(r,\omega)=:\kappa_{12}(\omega),\quad \omega\in\mathbb{S}^{n-1},
\end{eqnarray}
exists and holds in the $C^{k}(\Sp^{n-1},\R^m)$ topology. Moreover, $u_1=u_2$ if and only if $\kappa_{12}=0$.
\end{theo}

\begin{rk}
Theorem \ref{uni-cont-inf} should be interpreted as a unique continuation at infinity, i.e. it involves a Carleman estimate adapted to the harmonic oscillator: it is a non-linear version of Theorem \ref{Analysis-Jacobi-field}. The heuristics behind this rate are the following: the difference $u_1-u_2$ lies approximately in the kernel of the weighted operator $\Delta_f$ and goes to zero at infinity. The solutions that lies in the kernel of $\Delta_f$ are of two kinds: either they behave like $f^{-1}$ or they decay like $f^{-\frac{n}{2}}e^{-f}$. 
\end{rk}

\begin{rk}
The restriction $k\geq 4$ on the regularity at infinity is due to Proposition \ref{autom-reg-infty} whose proof asks for two more derivatives: see Remark \ref{rk-autom-reg}. This explains the regularity of the limit $\kappa_{12}$: this is far from being optimal but it is sufficient to prove Theorem \ref{theo-gen-uni} on generic uniqueness.
\end{rk}
We use and adapt the results from \cite{Uni-Con-Egs-Der} on unique continuation at infinity for approximate eigentensors of the drift Laplacian $\Delta_f$: as the methods are very similar, we only give the main steps. We start with a key estimate based on a commutator estimate: see \cite{Uni-Con-Egs-Der} and the references therein for the motivations of such estimates.

By using the same notations introduced in [Section $3$, \cite{Uni-Con-Egs-Der}], one defines two operators acting on $C^{\infty}_0(\R^n,\R^m)$ as follows:
\begin{eqnarray*}
H\phi:=-\Delta_f\phi,\quad A\phi:=\nabla_{\nabla f}\phi+\frac{f}{2}\phi,\quad \phi\in C^{\infty}_0(\R^n,\R^m).
\end{eqnarray*}
Then one checks that $A$ is anti-symmetric with respect to the weighted measure $e^fdx$. [Corollary $3.3$, \cite{Uni-Con-Egs-Der}] applied to the Ricci expanding soliton $(\R^n,\eucl,\nabla f)$ gives:
\begin{eqnarray}
\left<[H,A]\phi,\phi\right>_{L^2_f}=\int_{\R^n}\left<H\phi,\phi\right>-\frac{f}{2}|\phi|^2e^fdx,\quad\phi\in C^{\infty}_0(\R^n,\R^m),\label{Mourre-est}
\end{eqnarray}
where the notation $L^2_f$ stands for the space $L^2(e^fdx)$.

The next lemma corresponds to [Lemma $3.6$, \cite{Uni-Con-Egs-Der}] that needs some modification in the setting of expanding solutions to the harmonic map flow. For this purpose, define the following weight:
\begin{eqnarray*}
F_{\alpha}:=\frac{f}{2}+\frac{2\alpha+n}{4}\ln f,\quad \alpha \geq 0.
\end{eqnarray*}
In the sequel, we intend to estimate an element in the kernel of the Jacobi operator or the difference of two expanding solutions of the Harmonic map flow in the $L^2$ spaces defined with respect to the weighted measures $e^{f+2F_{\alpha}}dx$ for all $\alpha\geq 0$.
\begin{lemma}\label{lemma-carleman}
Let $u$ be an expanding solution of the Harmonic map flow in $\Ent_{xp}^{k,\alpha}(N)$ with $k\geq 2$, $\alpha\in(0,1)$.
Then there exist some positive constant $c$ and a positive radius $R$ such that for any positive $\alpha$ with $\alpha R^2\geq c$ and for any tensor $\phi\in C^{\infty}_0(\R^n,T_uN)$ supported outside $B(0,R)$,
\begin{eqnarray}
\int_{\R^n}\left(\alpha+\frac{\alpha^2}{f^2}\right)|\phi|^2e^{2F_{\alpha}+f}dx\leq c\left(1+\frac{1}{\alpha}\right)\|(H\phi)^{\top}e^{F_{\alpha}}\|^2_{L^2_f}+c\left\|(H\phi)^{\perp}e^{F_{\alpha-1}}\right\|_{L^2_f}^2.\label{crucial-est-unique-est}
\end{eqnarray}
In particular, (\ref{crucial-est-unique-est}) holds for any smooth tensor supported outside $B(0,R)$ such that $\phi e^{F_{\alpha}}$, $(H\phi)^{\top}e^{F_{\alpha}}$, $(H\phi)^{\perp}e^{F_{\alpha-1}}$ and $(\nabla \phi)e^{F_{\alpha-1}}$ lie in $L^2(e^fdx)$.
\end{lemma}
\begin{proof}[Proof of Lemma \ref{lemma-carleman}]
Consider $\phi_{\alpha}:=e^{F_{\alpha}}\phi$ for $\alpha\geq 0$. Let us apply (\ref{Mourre-est}) to $\phi_{\alpha}$:
\begin{eqnarray}
\left<[H,A]\phi_{\alpha},\phi_{\alpha}\right>_{L^2_f}=\int_{\R^n}\left<H\phi_{\alpha},\phi_{\alpha}\right>-\frac{f}{2}|\phi_{\alpha}|^2e^fdx.\label{equ-mourre}
\end{eqnarray}
Define a function $w_{\alpha}$ by: $$\nabla F_{\alpha}=\left(\frac{1}{2}+\frac{2\alpha+n}{4f}\right)\nabla f=:w_{\alpha}\nabla f.$$ Note that $w_{\alpha}\geq 1/2$. Now we compute $H\phi_{\alpha}$ as in the proof of [Lemma $3.6$, \cite{Uni-Con-Egs-Der}]:
\begin{eqnarray}
H\phi_{\alpha}=(H\phi)e^{F_{\alpha}}+|\nabla F_{\alpha}|^2\phi_{\alpha}-B\phi_{\alpha},\label{H-phi-B}
\end{eqnarray}
where the operator $B$ is defined by:$$B\phi_{\alpha}:=2w_{\alpha}A\phi_{\alpha}+<\nabla f,\nabla w_{\alpha}>\phi_{\alpha}. $$One can check that $B$ is anti-symmetric with respect to the weighted measure $e^fdx$. The next step consists in estimating the scalar product $\left<[H,A]\phi_{\alpha},\phi_{\alpha}\right>_{L^2_f}$ from above:
\begin{eqnarray*}
2\left<H\phi_{\alpha},A\phi_{\alpha}\right>_{L^2_f}&\leq& 2\left<A\phi_{\alpha},(H\phi)e^{F_{\alpha}}+|\nabla F_{\alpha}|^2\phi_{\alpha}-\left<\nabla f,\nabla w_{\alpha}\right>\phi_{\alpha}\right>_{L^2_f}-|A\phi_{\alpha}|^2_{L^2_f}\\
\end{eqnarray*}
Now, since $\phi$ takes its values into $T_uN$:
\begin{eqnarray*}
\left<A\phi_{\alpha},(H\phi)e^{F_{\alpha}}\right>_{L^2_f}&=&\left<A\phi_{\alpha},(H\phi)^{\top}e^{F_{\alpha}}\right>_{L^2_f}+\left<A\phi_{\alpha},(H\phi)^{\perp}e^{F_{\alpha}}\right>_{L^2_f}\\
&=&\left<A\phi_{\alpha},(H\phi)^{\top}e^{F_{\alpha}}\right>_{L^2_f}+\left<\nabla_{\nabla f}\phi_{\alpha},(H\phi)^{\perp}e^{F_{\alpha}}\right>_{L^2_f}\\
&\leq&\left\|A\phi_{\alpha}\right\|_{L^2_f}\left\|(H\phi)^{\top}e^{F_{\alpha}}\right\|_{L^2_f}+C\left\|\phi_{\alpha}f^{-1/2}\right\|_{L^2_f}\left\|(H\phi)^{\perp}f^{-1/2}e^{F_{\alpha}}\right\|_{L^2_f},
\end{eqnarray*}
since the operator that  sends $\phi\in C^{\infty}_0(\R^n,T_uN)$ to $(A\phi)^{\perp}$ is actually a zero-order operator decaying like $\nabla_{\nabla f}u$, i.e. like $f^{-1}$. Therefore, by using Young's inequality,
\begin{eqnarray*}
2\left<H\phi_{\alpha},A\phi_{\alpha}\right>_{L^2_f}&\leq&\left\|(H\phi)^{\top}e^{F_{\alpha}}\right\|_{L^2_f}^2+2\left<A\phi_{\alpha},|\nabla F_{\alpha}|^2\phi_{\alpha}-\left<\nabla f,\nabla w_{\alpha}\right>\phi_{\alpha}\right>_{L^2_f}\\
&&+C\left\|\phi_{\alpha}f^{-1/2}\right\|_{L^2_f}\left\|(H\phi)^{\perp}f^{-1/2}e^{F_{\alpha}}\right\|_{L^2_f}.
\end{eqnarray*}
By using the definition of $F_{\alpha}$ and an integration by parts to get rid of the presence of the operator $A$ in the previous estimate, one gets:
\begin{eqnarray}
2\left<H\phi_{\alpha},A\phi_{\alpha}\right>_{L^2_f}&\leq&\left\|(H\phi)^{\top}e^{F_{\alpha}}\right\|_{L^2_f}^2+\left\|(H\phi)^{\perp}e^{F_{\alpha-1}}\right\|_{L^2_f}^2\label{est-comm-1}\\
&&+\left<\phi_{\alpha},\nabla f\cdot(\left<\nabla f,\nabla w_{\alpha}\right>-|\nabla F_{\alpha}|^2)\phi_{\alpha}\right>_{L^2_f}
+C\left\|\phi_{\alpha-1}\right\|_{L^2_f}^2.\label{est-comm-2}
\end{eqnarray}
 We use finally the anti-symmetry of the operator $B$ showing up in (\ref{H-phi-B}) together with inequalities (\ref{est-comm-1}), (\ref{est-comm-2}) and equality (\ref{equ-mourre}) to establish the following estimate:
 \begin{eqnarray*}
&&\int_{\R^n}\left< e^{F_{\alpha}}(H\phi)^{\top},\phi_{\alpha}\right>+\left(|\nabla F_{\alpha}|^2+\nabla f\cdot\left(|\nabla F_{\alpha}|^2-\nabla f\cdot\nabla w_{\alpha}\right)-\frac{f}{2}\right)|\phi_{\alpha}|^2e^fdx\leq\\
&&\left\|(H\phi)^{\top}e^{F_{\alpha}}\right\|_{L^2_f}^2+\left\|(H\phi)^{\perp}e^{F_{\alpha-1}}\right\|_{L^2_f}^2+C\left\|\phi_{\alpha-1}\right\|_{L^2_f}^2.
\end{eqnarray*}
A lengthy computation similar to [p. $3123$, \cite{Uni-Con-Egs-Der}] shows that:
\begin{eqnarray*}
|\nabla F_{\alpha}|^2+\nabla f\cdot\left(|\nabla F_{\alpha}|^2-\nabla f\cdot\nabla w_{\alpha}\right)-\frac{f}{2}=\left(\frac{\alpha}{2}+\frac{n}{4}\frac{\alpha^2}{f^2}\right)(1+\textit{O}(f^{-1})).
\end{eqnarray*}
Therefore, there exists a radius $R$ sufficiently large such that for any positive $\alpha$ and any tensor $\phi$ compactly supported outside $B(0,R)$:
\begin{eqnarray*}
&&\int_{\R^n}\left< e^{F_{\alpha}}(H\phi)^{\top},\phi_{\alpha}\right>+\left(\frac{\alpha}{2}+\frac{n}{4}\frac{\alpha^2}{f^2}\right)(1+\textit{O}(f^{-1}))|\phi_{\alpha}|^2e^fdx\leq\\
&& \left\|(H\phi)^{\top}e^{F_{\alpha}}\right\|_{L^2_f}^2+\left\|(H\phi)^{\perp}e^{F_{\alpha-1}}\right\|_{L^2_f}^2+C\left\|\phi_{\alpha-1}\right\|_{L^2_f}^2.
\end{eqnarray*}
By using Young's inequality on the first term of the left-hand side of the previous inequality:
 \begin{eqnarray*}
&&\int_{\R^n}\left(\alpha+\frac{\alpha^2}{f^2}\right)|\phi_{\alpha}|^2e^fdx\leq c(1+\alpha^{-1}) \left\|(H\phi)^{\top}e^{F_{\alpha}}\right\|_{L^2_f}^2+c\left\|(H\phi)^{\perp}e^{F_{\alpha-1}}\right\|_{L^2_f}^2+c\left\|\phi_{\alpha-1}\right\|_{L^2_f}^2,
\end{eqnarray*}
where $c$ is a uniform positive constant.

This implies the expected Lemma by absorbing the last term $\left\|\phi_{\alpha-1}\right\|_{L^2_f}^2$ by the left-hand side if $\alpha R^2$ is universally large enough. The last statement follows from a density argument.

\end{proof}

Define for any nonnegative integer $k$ and any smooth tensor $\phi\in C^{\infty}_0(\R^n,\R^m)$ the following weighted $L^2$ norms:
\begin{eqnarray*}
I^k_{\alpha}(\phi)&:=&\int_{\R^n}f^k|\nabla^k\phi|^2e^{2F_{\alpha}}e^fdx,\\
J^{\top}_{\alpha}(\phi)&:=&\int_{\R^n}|(H\phi)^{\top}|^2e^{2F_{\alpha}}e^fdx,\quad J^{\perp}_{\alpha}(\phi):=\int_{\R^n}|(H\phi)^{\perp}|^2e^{2F_{\alpha}}e^fdx.
\end{eqnarray*}

We state without proof a proposition analogous to [Proposition $3.5$, \cite{Uni-Con-Egs-Der}]: 
\begin{prop}\label{est-various-norm-H}
Let $u\in\Ent^{k,\alpha}(N)$ be an expanding solution to the Harmonic map flow with $k\geq 2$. Let $\phi\in C^{\infty}_0(\R^n,T_uN)$. Then, for any nonnegative $\alpha$,
\begin{eqnarray}
\int_{\R^n}|\nabla \phi|^2e^{2F_{\alpha}}e^fdx=I^1_{\alpha-1}(\phi)\lesssim J^{\top}_{\alpha}(\phi)+I^0_{\alpha+1}(\phi)+\alpha I^0_{\alpha}(\phi)+\alpha^2 I^0_{\alpha-1}(\phi),\label{est-various-norm-H-inequ}
\end{eqnarray}
where the sign $\lesssim$ means up to a positive multiplicative constant uniform in the parameters.

\end{prop}
The crucial difference here is that the right-hand side of (\ref{est-various-norm-H-inequ}) only contains the tangential part of the weighted Laplacian $\Delta_f$ with respect to the weight $e^{2F_{\alpha}+f}$.

We are in a good position to prove Theorem \ref{uni-cont-inf} together with the unique continuation result stated in Theorem \ref{Analysis-Jacobi-field}.

\begin{proof}[Proof of Theorem \ref{Analysis-Jacobi-field} (unique continuation part)]
 Let $\kappa \in \ker_0L_u$ be a Jacobi field along an expanding solution $u\in\Ent_{xp}^{k,\alpha}(N)$. Then:
\begin{eqnarray}
H\kappa=D_uA(\kappa)(\nabla u,\nabla u)+2A(u)(\nabla u,\nabla \kappa).\label{jac-field-H-formalism}
\end{eqnarray}
 In particular, (\ref{jac-field-H-formalism}) can be rewritten schematically as:
 \begin{eqnarray}
(H\kappa)^{\top}=\textit{O}(f^{-1})\ast \kappa,\quad (H\kappa)^{\perp}=\textit{O}(f^{-1/2})\ast\nabla\kappa.\label{system-project-jac-equ}
\end{eqnarray}
We start by showing that both the gradient and the tangential part of the weighted Laplacian of a Jacobi field lie in $L^2\left(e^{2F_{\alpha}+f}dx\right)$ as soon as $\kappa$ does for a different $\alpha$ eventually:
\begin{lemma}\label{weighted-grad-lap-L^2-lemma}
If $\kappa \in \ker_0L_u$ lies in $L^2\left(e^{2F_{\alpha}+f}dx\right)$ for some nonnegative $\alpha$ then, $\nabla \kappa\in L^2\left(e^{2F_{\alpha-1}+f}dx\right)$, $(H\kappa)^{\top}\in L^2\left(e^{2F_{\alpha}+f}dx\right)$ and $(H\kappa)^{\perp}\in L^2\left(e^{2F_{\alpha-1}+f}dx\right)$
\end{lemma}
The proof of Lemma \ref{weighted-grad-lap-L^2-lemma} is exactly the same as the proof of [Lemma $4.1$, \cite{Uni-Con-Egs-Der}]. Note that Lemma \ref{weighted-grad-lap-L^2-lemma} enables to apply Lemma \ref{lemma-carleman}.
The next claim establishes the expected unique continuation in case $\kappa\in L^2\left(e^{2F_{\alpha}+f}dx\right)$ for all nonnegative $\alpha$:
\begin{claim}\label{Unique-cont-jac-field-claim}
Let $\kappa\in \ker_0L_u$ such that $\kappa\in L^2\left(e^{2F_{\alpha}+f}dx\right)$ for all nonnegative $\alpha$ then $\kappa\equiv 0$.
\end{claim}
\begin{proof}[Proof of Claim \ref{Unique-cont-jac-field-claim}]
Again, the statement of Claim \ref{Unique-cont-jac-field-claim} echoes [Theorem $4.2$, \cite{Uni-Con-Egs-Der}] and so does its proof. It is sufficient to prove that $\kappa\equiv 0$ outside a ball $B(0,R)$ by classical unique continuation results for second order elliptic equations: see \cite{Aron-Uni-Con} for instance.

Let us localize $\kappa$ at infinity by multiplying it by a cut-off function $\eta$ supported outside a ball $B(0,R)$ of radius $R$ sufficiently large such that Lemma \ref{lemma-carleman} is applicable: if $\alpha\geq 1$, the condition $\alpha R^2\geq c$ from Lemma \ref{lemma-carleman}  will be automatically satisfied. Then define the vector field $\tilde{\kappa}:=\eta\kappa$ and apply Lemma \ref{lemma-carleman} in terms of the quantities $I^{k}_{\alpha}(\tilde{\kappa})$ introduced previously:
\begin{eqnarray}
\alpha I^{0}_{\alpha}(\tilde{\kappa})+\alpha^2I^0_{\alpha-2}(\tilde{\kappa})\lesssim \left(1+\frac{1}{\alpha}\right)J^{\top}_{\alpha}(\tilde{\kappa})+J^{\perp}_{\alpha-1}(\tilde{\kappa}).\label{inequ-carleman-2}
\end{eqnarray}
Now, one can write schematically:
\begin{eqnarray*}
H\tilde{\kappa}=\eta(\kappa)+\eta H\kappa,
\end{eqnarray*}
where $\eta(\kappa)$ denotes a (generic) vector field compactly supported in $B(0,R)$ depending on $\kappa$ and $\nabla \kappa$. By projecting on $T_uN$, one gets:
\begin{eqnarray*}
(H\tilde{\kappa})^{\top}=\eta(\kappa)+\textit{O}(f^{-1})\ast \tilde{\kappa},\quad (H\tilde{\kappa})^{\perp}=\eta(\kappa)+\textit{O}(f^{-1/2})\ast\nabla\tilde{\kappa}.
\end{eqnarray*}
Therefore, for $\alpha\geq 1$, 
\begin{eqnarray}
J^{\top}_{\alpha}(\tilde{\kappa})+J^{\perp}_{\alpha-1}(\tilde{\kappa})&\lesssim& C(R,\tilde{\kappa})e^{\sup_{B(0,R)}2F_{\alpha}+f}+I^0_{\alpha-2}(\tilde{\kappa})+\int_{\R^n}|\nabla \tilde{\kappa}|^2e^{2F_{\alpha-2}+f}dx\\
&\lesssim&C(R,\tilde{\kappa})e^{\sup_{B(0,R)}2F_{\alpha}+f}+I^0_{\alpha-2}(\tilde{\kappa})+I^1_{\alpha-3}(\tilde{\kappa}).\label{inequ-amorce-1}
\end{eqnarray}
According to Proposition \ref{est-various-norm-H}:
\begin{eqnarray}
I^1_{\alpha-3}(\tilde{\kappa})\lesssim J^{\top}_{\alpha-2}(\tilde{\kappa})+I^0_{\alpha-1}(\tilde{\kappa})+\alpha I^0_{\alpha-2}(\tilde{\kappa})+\alpha^2 I^0_{\alpha-3}(\tilde{\kappa}).\label{inequ-amorce-2}
\end{eqnarray}
 By concatenating inequalities (\ref{inequ-amorce-1}) and (\ref{inequ-amorce-2}), one checks that for $R$ sufficiently large but independent of $\alpha\geq 1$:
 \begin{eqnarray*}
J^{\top}_{\alpha}(\tilde{\kappa})+J^{\perp}_{\alpha-1}(\tilde{\kappa})\lesssim C(R,\tilde{\kappa})e^{\sup_{B(0,R)}2F_{\alpha}+f}+I^0_{\alpha-1}(\tilde{\kappa})+\alpha I^0_{\alpha-2}(\tilde{\kappa})+\alpha^2 I^0_{\alpha-3}(\tilde{\kappa}).
\end{eqnarray*}
Using now inequality (\ref{inequ-carleman-2}) gives:
\begin{eqnarray*}
\alpha I^{0}_{\alpha}(\tilde{\kappa})+\alpha^2I^0_{\alpha-2}(\tilde{\kappa})\lesssim C(R,\kappa)e^{\sup_{B(0,R)}2F_{\alpha}+f}+I^0_{\alpha-1}(\tilde{\kappa})+\alpha I^0_{\alpha-2}(\tilde{\kappa})+\alpha^2 I^0_{\alpha-3}(\tilde{\kappa}),
\end{eqnarray*}
which implies in turn that:
\begin{eqnarray*}
\alpha I^{0}_{\alpha}(\tilde{\kappa})+\alpha^2I^0_{\alpha-2}(\tilde{\kappa})\lesssim C(R,\kappa)e^{\sup_{B(0,R)}2F_{\alpha}+f}.
\end{eqnarray*}
Therefore, we have proved in particular that for all $\alpha\geq 1$:
\begin{eqnarray*}
\alpha\int_{\R^n\setminus B(0,R)}|\kappa|^2dx\lesssim C(R,\kappa)e^{\sup_{B(0,R)}(2F_{\alpha}+f)-\inf_{\R^n\setminus B(0,R)}(2F_{\alpha}+f)}\lesssim  C(R,\kappa).
\end{eqnarray*}
Hence $\kappa\equiv 0$ on $\R^n\setminus B(0,R)$ by letting $\alpha$ go to $+\infty$.

This ends the proof of Claim \ref{Unique-cont-jac-field-claim}.

\end{proof}
To finish the proof of Theorem \ref{Analysis-Jacobi-field} (unique continuation part), it suffices to check that the assumption of Claim \ref{Unique-cont-jac-field-claim} holds. As the arguments are very close to [Section $4.2$, \cite{Uni-Con-Egs-Der}], we will be very sketchy. Consider the set $$S:=\left\{\alpha\in [0,+\infty)\quad|\quad \kappa\in L^2\left(e^{2F_{\alpha}+f}dx\right)\right\}.$$ We proceed by showing that $S$ is a non-empty, closed and open set of $[0,+\infty)$. 

The fact that $S$ is open can be proved as in [Claim $3$, \cite{Uni-Con-Egs-Der}]: the estimates that are being used are taken from the proof of Claim \ref{Unique-cont-jac-field-claim} in our setting. The same holds for proving that $S$ is closed: this fact corresponds to [Claim $2$, \cite{Uni-Con-Egs-Der}]. 

We give some details to show that $S$ is non-empty: recall from the proof of Theorem \ref{Analysis-Jacobi-field} that the rescaled vector field $\kappa_f:=f^{n/2}e^f\kappa$ satisfies 
\begin{eqnarray}
&&\Delta_{-f}\kappa_f=V(u)\ast\kappa_f+W(u)\ast\nabla \kappa_f,\label{back-heat-eq-resc-jac-fiel-bis}\\
&&V(u)\in C_{f}^{k-2,\alpha}(\R^n,(\R^m)^*\otimes\R^m),\quad W(u)\in C_{f^{1/2}}^{k-1,\alpha}(\R^n,(\nabla \R^m)^*\otimes\R^m)).\label{pot-bac-heat-eq-bis}
\end{eqnarray}
Then, by parabolic Schauder estimates for backward heat equations, one gets that $\nabla^i\kappa_f=\textit{O}(f^{-i/2})$ for $i=1,2$. Therefore, the radial derivative $\nabla_{\nabla f}\kappa_f$ decays at least like $f^{-1}$. In particular, it shows that $\kappa_f=\textit{O}(f^{-1})$ since by assumption, $\kappa_f$ converges to $0$ radially at infinity. Finally, it implies that $\kappa$ lies in $L^2(e^{F_1+f}dx)$, i.e. $1\in S$.
\end{proof}

\begin{proof}[Proof of Theorem \ref{uni-cont-inf}]
We now prove Theorem \ref{uni-cont-inf}. 

Let $u_1$ and $u_2$ be two expanding solutions in $\Ent_{xp}^{k,\alpha}(N)$ coming out of the same $0$-homogeneous map $u_0\in C^{k,\alpha}(\mathbb{S}^{n-1},N)$. Then, by linearizing around the map $u_1$:
\begin{eqnarray*}
\Delta_f(u_2-u_1)=-2A(u_1)(\nabla(u_2-u_1),\nabla u_1)-D_{u_1}A(u_2-u_1)(\nabla u_1,\nabla u_1)+ Q(u_2-u_1),
\end{eqnarray*}
where the term $Q(u_2-u_1)$ satisfies pointwise:
\begin{eqnarray*}
|Q(u_2-u_1)|\leq C(N)\left(|u_2-u_1||\nabla(u_2-u_1)|+|\nabla(u_2-u_1)|^2\right).
\end{eqnarray*}
Therefore, the difference $u_2-u_1$ satisfies a system of equations very similar to (\ref{system-project-jac-equ}):
\begin{eqnarray}
(H(u_2-u_1))^{\top}&=&\textit{O}(f^{-1})\ast (u_2-u_1)+(Q(u_2-u_1))^{\top},\label{system-project-nonlin-equ-1}\\
 (H(u_2-u_1))^{\perp}&=&\textit{O}(f^{-1/2})\ast\nabla(u_2-u_1)+(Q(u_2-u_1))^{\perp}.\label{system-project-nonlin-equ-2}
\end{eqnarray}
The proof is now along the same lines of the proof of the unique continuation part of Theorem \ref{Analysis-Jacobi-field}. 
However, Lemma \ref{lemma-carleman} and Proposition \ref{est-various-norm-H} need to be slightly adjusted since the difference of the two solutions $\xi:=u_2-u_1$ does not take its values into $T_{u_1}N$  a priori. To circumvent this issue, we use the following observation already noticed in the proof of Theorem \ref{prop-dec-time-diff-sol}: the orthogonal projection of $\xi$ on $(T_{u_1}N)^{\perp}$ is depending quadratically on the norm of $\xi$, i.e. there is a positive constant $C(N)$ such that 
\begin{eqnarray}
|\xi^{\perp}|\leq C(N)|\xi|^2,\quad \quad |\xi|\leq \delta(N), \label{pt-better-est-perp-part}
\end{eqnarray}
pointwise, where $\delta(N)$ is a sufficiently small positive constant depending on the geometry of $N$. Moreover, according to [$(3.4.6)$, \cite{Har-Mou}], one can also estimate the first-derivative of $\xi^{\perp}$ in a better way:
\begin{eqnarray}
|\nabla \xi^{\perp}|\leq C(N)\left(|\nabla \xi||\xi|+|\nabla u||\xi|^2\right),\quad |\xi|\leq \delta(N).\label{pt-better-est-grad-perp-part}
\end{eqnarray}

 In particular, the proof of Lemma \ref{lemma-carleman} applied to $\xi$ leads to:
\begin{claim}\label{claim-lemma-carleman-non-linear}
Define the map $\xi_{\alpha}:=e^{F_{\alpha}}\xi$ for $\alpha\geq 0$. Then there exist some positive constant $c$ and a positive radius $R$ such that for any positive $\alpha$ with $\alpha R^2\geq c$,
\begin{eqnarray}
\alpha\|\eta\xi_{\alpha}\|_{L^2_f}+\alpha^2\|\eta\xi_{\alpha-2}\|^2_{L^2_f}\lesssim \left(1+\frac{1}{\alpha}\right)\|(H\eta\xi)^{\top}e^{F_{\alpha}}\|^2_{L^2_f}+\left\|(H\eta\xi)^{\perp}e^{F_{\alpha-1}}\right\|_{L^2_f}^2,\label{crucial-est-unique-est-nonlin}
\end{eqnarray}
for any smooth function $\eta\in C^{\infty}_0(\R^n,\R)$ compactly supported outside $B(0,R)$.
\end{claim}
 \begin{proof}[Proof of Claim \ref{claim-lemma-carleman-non-linear}]
 Since the proof does not use the equations (\ref{system-project-nonlin-equ-1}) and (\ref{system-project-nonlin-equ-2}), one can assume that $\xi$ is already supported outside a ball $B(0,R)$ where $R$ is a radius to be determined later. This remark will lighten the notations.
 
Let us notice that identities (\ref{equ-mourre}) and (\ref{H-phi-B}) remain unchanged when applied to the map $\xi_{\alpha}:=e^{F_{\alpha}}\xi.$ 
 To estimate the scalar product $\left<A\xi_{\alpha},(H\xi)e^{F_{\alpha}}\right>_{L^2_f}$: 
 \begin{eqnarray*}
\left<A\xi_{\alpha},(H\xi)e^{F_{\alpha}}\right>_{L^2_f}&=&\left<A\xi_{\alpha},(H\xi)^{\top}e^{F_{\alpha}}\right>_{L^2_f}+\left<A\xi_{\alpha},(H\xi)^{\perp}e^{F_{\alpha}}\right>_{L^2_f}\\
&=&\left<A\xi_{\alpha},(H\xi)^{\top}e^{F_{\alpha}}\right>_{L^2_f}+\left<\nabla_{\nabla f}\xi_{\alpha}+\frac{f}{2}\xi_{\alpha}^{\perp},(H\xi)^{\perp}e^{F_{\alpha}}\right>_{L^2_f},
\end{eqnarray*}
we observe by using (\ref{pt-better-est-perp-part}) and (\ref{pt-better-est-grad-perp-part}) that, pointwise,
\begin{eqnarray*}
f\left|\xi_{\alpha}^{\perp}\right|&\leq& fe^{F_{\alpha}}|\xi^{\perp}|\leq C(N)f|\xi||\xi_{\alpha}|,\\
\left|(\nabla_{\nabla f}\xi_{\alpha})^{\perp}\right|&\lesssim& \left|\nabla_{\nabla f}\left(\xi_{\alpha}^{\perp}\right)\right|+\frac{|\xi_{\alpha}|}{f}\\
&\lesssim& (f+\alpha)\left|\xi_{\alpha}^{\perp}\right|+e^{F_{\alpha}}\left|\nabla_{\nabla f}(\xi^{\perp})\right|+\frac{|\xi_{\alpha}|}{f}\\
&\lesssim&(f+\alpha)|\xi||\xi_{\alpha}|+\left(f^{\frac{1}{2}}|\nabla \xi|+|\xi|+f^{-1}\right)|\xi_{\alpha}|,
\end{eqnarray*}
since $f^{1/2}|\nabla u|$ is bounded uniformly. Since $\xi$ and $\nabla \xi$ are decaying exponentially by Theorem \ref{prop-dec-time-diff-sol} and Proposition \ref{autom-reg-infty}, the following estimate holds:
\begin{eqnarray*}
\left<A\xi_{\alpha},(H\xi)e^{F_{\alpha}}\right>_{L^2_f}&\lesssim&\left\|A\xi_{\alpha}\right\|_{L^2_f}\left\|(H\xi)^{\top}e^{F_{\alpha}}\right\|_{L^2_f}+\left\|(H\xi)^{\perp}e^{F_{\alpha-1}}\right\|_{L^2_f}^2\\
&&+\left\|\xi_{\alpha-1}\right\|_{L^2_f}^2+\alpha^2\|\xi_{\alpha-3}\|^2_{L^2_f}.
\end{eqnarray*}
By inspecting the proof of Lemma \ref{lemma-carleman}, one arrives at the following conclusion: there exists a radius $R$ sufficiently large such that for any positive $\alpha$:
\begin{eqnarray*}
&&\int_{\R^n}\left< e^{F_{\alpha}}(H\xi),\xi_{\alpha}\right>+\left(\frac{\alpha}{2}+\frac{n}{4}\frac{\alpha^2}{f^2}\right)(1+\textit{O}(f^{-1}))|\xi_{\alpha}|^2e^fdx\leq\\
&& \left\|(H\xi)^{\top}e^{F_{\alpha}}\right\|_{L^2_f}^2+\left\|(H\xi)^{\perp}e^{F_{\alpha-1}}\right\|_{L^2_f}^2+\left\|\xi_{\alpha-1}\right\|_{L^2_f}^2+\alpha^2\|\xi_{\alpha-3}\|^2_{L^2_f}.
\end{eqnarray*}
Using (\ref{pt-better-est-perp-part}) again leads to the proof of Claim \ref{claim-lemma-carleman-non-linear}.
 \end{proof}
 
 Similarly, Proposition \ref{est-various-norm-H} becomes in this setting:
 \begin{claim}\label{est-various-norm-H-nonlin}
For any nonnegative $\alpha$, and any $\eta\in C^{\infty}_0(\R^n,\R)$,
\begin{eqnarray}
\int_{\R^n}|\nabla (\eta\xi)|^2e^{2F_{\alpha}}e^fdx&=&I^1_{\alpha-1}(\eta\xi)\lesssim J^{\top}_{\alpha}(\eta\xi)+J^{\perp}_{\alpha-1}(\eta\xi)\\
&&+I^0_{\alpha+1}(\eta\xi)+\alpha I^0_{\alpha}(\eta\xi)+\alpha^2 I^0_{\alpha-1}(\eta\xi).\label{est-various-norm-H-inequ-nonlin}
\end{eqnarray}

 \end{claim}
Notice the extra-term $J^{\perp}_{\alpha-1}(\eta\xi)$ involving the normal part of the operator $H$ acting on $\eta\xi$.
Then the proof of the unique continuation is exactly the same as the proof of Theorem \ref{Analysis-Jacobi-field} (unique continuation part): Lemma \ref{weighted-grad-lap-L^2-lemma} continues to hold in this setting, Lemma \ref{lemma-carleman} is replaced by Claim \ref{claim-lemma-carleman-non-linear} and Claim \ref{est-various-norm-H-nonlin} plays the role of Proposition \ref{est-various-norm-H}. Finally, to prove that the set $S$ is not empty, one uses Proposition \ref{autom-reg-infty} to show that $f^{n/2}e^f\xi=\textit{O}(f^{-1})$, i.e. $1\in S$.

\end{proof}

Recall that a set $E$ is of codimension $m$ in a Banach space $X$ if $E\subset\cup_{i\geq 1}\Pi_i(X_i)$ where, for each index $i\geq 1$, $X_i$ is a submanifold of $X\times\R^{k_i}$ of codimension $m+k_i$ and where $\Pi_i: X\times\R^{k_i}\rightarrow X$ denotes the projection onto the first factor. 

Note as explained in [Section $1.6$, \cite{Whi-para-ell-fct}] that a set of positive codimension is of first (Baire) category.

\begin{theo}[Generic uniqueness]\label{theo-gen-uni}
The set of regular values of $\Pi$ in $C^{k,\alpha}(\mathbb{S}^{n-1},N)$, $k\geq 4$, that are smoothed out by more than one expanding solution in $\Ent_{xp}^{k,\alpha}(N))$ with $0$ relative entropy is of codimension $1$. In particular, the set of boundary maps in $C^{k,\alpha}(\mathbb{S}^{n-1},N)$ that are smoothed out by more than one expanding solution in $\Ent_{xp}^{k,\alpha}(N)$ with $0$ relative entropy is of first category.
\end{theo}
\begin{proof}
Consider the following set of expanding maps coming out of the same map with $0$ relative entropy:
\begin{eqnarray*}
\mathcal{E}^{k,\alpha}_0(N):=\left\{(u_i)_{i=1,2}\in \Ent_{xp}^{k,\alpha}(N),\quad \Pi(u_1)=\Pi(u_2),\quad u_1\neq u_2,\quad \mathcal{E}(u_2,u_1)=0 \right\}.
\end{eqnarray*}
Let $(u_i)_{i=1,2}\in \mathcal{E}^{k,\alpha}_0(N)$ with $\Pi(u_1)=\Pi(u_2)=:u_0$ and let $(\iota_{u_i},U_i\cap s_{u_i}^{-1}(0),W_i\cap \Ent_{xp}^{k,\alpha}(N))_{i=1,2}$ be the corresponding charts given by [(\ref{item-chart}), Theorem \ref{loc-str}] where the neighborhoods $(U_i)_{i=1,2}$ and $(W_i)_{i=1,2}$ are open for the $C^{k,\alpha'}$ topology. From now on, rename the neighborhoods $U_i\cap s_{u_i}^{-1}(0)$, $i=1,2$ by $U_i$. Then define the following functional that will serve as a Morse functional on a neighborhood of $(u_0,0)$:
$$\begin{array}{ccccc}
\mathcal{E}_{12} & : & U:=U_1\cap U_2\subset C^{k,\alpha}(\mathbb{S}^{n-1},N)\times\{0\} & \mapsto & \R \\
 & & (v_0,0) & \mapsto & \mathcal{E}(\iota_{u_2}(v_0,0),\iota_{u_1}(v_0,0)). \\
\end{array}$$
The maps $\iota_{u_1}(v_0,0)$ and $\iota_{u_2}(v_0,0)$ are two expanding maps in $\mathcal{E}_{xp}^{k,\alpha}(N)$ coming out of the same map $v_0$ by [(\ref{item-str-inf-iota}), Theorem \ref{loc-str}], therefore Theorem \ref{mono-rel-ent} ensures the functional $\mathcal{E}_{12}$ is well-defined. Let $(v_0(\tau))_{-\varepsilon\leq \tau\leq \varepsilon}$ be a smooth curve of maps in $U\subset C^{k,\alpha}(\mathbb{S}^{n-1},N)$ such that:
\begin{eqnarray*}
v_0(0)=u_0,\quad \frac{d}{d\tau}\bigg\rvert_{\tau=0}v_0(\tau)=:\xi_0\in C^{k,\alpha}(\mathbb{S}^{n-1},T_{u_0}N),
\end{eqnarray*}
and let $\kappa_{12}:=\lim_{+\infty}f^{n/2}e^f(u_2-u_1)$ whose existence is ensured by Theorem \ref{uni-cont-inf}. We adapt the proof of the monotonicity property established in Theorem \ref{mono-rel-ent} as follows:
\begin{eqnarray*}
D_1\mathcal{E}_{12}(u_0,0)(\xi_0)&=& 
\lim_{R\rightarrow+\infty}\int_{B(0,R)}\left<\nabla \xi_2,\nabla u_2\right>-\left<\nabla \xi_1,\nabla u_1\right>d\mu_f,\\
\xi_i&:=&D_1\iota_{u_i}(u_0,0)(\xi_0)\in T_{u_i}\Ent_{xp}^{k,\alpha}(N),\quad i=1,2.
\end{eqnarray*}
Now, by integrating by parts once and by using the facts that $\Delta_fu_i\perp T_{u_i}N$ for $i=1,2$:
\begin{eqnarray*}
\int_{B(0,R)}\left<\nabla \xi_2,\nabla u_2\right>-\left<\nabla \xi_1,\nabla u_1\right>d\mu_f&=&\int_{B(0,R)}\left< \xi_1,\Delta_f u_1\right>-\left< \xi_2,\Delta_fu_2\right>d\mu_f\\
&&+\int_{S(0,R)}\left<\xi_2,\nabla_{\partial_r} u_2\right>-\left< \xi_1,\nabla_{\partial_r} u_1\right>d\sigma_f\\
&=&\int_{S(0,R)}\left<\xi_2,\nabla_{\partial_r} u_2\right>-\left< \xi_1,\nabla_{\partial_r} u_1\right>d\sigma_f.
\end{eqnarray*}

We claim that $\xi:=\xi_1-\xi_2$ decays as fast as Jacobi fields along an expanding solution vanishing at infinity, i.e.
$$\xi=\textit{O}\left(f^{-\frac{n}{2}}e^{-f}\right).$$

First of all, the projection $\xi^{\perp_1}$ onto $(T_{u_1}N)^{\perp}$ satisfies:
\begin{eqnarray}
\xi^{\perp_1}=\xi_2^{\perp_1}=\xi_2^{\perp_1}-\xi_2^{\perp_2}=\textit{O}(u_2-u_1)\ast\xi_2=\textit{O}\left(f^{-\frac{n}{2}}e^{-f}\right),\label{dec-xi-perp}
\end{eqnarray}
since $\xi_2\in T_{u_2}N$.

Now, let us derive the evolution equation of the projection $\xi^{\top_1}$ onto $T_{u_1}N$:
\begin{eqnarray*}
\left<\Delta_f\xi^{\top_1},\xi^{\top_1}\right>&=&-2\left<A(u_2)(\nabla u_2,\nabla\xi_2),\xi^{\top_1}\right>-\left<D_{u_2}A(\xi_2)(\nabla u_2,\nabla u_2),\xi^{\top_1}\right>\\
&&-\left<\Delta_f\xi^{\perp_1},\xi^{\top_1}\right>+\left<D_{u_1}A(\xi_1)(\nabla u_1,\nabla u_1),\xi^{\top_1}\right>.
\end{eqnarray*}
Let us estimate each term on the righthand side of the previous equation:
\begin{eqnarray*}
\left<A(u_2)(\nabla u_2,\nabla\xi_2),\xi^{\top_1}\right>&=&\left<\left(A(u_2)(\nabla u_2,\nabla\xi_2)\right)^{\perp_2}-\left(A(u_2)(\nabla u_2,\nabla\xi_2)\right)^{\perp_1},\xi^{\top_1}\right>\\
&=&\left<\textit{O}\left(f^{-1-\frac{n}{2}}e^{-f}\right),\xi^{\top_1}\right>,
\end{eqnarray*}
where we used Theorem \ref{prop-dec-time-diff-sol} to estimate the difference $u_2-u_1$.
\begin{eqnarray*}
&&\left<D_{u_2}A(\xi_2)(\nabla u_2,\nabla u_2)-D_{u_1}A(\xi_1)(\nabla u_1,\nabla u_1),\xi^{\top_1}\right>=\\
&&\left<D_{u_1}A(\xi^{\top_1})(\nabla u_1,\nabla u_1),\xi^{\top_1}\right>+\left<D_{u_1}A(\xi^{\perp_1})(\nabla u_1,\nabla u_1),\xi^{\top_1}\right>\\
&&+\left<D_{u_2}A(\xi_2)(\nabla u_2,\nabla u_2)-D_{u_1}A(\xi_2)(\nabla u_1,\nabla u_1),\xi^{\top_1}\right>.
\end{eqnarray*}
By using (\ref{dec-xi-perp}) together with Theorem \ref{prop-dec-time-diff-sol} and Proposition \ref{autom-reg-infty},
\begin{eqnarray*}
&&\left<D_{u_2}A(\xi_2)(\nabla u_2,\nabla u_2)-D_{u_1}A(\xi_1)(\nabla u_1,\nabla u_1),\xi^{\top_1}\right>=\\
&&\textit{O}(f^{-1})\ast \xi^{\top_1}\ast \xi^{\top_1}+\textit{O}(f^{-1-\frac{n}{2}}e^{-f})\ast\xi^{\top_1}.
\end{eqnarray*}
Similarly, one gets: $\left<\Delta_f\xi^{\perp_1},\xi^{\top_1}\right>=\textit{O}(f^{-1-\frac{n}{2}}e^{-f})$.

 Consequently, the norm of the difference $|\xi^{\top_1}|$ is a weak subsolution of the following differential inequality:
\begin{eqnarray*}
\Delta_f|\xi^{\top_1}|\geq -\textit{O}(f^{-1})|\xi^{\top_1}|-\textit{O}(f^{-n/2-1}e^{-f}).
\end{eqnarray*}

Adapting the proof of Theorem \ref{Analysis-Jacobi-field} on the decay of Jacobi fields that vanish at infinity, one gets, $\xi^{\top_1}=\textit{O}(f^{-n/2}e^{-f}),$ as expected.\\
 
 By using Theorem \ref{uni-cont-inf},
\begin{eqnarray*}
&&\lim_{R\rightarrow+\infty}\int_{B(0,R)}\left<\nabla \xi_2,\nabla u_2\right>-\left<\nabla \xi_1,\nabla u_1\right>d\mu_f=\\
&&\lim_{R\rightarrow+\infty} \int_{S(0,R)}\left<\xi_1,\nabla_{\partial_r} (u_2-u_1)\right>d\sigma_f+\lim_{R\rightarrow+\infty} \int_{S(0,R)}\left<\xi_2-\xi_1,\nabla_{\partial_r} u_2\right>d\sigma_f\\
&=&\lim_{R\rightarrow+\infty} \left(\int_{S(0,R)}\left<\xi_1,\nabla_{\partial_r} (f^{-\frac{n}{2}}e^{-f})f^{\frac{n}{2}}e^f(u_2-u_1)\right>d\sigma_f+\textit{O}(R^{-2})\right)\\
&=&-c_n\int_{\mathbb{S}^{n-1}}\left<\xi_0, \kappa_{12}\right>d\sigma,
\end{eqnarray*}
for some positive constant $c_n$. Consequently:
\begin{eqnarray*}
D_1\mathcal{E}_{12}(u_0,0)(\xi_0)=-c_n\int_{\mathbb{S}^{n-1}}\left<\xi_0, \kappa_{12}\right>d\sigma.
\end{eqnarray*}
Since $u_1\neq u_2$, Theorem \ref{uni-cont-inf} ensures that $\kappa_{12}\neq 0$ which implies that $\mathcal{E}_{12}$ is a local submersion at $(u_0,0)$. Therefore, $\mathcal{E}_{12}^{-1}\{0\}\cap U $ is of codimension $1$ by the implicit function theorem. 

We conclude by invoking the separability of $\Ent_{xp}^{k,\alpha}(N)$ with respect to the $C^{k,\alpha'}$ topology, $\alpha'\in(0,\alpha)$, established in [(\ref{item-Mod-spa}), Theorem \ref{theo-glo-str}]: indeed, there exists a countable subcover $(U^i:=\mathcal{E}_{12}^{-1}\{0\}\cap U^i_1\cap U^i_2)_{i\geq 1}$ of $\Pi(\mathcal{E}^{k,\alpha}_0(N))$ such that each $U^i$ has codimension $1$.  
\end{proof}

\section{Compactness and asymptotic estimates}\label{sec-com-asy-est-neg-cur}

The purpose of this section is to prove that the set of smooth expanding solutions coming out of smooth $0$-homogeneous maps $u_0:\R^n\rightarrow N$ that are regular at infinity is compact provided the set of initial conditions $u_0$ is and an a priori uniform bound on the $C^0$ norm of the gradient of such expanders holds. These results are reminiscent of a previous work due to the author on expanding Ricci solitons \cite{Der-Asy-Com-Egs}.

\begin{theo}\label{Asy-Est-Grad-Theo}
Let $u:\R^n\rightarrow N\subset \R^m$ be a smooth expanding solution coming out of the $0$-homogeneous map $u_0:\R^n\rightarrow N$. Assume $u$ is regular at infinity. Then, 
\begin{eqnarray*}
\sup_{x\in\R^n}f^{\frac{1}{2}}(x)|\nabla u|(x)\leq C\left(n,N,\|\nabla u\|_{C^0(\R^n)},\|\nabla u_0\|_{C^0(\mathbb{S}^{n-1})}\right).
\end{eqnarray*}

\end{theo}

\begin{proof}
Let us derive the evolution equation satisfied by $|\nabla u|^2$ with the help of Bochner's formula:
\begin{eqnarray*}
\Delta_f|\nabla u|^2&=&2|\nabla^2u|^2+2<\Delta_f\nabla u,\nabla u>\\
&=&2|\nabla^2u|^2-|\nabla u|^2+2<\nabla(\Delta_f u),\nabla u>\\
&=&2|\nabla^2u|^2-|\nabla u|^2-2<\nabla(A(u)(\nabla u,\nabla u)),\nabla u>.
\end{eqnarray*}
Since $A(u)(\nabla u,\nabla u)\perp\nabla u$,
\begin{eqnarray*}
\Delta_f|\nabla u|^2&=&2|\nabla^2u|^2-|\nabla u|^2+2<A(u)(\nabla u,\nabla u),\Delta u>\\
&=&2|\nabla^2u|^2-|\nabla u|^2+2<A(u)(\nabla u,\nabla u),\Delta_f u>\\
&=&2|\nabla^2u|^2-|\nabla u|^2-2\left|A(u)(\nabla u,\nabla u)\right|^2.\\
\end{eqnarray*}

Therefore,
\begin{eqnarray*}
\Delta_f|\nabla u|^2&\geq&2|\nabla^2u|^2-|\nabla u|^2-C(N)|\nabla u|^4.
\end{eqnarray*}

Then the function $U:=|\nabla u|^2$ satisfies: $$\Delta_fU\geq -U-C(N)U^2.$$
By [Propositon $1.9$, \cite{Der-Asy-Com-Egs}] applied to $U$ and to the expanding solution $(M^n,g,\nabla^gf)=(\R^n,\eucl,\nabla(r^2/4))$, one gets:
\begin{eqnarray}
\sup_{\R^n}fU\leq C\left(n,N,\sup_{\R^n}U,\limsup_{+\infty}fU\right).
\end{eqnarray}
This in turn means in our setting:
\begin{eqnarray*}
\sup_{\R^n}f|\nabla u|^2\leq C\left(n,N,\sup_{\R^n}|\nabla u|,\|\nabla u_0\|_{C^0(\mathbb{S}^{n-1})}\right),
\end{eqnarray*}
which leads to the expected estimate.

\end{proof}
The next theorem ensures that the higher derivatives of expanding solutions satisfying the assumptions of Theorem \ref{Asy-Est-Grad-Theo} are a priori controlled as well:
\begin{theo}\label{theo-high-der-a-priori}
Let $u:\R^n\rightarrow N\subset \R^m$ be a smooth expanding solution coming out of the $0$-homogeneous map $u_0:\R^n\rightarrow N$. Assume $u$ is regular at infinity. Then, for every integer $k\geq 2$,
\begin{eqnarray}
\sup_{x\in\R^n}f^{\frac{k}{2}}(x)|\nabla^k u|(x)\leq C\left(k,n,\|\nabla u\|_{C^0(\R^n)},\sup_{1\leq i\leq k}\|\nabla ^iu_0\|_{C^0(\mathbb{S}^{n-1})}\right).\label{a-priori-bdy-inf-der}
\end{eqnarray}

\end{theo}

\begin{rk}\label{rk-theo-high-der-a-priori}
Theorem \ref{theo-high-der-a-priori} is stated in terms of expanding solutions that are regular at infinity. If a finite number of rescaled derivatives are bounded and converge to the corresponding covariant derivative of the map at infinity then the same result holds.
\end{rk}

\begin{proof}
The inequalities (\ref{a-priori-bdy-inf-der}) are true without the polynomial weights by classical parabolic estimates applied to $u$: 
\begin{eqnarray}
\sup_{x\in\R^n}|\nabla^k u|(x)\leq C\left(k,n,\|\nabla u\|_{C^0(\R^n)}\right),\quad k\geq 2.\label{class-par-a-priori-loc-bds}
\end{eqnarray}

Therefore, it suffices to prove these inequalities outside a ball $B(0,R)$ whose radius $R$ depends on the constants involved in (\ref{a-priori-bdy-inf-der}) only.

We proceed by induction on $k\geq 1$: the case $k=1$ is exactly the content of Theorem \ref{Asy-Est-Grad-Theo}. Assume $k\geq 2$ and let us compute the evolution equation satisfied by $\nabla^k u$:
\begin{eqnarray*}
\Delta_f\nabla^ku&=&\nabla^k(\Delta_fu)-\frac{k}{2}\nabla^ku\\
&=&-\frac{k}{2}\nabla^ku-\nabla^k(A(u)(\nabla u,\nabla u))\\
&=&-\frac{k}{2}\nabla^ku+\sum_{l=0}^k\nabla^{k-l}(A(u))\ast\nabla^l(\nabla u\ast\nabla u).
\end{eqnarray*}
Now, by Fa\`a di Bruno's formula and by the induction assumption:
\begin{eqnarray*}
\nabla^{m}(A(u))=\textit{O}(\nabla^mu)+\textit{O}(f^{-m/2}),\quad 0\leq m\leq k,
\end{eqnarray*}
where $\textit{O}(\cdot)$ is only depending on the quantities described in (\ref{a-priori-bdy-inf-der}).

By isolating the terms involving the $(k+1)$-th and $k$-th derivatives of $u$ and using the induction assumption again, one gets:
\begin{eqnarray*}
\Delta_f\nabla^ku&=&-\frac{k}{2}\nabla^ku+\nabla^{k+1}u\ast\nabla u+\nabla^{k}u\ast(\nabla u^{*2}+\nabla^2u)+\textit{O}\left(f^{-\frac{k+2}{2}}\right).
\end{eqnarray*}
In particular, if $k\geq 3$, by Young's inequality applied to the norm of the $(k+1)$-th derivatives of $u$,
\begin{eqnarray*}
\Delta_f|\nabla^ku|^2&\geq&-\left(k+\textit{O}(f^{-1})\right)|\nabla^ku|^2-\textit{O}\left(f^{-k-1}\right).
\end{eqnarray*}
Similarly, if $k=2$:
\begin{eqnarray*}
\Delta_f|\nabla^2u|^2&\geq&-\left(2+c|\nabla^2u|+\textit{O}(f^{-1})\right)|\nabla^2u|^2-\textit{O}\left(f^{-3}\right).
\end{eqnarray*}
Now recall the (schematic) equation satisfied by $\nabla u$:
\begin{eqnarray*}
\Delta_f\nabla u&=&-\frac{\nabla u}{2}+\nabla u^{*3}+\nabla^2u\ast\nabla u\\
&=&\textit{O}(f^{-1/2}),
\end{eqnarray*}
where we used Theorem \ref{Asy-Est-Grad-Theo} together with Claim \ref{class-par-a-priori-loc-bds} for $\nabla^2u$. By classical parabolic estimates applied to $\nabla u$: $\nabla^2u=\textit{O}(f^{-1/2})$ uniformly.

 Therefore, in any case, one is reduced to the following differential inequality:
\begin{eqnarray*}
\Delta_f|\nabla^ku|^2&\geq&-\left(k+\textit{O}\left(f^{-1/2}\right)\right)|\nabla^ku|^2-\textit{O}\left(f^{-k-1}\right).
\end{eqnarray*}
Now, one can use [Lemma $2.9$, \cite{Der-Asy-Com-Egs}] or adapt its proof to show the expected result.
Indeed, by elementary but tedious computations in the spirit of the previous ones, the function $f^{k}e^{-Cf^{-1/2}}|\nabla^k u|^2-Af^{-1}$ where $A$ and $C$ are positive constants sufficiently large depending only on the constants involved in (\ref{a-priori-bdy-inf-der}), satisfies the maximum principle outside a large ball $B(0,R)$ whose radius $R$ depends only on the constants involved in (\ref{a-priori-bdy-inf-der}). 

\end{proof}

By applying a Nash-Moser iteration with the help of a suitable Bochner formula, compactness holds when the target manifold is non-positively curved:
\begin{coro}[Non-positively curved target]\label{coro-non-neg-tar-comp}
Let $(N,g)$ be a Riemannian manifold with non-positive sectional curvature. Let $k\geq 2$ and $\alpha\in(0,1)$. Then the following set
\begin{eqnarray*}
\mathcal{E}_{xp}^{k,\alpha}(N,\Lambda):=\left\{u\in \Ent_{xp}^{k,\alpha}( N):   \|u_0\|_{C^{k,\alpha}(\mathbb{S}^{n-1},N)}\leq \Lambda \right\},
\end{eqnarray*}
is compact in the $C_{con}^{k,\alpha'}(\R^n,N)$-topology for $\alpha'\in(0,\alpha).$

Similarly, the set
\begin{eqnarray*}
\mathcal{E}_{xp}^{\infty}(N,(\Lambda_k)_{k\geq 1}):=&\{&u\in C^{\infty}_{loc}(\R^n, N): \text{ $u$ expander regular at infinity}\quad|\\
&& \quad  \|\nabla^k u_0\|_{C^0(\mathbb{S}^{n-1},N)}\leq \Lambda_k,\quad k\geq 1 \},
\end{eqnarray*}
is compact in the smooth conical topology, i.e. in the topology defined by $\cap_{k\geq 0}C^{k,\alpha}_{con}(\R^n,N)$, for any $\alpha\in(0,1)$.

\end{coro}


\begin{proof}
According to Theorems \ref{Asy-Est-Grad-Theo} and \ref{theo-high-der-a-priori}, it is sufficient to prove the following a priori bound: if an expanding solution $u\in C^{\infty}_{loc}(\R^n,N)$ is regular at infinity or in $C^{k,\alpha}_{con}(\R^n,N)$ then there exists a uniform positive constant $C$ such that:
\begin{eqnarray}
\|\nabla u\|_{C^0(\R^n)}\leq C\|\nabla u_0\|_{L^2(\Sp^{n-1})}.\label{exp-inequ}
\end{eqnarray}
The proof is standard: let $u(\cdot,t):=u(\cdot/\sqrt{t})$ be the solution to the harmonic map flow associated to $u$. The Bochner formula in its parabolic version (see [Lemma $5.3.3$, \cite{Lin-Wang-Boo}] for instance) gives:
\begin{eqnarray*}
(\partial_t-\Delta)\left(\frac{|\nabla u|^2}{2}\right)=-|\nabla du|^2+\sum_{i,j=1}^n\det A(u)(\nabla_iu,\nabla_ju)\leq0,
\end{eqnarray*}
by the Gauss-Codazzi equations together with the fact that the sectional curvature of $(N,g)$ is non-positive.  Now, a Nash-Moser iteration applied to the subsolution $|\nabla u|^2$ of the heat equation (see \cite{Str-Har-Map} in a harmonic map flow context) gives:
\begin{eqnarray*}
|\nabla u|^2(x,t)\leq C\fint_{P_r(x,t)}|\nabla u|^2(y,s)dyds,\quad (x,t)\in\R^n\times(0,+\infty), \quad r^2\in(0,t),
\end{eqnarray*}
for some uniform positive constant $C$, where $P_r(x,t):=B(x,r)\times(t-r^2,t]\subset\R^n\times\R_+$ is a parabolic neighborhood of $(x,t)$ in $\R^n\times\R_+$. Take $t=2$ and $r^2=1$ and use the Pohozaev identity from Proposition \ref{Poho-for-gal-sol} to prove as in [Proposition $3.15$, \cite{Der-Lam-HMF}] that:
\begin{eqnarray*}
\|\nabla u\|_{L^2(B(x,1))}(s)\leq C\|\nabla u_0\|_{L^2(B(x,2))},\quad \forall x\in\R^n,\quad s\in[1,2],
\end{eqnarray*}
for some uniform positive constant $C$. The expected inequality (\ref{exp-inequ}) follows by combining the previous inequalities together with the fact that $u_0$ is $0$-homogeneous. This ends the proof if $u$ is assumed to be regular at infinity.

Let us assume now that $u\in\Ent_{xp}^{2,\alpha}(N)$ (the cases $k\geq 3$ can be handled similarly) and let us prove that $\|u\|_{C^{2,\alpha}_{con}(\R^n,N)}$ is uniformly bounded as expected. According to Theorem \ref{theo-high-der-a-priori} and Remark \ref{rk-theo-high-der-a-priori}, there exists a uniform positive constant $C(n,\Lambda)$ such that 
\begin{eqnarray}
\sup_{\R^n}\left\{|f^{1/2}\cdot \nabla u|+|f\cdot \nabla^2 u|\right\}\leq C(n,\Lambda). \label{a-priori-decay-first-sec-der-exp-non-pos-cur}
\end{eqnarray}
In particular, 
\begin{eqnarray*}
|\nabla_{\nabla f}u|\leq\frac{C(n,\Lambda)}{f},
\end{eqnarray*}
which implies by integrating radially that: $|u-u_0|\leq C(n,\Lambda)/f$. 

Now, if $\phi$ is a smooth cut-off function such that $\phi=0$ on $B(0,1)$ and $\phi=1$ outside $B(0,2)$ then:
\begin{eqnarray}
\Delta_f(u-\phi u_0)=-A(u)(\nabla u,\nabla u)-\Delta_f(\phi u_0).\label{equ-sol-substract-init-cond}
\end{eqnarray}
Since $u_0$ is $0$-homogeneous, $\|\Delta_f(\phi u_0)\|_{C^{0,\alpha}_{f}(\R^n,\R^m)}\leq C(n,\Lambda)$ and the same is true for the term $A(u)(\nabla u,\nabla u)$ by (\ref{a-priori-decay-first-sec-der-exp-non-pos-cur}). By (the proof) of Theorem \ref{theo-fred-prop-jac-op}, there exists a solution $v\in \mathcal{D}_f^{k+2,\alpha}(\R^n,\R^m)$ to $$\Delta_fv=-A(u)(\nabla u,\nabla u)-\Delta_f(\phi u_0),\quad \|v\|_{\mathcal{D}_f^{2,\alpha}(\R^n,\R^m)}\leq C(n,\Lambda).$$ Consequently, by using the maximum principle, $u-\phi u_0=v$ and in particular, this implies that $u-\phi u_0\in \mathcal{D}_f^{2,\alpha}(\R^n,\R^m)$ with the corresponding estimate: $\|u-\phi u_0\|_{\mathcal{D}_f^{2,\alpha}(\R^n,\R^m)}\leq C(n,\Lambda)$. This ends the proof of Corollary \ref{coro-non-neg-tar-comp} by invoking Arzela-Ascoli's Theorem.

\end{proof}

Corollary \ref{coro-non-neg-tar-comp} allows to prove the following existence and uniqueness theorem:
\begin{theo}\label{exi-uni-non-neg-cur}
Let $(N,g)$ be a Riemannian manifold with non-positive sectional curvature.
Let $n\geq 3$ and let $u_0\in C^{\infty}(\mathbb{S}^{n-1},N)$. Then there exists a unique smooth solution coming out of $u_0$ that is regular at infinity: this solution must be expanding.
\end{theo}
\begin{rk}\label{rk-theo-rf-vs-hmf}
Theorem \ref{exi-uni-non-neg-cur} is analogous to a theorem proved by the author for expanding gradient Ricci solitons coming out of metric cones with simply connected smooth sections endowed with a positively curved metric: \cite{Der-Smo-Pos-Met-Con}. However, Theorem \ref{exi-uni-non-neg-cur} proves furthermore that the only solutions that come out of such initial data must be expanding.
\end{rk}

\begin{proof}
The proof is based on a continuity method: since $N$ is non-positively curved, it is aspherical by Hadamard's Theorem. In particular, there exists a smooth homotopy of maps $(u_0^{\sigma})_{\sigma\in[0,1]}:\Sp^{n-1}\rightarrow N$ from $u_0$ to a constant map $P_1\in N$. Consider the set of solutions 
\begin{eqnarray*}
S:=\left\{\sigma\in[0,1]\quad|\quad \text{there exists a solution $u^{\sigma}\in\mathcal{E}_{xp}^{\infty}(N)$ coming out of $u_0^{\sigma}$}\right\}.
\end{eqnarray*}
We prove that $S$ is a closed, open and non empty set proving thereby $S=[0,1]$ by connectedness of $[0,1]$.

$S$ is not empty since the constant map $P_1\in N$ is a smooth expanding solution coming out of itself, i.e. $1\in S$.

$S$ is closed by Corollary \ref{coro-non-neg-tar-comp}: indeed, since $(u_0^{\sigma})_{\sigma\in[0,1]}$ is a smooth homotopy path, there exists a sequence of nonnegative numbers $(\Lambda_k)_{k\geq 1}$ such that $u^{\sigma}\in  \mathcal{E}_{xp}^{\infty}(N,(\Lambda_k)_{k\geq 1})$ as soon as $\sigma\in S$.

Let us show that $S$ is open. Let $\sigma_0\in S$ and $u^{\sigma_0}$ be an expanding solution coming out of $u_0^{\sigma_0}$ that is regular at infinity. Let $k\geq 2$ and $\alpha\in(0,1)$. Then Theorem \ref{theo-fred-prop-jac-op} ensures that the corresponding Jacobi operator $L_{u^{\sigma_0}}$ is a Fredholm operator of degree $0$. Moreover, since $N$ is non-positively curved, $L_{u^{\sigma_0}}$ is injective: indeed, by [Proposition $1.6.2$, \cite{Lin-Wang-Boo}], 
\begin{eqnarray*}
\left<L_{u^{\sigma_0}}\kappa,\kappa\right>=\left<-\Delta_f\kappa,\kappa\right>-\tr \left(\Rm(g_N)(\kappa,\nabla u^{\sigma_0},\nabla u^{\sigma_0},\kappa)\right),\quad \kappa\in C^{\infty}_0(\R^n,T_{u^{\sigma_0}}N),
\end{eqnarray*}
where $\Rm(g_N)(e_1,e_2,e_2,e_1)$ denotes the sectional curvature of the metric $g_N$ evaluated on the two-plane spanned by $(e_1,e_2)$. Therefore, if $\kappa\in \ker_0L_{u^{\sigma_0}}$ then $\|\nabla \kappa\|_{L^2_f}=0$, i.e. $\kappa$ is parallel which implies in particular that its norm is constant on $\R^n$. This fact shows that $\kappa\equiv 0$ since it converges to $0$ at infinity.

To sum it up, $L_{u^{\sigma_0}}$ is an isomorphism: by the implicit function theorem, there exists a neighborhood $\mathcal{U}^{k,\alpha}(u_0^{\sigma_0})$ of $u_0^{\sigma_0}$ in $C^{k,\alpha}(\mathbb{S}^{n-1},N)$ such that any map $u_0^{\sigma}\in \mathcal{U}^{k,\alpha}(u_0^{\sigma_0})$ admits an expanding solution $u^{\sigma}\in C^{k,\alpha}_{con}(\R^n,N)$.

Now, we want to ensure that any expanding solution $u^{\sigma}\in C_{con}^{k,\alpha}(\R^n,N)$ coming out of a smooth map $u_0^{\sigma}$ is actually in $C^{\infty}_{con}(\R^n,N):=\cap_{k\geq 0}C_{con}^{k,\alpha}(\R^n,N)$. We proceed similarly to the proof of Corollary \ref{coro-non-neg-tar-comp}. Using the same notations, we observe that $u^{\sigma}-\phi u^{\sigma}_0$ satisfies (\ref{equ-sol-substract-init-cond}). Since the right-hand side is in $C^{k+1,\alpha}_f(\R^n,\R^m)$, there exists a solution $v\in \mathcal{D}_f^{k+3,\alpha}(\R^n,\R^m)$ such that $\Delta_fv=\Delta_f(u^{\sigma}-\phi u^{\sigma}_0)$. Again, by the maximum principle, $u^{\sigma}-\phi u_0^{\sigma}\in C^{k+1,\alpha}_f(\R^n,\R^m)$. In particular, this shows that $u^{\sigma} \in C^{k+1,\alpha}_{con}(\R^n,\R^m)$. One ends the proof by induction on $k$.\\

To prove the uniqueness statement, we proceed in two steps.\\

First of all, we prove the uniqueness statement among expanding solutions: let $u_1$ and $u_2$ be two solutions in $\mathcal{E}_{xp}^{\infty}(N)$ coming out of the same $0$-homogeneous map $u_0$. Consider as before a smooth homotopy path $(u_0^{\sigma})_{\sigma\in[0,1]}$ connecting $u_0$ to a constant map $P_1\in N$. By the existence part, there exists two continuous paths of expanding solutions $(u_1^{\sigma})_{\sigma\in[0,1]}$ and $(u_2^{\sigma})_{\sigma\in[0,1]}$ in $\mathcal{E}_{xp}^{\infty}(N)$. Define the following set where these two paths coincides: 
\begin{eqnarray*}
\bar{S}:=\left\{\sigma\in[0,1]\quad|\quad u_1^{\sigma}=u_2^{\sigma}\right\}.
\end{eqnarray*}
By Corollary \ref{coro-unique-exp-pt}, $\sigma=1\in \bar{S}$. By its very definition, $\bar{S}$ is closed. Now, if $\overline{\sigma}\in\bar{S}$, $u_1^{\sigma}-u_2^{\sigma}$ is of arbitrary small norm in $\mathcal{E}_{xp}^{k,\alpha}(N)$ for some $k\geq 2$ and some $\alpha\in(0,1)$ when $\sigma$ is close to $\overline{\sigma}$ by Proposition \ref{prop-dec-time-diff-sol} together with the fact that the dependence on $\sigma $ is continuous. Since the Jacobi operator is an isomorphism as explained above, the implicit function theorem shows that $u_1^{\sigma}=u_2^{\sigma}$, i.e. $\sigma\in \bar{S}$ when $\sigma$ is close to $\bar{\sigma}$.

Finally, let $u$ be a solution of the Harmonic map flow that comes out of $u_0$. Assume it is regular at infinity. Let $u_b$ be an expanding solution coming out of $u_0$ that is regular at infinity. Then by Theorem \ref{ilmanen-smooth-conj}, there exist two smooth expanding solutions $u_1$ and $u_2$ that are regular at infinity and which come out of $u_0$. By the previous uniqueness result for expanding solutions, these two solutions coincide (with $u_b$). Recall that $u_1$ and $u_2$ were obtained by blowing up and blowing down the solution $u$. Therefore, by the monotonicity of the relative entropy, the relative entropy is constant in time and Theorem \ref{ilmanen-smooth-conj} implies that $u$ is an expanding solution. Now, the previous uniqueness result implies that $u$ and $u_b$ coincide. 

\end{proof}

\bibliographystyle{alpha.bst}
\bibliography{bib-Entropy-Hmf}

\def\cprime{$'$}
\begin{thebibliography}{{Der}17a}

\bibitem[Aro57]{Aron-Uni-Con}
N.~Aronszajn.
\newblock A unique continuation theorem for solutions of elliptic partial
  differential equations or inequalities of second order.
\newblock {\em J. Math. Pures Appl. (9)}, 36:235--249, 1957.

\bibitem[BB11]{Bie-Biz}
Pawe\l Biernat and Piotr Biz\'on.
\newblock Shrinkers, expanders, and the unique continuation beyond generic
  blowup in the heat flow for harmonic maps between spheres.
\newblock {\em Nonlinearity}, 24(8):2211--2228, 2011.

\bibitem[BW17]{Ber-Wan-MCF}
J.~{Bernstein} and L.~{Wang}.
\newblock {The space of asymptotically conical self-expanders of mean curvature
  flow}.
\newblock {\em ArXiv e-prints}, December 2017.

\bibitem[CLN06]{Cho-Boo}
Bennett Chow, Peng Lu, and Lei Ni.
\newblock {\em Hamilton's {R}icci flow}, volume~77 of {\em Graduate Studies in
  Mathematics}.
\newblock American Mathematical Society, Providence, RI, 2006.

\bibitem[Der16]{Der-Smo-Pos-Met-Con}
Alix Deruelle.
\newblock Smoothing out positively curved metric cones by {R}icci expanders.
\newblock {\em Geom. Funct. Anal.}, 26(1):188--249, 2016.

\bibitem[{Der}17a]{Der-Asy-Com-Egs}
A.~{Deruelle}.
\newblock Asymptotic estimates and compactness of expanding gradient ricci
  solitons.
\newblock {\em Annali della Scuola Normale Superiore di Pisa}, Vol.
  XVII:485--530, 2017.

\bibitem[Der17b]{Uni-Con-Egs-Der}
Alix Deruelle.
\newblock Unique continuation at infinity for conical {R}icci expanders.
\newblock {\em Int. Math. Res. Not. IMRN}, (10):3107--3147, 2017.

\bibitem[DL18]{Der-Lam-HMF}
A.~{Deruelle} and T.~{Lamm}.
\newblock {Existence of expanders of the harmonic map flow}.
\newblock {\em ArXiv e-prints}, January 2018.

\bibitem[GGM16]{Ger-Gho-Miu}
P.~{Germain}, T.-E. {Ghoul}, and H.~{Miura}.
\newblock {On uniqueness for the Harmonic Map Heat Flow in supercritical
  dimensions}.
\newblock {\em ArXiv e-prints}, January 2016.

\bibitem[GR11]{Ger-Rup}
Pierre Germain and Melanie Rupflin.
\newblock Selfsimilar expanders of the harmonic map flow.
\newblock {\em Ann. Inst. H. Poincar\'e Anal. Non Lin\'eaire}, 28(5):743--773,
  2011.

\bibitem[Ham75]{Ham-HMF-Bdy}
Richard~S. Hamilton.
\newblock {\em Harmonic maps of manifolds with boundary}.
\newblock Lecture Notes in Mathematics, Vol. 471. Springer-Verlag, Berlin-New
  York, 1975.

\bibitem[HM92]{Har-Mou}
Robert Hardt and Libin Mou.
\newblock Harmonic maps with fixed singular sets.
\newblock {\em J. Geom. Anal.}, 2(5):445--488, 1992.

\bibitem[Ilm]{Ilm-Lec-Not}
Tom Ilmanen.
\newblock {{L}ectures on {M}ean {C}urvature {F}low and {R}elated {E}quations}.
\newblock {preliminary version, available under \url{
  http://www.math.ethz.ch/~ilmanen/papers/notes.ps}}.

\bibitem[INS14]{Ilm-Nev-Sch}
T.~{Ilmanen}, A.~{Neves}, and F.~{Schulze}.
\newblock {On short time existence for the planar network flow}.
\newblock {\em ArXiv e-prints}, July 2014.

\bibitem[LW08]{Lin-Wang-Boo}
Fanghua Lin and Changyou Wang.
\newblock {\em The analysis of harmonic maps and their heat flows}.
\newblock World Scientific Publishing Co. Pte. Ltd., Hackensack, NJ, 2008.

\bibitem[Shi89]{Shi-Def}
Wan-Xiong Shi.
\newblock Deforming the metric on complete {R}iemannian manifolds.
\newblock {\em J. Differential Geom.}, 30(1):223--301, 1989.

\bibitem[Sma65]{Sard-Smale}
S.~Smale.
\newblock An infinite dimensional version of {S}ard's theorem.
\newblock {\em Amer. J. Math.}, 87:861--866, 1965.

\bibitem[Str88]{Str-Har-Map}
Michael Struwe.
\newblock On the evolution of harmonic maps in higher dimensions.
\newblock {\em J. Differential Geom.}, 28(3):485--502, 1988.

\bibitem[Tay96]{Tay-Boo-II}
Michael~E. Taylor.
\newblock {\em Partial differential equations. {II}}, volume 116 of {\em
  Applied Mathematical Sciences}.
\newblock Springer-Verlag, New York, 1996.
\newblock Qualitative studies of linear equations.

\bibitem[Whi87]{Whi-para-ell-fct}
Brian White.
\newblock The space of {$m$}-dimensional surfaces that are stationary for a
  parametric elliptic functional.
\newblock {\em Indiana Univ. Math. J.}, 36(3):567--602, 1987.

\bibitem[Whi91]{White-var-met}
Brian White.
\newblock The space of minimal submanifolds for varying {R}iemannian metrics.
\newblock {\em Indiana Univ. Math. J.}, 40(1):161--200, 1991.

\end{thebibliography}

\end{document}